\documentclass[letterpaper, reqno,final]{amsart}

\usepackage{amsmath, amssymb, amsthm}
\usepackage{dsfont}
\usepackage[hidelinks]{hyperref}
\usepackage{cleveref}
\usepackage{cite}
\usepackage{enumerate}
\usepackage{bookmark}
\usepackage{tikz-cd} 
\usepackage{subcaption}
\usepackage{mathtools}

\calclayout

\usepackage[right]{showlabels}

\DeclareMathOperator{\supp}{supp}
\DeclareMathOperator{\Vol}{Vol}
\DeclareMathOperator{\sign}{sgn}

\DeclareMathOperator{\sgn}{sgn}

\DeclareMathOperator{\Hess}{Hess}

\newcommand{\jBra}[1]{\langle #1 \rangle}

\newcommand{\bbZ}{\mathbb{Z}}
\newcommand{\bbR}{\mathbb{R}}
\newcommand{\bbT}{\mathbb{T}}
\newcommand{\bbC}{\mathbb{C}}
\newcommand{\bbOne}{\mathds{1}}

\newcommand{\cS}{\mathcal{S}}
\newcommand{\cT}{\mathcal{T}}
\newcommand{\cR}{\mathcal{R}}
\newcommand{\cF}{\mathcal{F}}

\newcommand{\fs}{\mathfrak{s}}
\newcommand{\fhPi}{\frac{\pi}{2}}
\newcommand{\fC}{\mathfrak{C}}
\newcommand{\fL}{\mathfrak{L}}

\newcommand{\rmI}{\mathrm{I}}
\newcommand{\rmII}{\mathrm{II}}
\newcommand{\rmIII}{\mathrm{III}}
\newcommand{\rmIV}{\mathrm{IV}}

\newcommand{\SSCutoff}[1][\xi]{\rho(t^\frac{1}{5} #1)}
\newcommand{\mSS}{\mathit{SS}}
\newcommand{\discrAvg}{\Psi}
\newcommand{\linProp}[1][t]{e^{i#1\discrAvg}}
\newcommand{\invLinProp}[1][t]{e^{-i#1\discrAvg}}
\newcommand{\deles}{\nabla_{\eta,\sigma}}
\newcommand{\bigO}[1]{O\left({#1}\right)}
\newcommand{\tmkdv}{\text{mKdV}}
\newcommand{\phimKdV}{\phi_{\tmkdv}}
\newcommand{\tcts}{\text{cts}}

\newcommand{\fProjc}[1]{{Q^{\tcts}_{#1}}}
\newcommand{\Rncp}[1]{R_n^{c,{#1}}}

\newcommand{\nablaes}{\nabla_{\eta,\sigma}}

\newcommand{\stPhErr}{{\frac{1}{52}}}

\newcommand{\btext}[1]{\left\{\textup{#1}\right\}}

\newcommand{\bbetter}{\btext{better}}
\newcommand{\bsim}{\btext{similar terms}}

\newcommand{\bsime}{\btext{similar or easier terms}}
\newtheorem{thm}{Theorem}

\newtheorem{lemma}[thm]{Lemma}
\newtheorem{cor}[thm]{Corollary}
\newtheorem{rmk}{Remark}

\crefname{equation}{}{}

\title{Asymptotics for small data solutions of the Ablowitz-Ladik equation}
\author{Gavin Stewart}
\address{Department of Mathematics, Rutgers University, 110 Frelinghuysen Rd., Piscataway, NJ, 08854, USA}
\email{\href{mailto:gavin.stewart@rutgers.edu}{gavin.stewart@rutgers.edu}}

\begin{document}
\maketitle
\begin{abstract}
    We study the asymptotics for the Ablowitz-Ladik equation.  By taking appropriate continuum limits, it can be shown that the behavior of the equation near degenerate frequencies is well approximated by a complex modified Korteweg-de Vries equation.  Using this connection, we use the method of space-time resonances to derive a description of the modified scattering behavior of the Ablowitz-Ladik equation, which includes two regions where the solution behaves like a self-similar solution to the complex mKdV equation.
\end{abstract}

\section{Introduction\label{sec:intro}}

\subsection{Background}

We will study the small data asymptotics for Ablowitz-Ladik equation
\begin{equation}\label{eqn:AL-eqn}
    \left\{\begin{aligned}
    i\partial_t v_n + v_{n+1} - 2 v_n + v_{n-1} &= (v_{n+1} + v_{n-1}) |v_n|^2\\ 
    v_n(0) &= v^{*}_n
\end{aligned}\right.
\end{equation}
where $v_n(t) : \bbZ \times \bbR \to \mathbb{C}$.  This equation was first considered by Ablowitz and Ladik in~\cite{ablowitzNonlinearDifferentialdifferenceEquations1976} as a discretization of the cubic Schr\"odinger equation which retains the complete integrability of the continuum equation.  With respect to the non-standard Poisson brackets
\begin{equation*}	\begin{split} 
    \{v_m,\overline{v_n}\} = i \left(1 - \overline{v_n}v_m\right) \delta_{n,m}\\
        \{v_n,v_m\} = \{\overline{v_n}, \overline{v_m}\} = 0
\end{split}\end{equation*}
the equation has a Hamiltonian structure, with Hamiltonian
\begin{equation*}
    H(v_n, \overline{v_n}) =  \sum \overline{v_n}\left( v_{n+1} - v_{n-1}\right) - 2 \sum \log \left(1 - |v_n|^2 \right)
\end{equation*}

As a completely integrable equation, Ablowitz-Ladik has infinitely many conserved quantities.  A particularly useful one is $c_{-\infty}$, which plays a role similar to the mass:
\begin{equation}
    c_{-\infty} = \prod (1 - |v_n|^2)
\end{equation}
(see~\cite{ablowitzDiscreteContinuousNonlinear2004}).  

In principle, the integrability of~\eqref{eqn:AL-eqn} means the inverse scattering transform provides an explicit solution formula~\cite{ablowitzDiscreteContinuousNonlinear2004}.  Indeed, certain modified scattering results have been found, see~\cite{yamaneLongtimeAsymptoticsIntegrable2019,yamaneLongtimeAsymptoticsDefocusing2014,yamaneLongTimeAsymptoticsDefocusing2015,chenLongTimeAsymptoticBehavior2020,meisenLongtimeAsymptoticAnalysis2022a}.	Historically, the complexity of the calculations required means that the modified scattering can often only be shown if the initial data is very strongly localized.  There has been recent progress in lowering the regularity requirements~\cite{chenL2SobolevSpaceBijectivity2023,meisenLongtimeAsymptoticAnalysis2022a}.  However, to the best of the author's knowledge, the self-similar asymptotics for $|n| \approx 2t$ have not yet been proved for low regularity data.  Thus, it is desirable to find other methods to treat the long-time behavior of solutions.

The slow dispersive decay due to the discreteness of the problem makes it more challenging to prove asymptotic results.  Like the complex mKdV equation, the Ablowitz-Ladik equation has a sharp $\ell^1 \to \ell^\infty$ decay rate of $t^{-1/3}$\cite{stefanovAsymptoticBehaviourSmall2005}.  This means that nonlinear effects will tend to be larger in the discrete systems than in the continuous case.  The weaker linear decay may seem surprising, especially since solution to~\eqref{eqn:AL-eqn} converge solutions of the cubic Schr\"odinger equation under the appropriate continuum limit~\cite{killipContinuumLimitAblowitz2022a}.  The apparent contradiction can be reconciled by noting that taking this continuum limit corresponds to focusing on frequencies near $0$ (which \emph{do} decay like $t^{-1/2}$ for Ablowitz-Ladik).  The degenerate decay for Ablowitz-Ladik occurs at frequencies near $\pm \frac{\pi}{2}$, so it is important to understand the limiting behavior near these frequencies.  To obtain a continuum limit near frequency $\fhPi$, we let $u_n(t) = (i)^n e^{-2it} h W(h(n-2t), h^3t)$.  Then, $W$ satisfies
\begin{equation}\label{eqn:W-mKdV-eqn}
    \partial_t W + \frac{1}{3}\partial_x^3 W = 2 |W|^2 \partial_x W + O(h)
\end{equation}
To leading order,~\eqref{eqn:W-mKdV-eqn} is invariant under rescaling $h$, so if we (formally) send $h \to 0$, we expect that $u_n$ should behave like a solution to the complex mKdV equation~\eqref{eqn:W-mKdV-eqn} translated to the right by $2t$ and modulated by $i^ne^{-2it}$.  If we instead take $u_n(t) = (-i)^n e^{-2it} h W(h(n+2t), -h^3t)$, a similar manipulation shows that
$ u_n(t) \approx i^{n}e^{-2it} \tilde{W}(n + 2t,-t)$
for a different solution $\tilde{W}$ to~\eqref{eqn:W-mKdV-eqn}.  This relation to complex mKdV will allow us to use arguments similar to those from~\cite{stewartLongTimeDecay2021} to obtain asymptotics for~\eqref{eqn:AL-eqn}.
			
\subsection{Main result}			
		
Before stating our results, it will be helpful to re-write equation~\eqref{eqn:AL-eqn} in a form which will simplify both the statement of the results and some later arguments. We first apply a simple change of variables $u_n(t) = e^{2it} v_n(t)$, which eliminates factors of $e^{-2it}$ which would otherwise make our formulas more cumbersome.  We also choose to pose the equation with initial data given at time $t = 1$, since this simplifies certain statements about decay rates without greatly changing the problem (it is not difficult to evolve from $t = 0$ to $t = 1$).  Thus, we are left to consider the equation
\begin{equation}\label{eqn:AL-eqn-rotated}
	\left\{\begin{aligned}
	    i\partial_t u_n + u_{n+1} + u_{n-1} &= (u_{n+1} + u_{n-1}) |u_n|^2\\
	    u_n(1) &= e^{-i\Psi} u^*_n
    \end{aligned}\right.
\end{equation}
where
\begin{equation*}
    \Psi u_n = u_{n+1} + u_{n-1}
\end{equation*}
Let us reiterate that the differences between~\eqref{eqn:AL-eqn-rotated} and~\eqref{eqn:AL-eqn} are purely cosmetic, and that any asymptotic result concerning~\eqref{eqn:AL-eqn-rotated} transfers naturally to an asymptotic result for~\eqref{eqn:AL-eqn}.  With this in mind, we now state the main result of this chapter:
\begin{thm}\label{thm:main-thm-AL}
	There exists an $\epsilon_0$ such that for all $\epsilon < \epsilon_0$, if $\lVert \jBra{n} u^*_n \rVert_{\ell^2} \leq \epsilon$, then the solution $u_n$ to~\eqref{eqn:AL-eqn-rotated} exists for all $t \geq 1$.  Moreover, the solution has the following asymptotics:
	
	If $|n| - 2|t| \gtrsim t^{1/3}$, then $u_n$ decays rapidly as $|n| \to \infty$:
	\begin{equation}\label{eqn:main-thm-al-decay-rgn}
	    |u_n(t)| \lesssim t^{-1/3} \left(\frac{t^{1/3}}{|n| - 2|t|}\right)^{3/4}
	\end{equation}
		
	If $2|t| - |n| \gtrsim t^{1/3}$, then we have modified scattering:  that is, there exists a function $f_\infty(\xi)$ independent of $t$ such that
	\begin{equation}\label{eqn:main-thm-al-osc-rgn}\begin{split}
	    u_n(t) =& \sum_{r=1}^2 \frac{1}{\sqrt[4]{4t^2 - n^2}} \exp\left(-i(-1)^r\left( \sqrt{4t^2 - n^2} + \frac{\pi}{4} + \frac{i}{4} \int_1^t \frac{|\hat{u}(\xi_r,s)|^2}{s}\;ds\right)\right) f_\infty(\xi_r)\\ &+ O\left(\epsilon t^{-1/3} |t^{1/3}\cos \xi_r|^{-\frac{5}{9}}\right)
	\end{split}\end{equation}
	where $\xi_r$ is defined by $\sin \xi_r = \frac{n}{2t}$, $\sgn \cos \xi_r = (-1)^r$.
	
	If $|n \mp 2t| \lesssim t^{1/3}$, then we have self-similar behavior
	\begin{equation}\label{eqn:main-thm-al-ss-rgn}
	    u_n(t) = S^\pm(n \pm 2t, \pm t) + O(\epsilon t^{-2/5})
	\end{equation}
	where $S^\pm$ are self-similar solutions to the complex mKdV equation
	\begin{equation}\label{eqn:cmkdv}
	    \partial_t S^\pm + \frac{1}{3} \partial_x^3 S^\pm = |S^\pm|^2 \partial_x S^\pm
	\end{equation}
\end{thm}

\subsection{Idea of the proof}

\subsubsection{Linear estimates}
We begin by considering estimates for the linear problem.  Suppose $u_n^L$ solves the linear equation
\begin{equation*}
    \left\{\begin{array}{l}
        i \partial_t u_n^L + u_{n+1}^L + u_{n-1}^L = 0\\
        u_n^L(0) = f_n
    \end{array}\right.
\end{equation*}
Then, taking the Fourier transform, we are lead to the solution formula
\begin{equation*}
    u_n^L(t) = e^{it\Psi} f_n = \frac{1}{\sqrt{2\pi}} \int_{\bbT} e^{-in \xi -2t \cos \xi} \hat{f}(\xi)\;d\xi
\end{equation*}
In particular, the second derivative of the phase vanishes at the degenerate frequencies $\xi =\pm \fhPi$, leading to a decay rate of $t^{-1/3}$ by the van der Korput lemma.  On the other hand, we also have the formula
\begin{equation*}
    u_{n+1}^L(t) + u_{n-1}^L(t) = \frac{2}{\sqrt{2\pi}} \int_{\bbT} \cos \xi e^{-in \xi -2t \cos \xi} \hat{f}(\xi)\;d\xi
\end{equation*}
with the vanishing of the $\cos \xi$ multiplier compensating for the degeneracy of the phase.  We will see in~\Cref{sec:lin-est} that this leads to the improved bilinear estimate
\begin{equation*}
    |(u_{n+1}^L + u_{n-1}^L) u_n^L| \lesssim t^{-1}
\end{equation*}
In particular, if we assume that $u_n$ satisfies the same sorts of estimates as $u_n^L$, then the nonlinear term in~\eqref{eqn:AL-eqn-rotated} decays like $t^{-1} |u_n|$, which is compatible with the long range corrections we see in~\Cref{thm:main-thm-AL}.  Let us define
$$f_n = \invLinProp u_n$$
If we want $u_n$ to decay like a linear solution, we need $f_n$ to be well-behaved.  We will see that it is enough to show that $\lVert f \rVert_X \lesssim \epsilon$, where $X$ is the space defined by the norm
\begin{equation}\label{eqn:X-norm-def}
    \lVert f_n \rVert_X = \sup_{t > 1} \left(\lVert \hat{f}(\xi,t) \rVert_{L^\infty_{\xi}} + t^{-1/6}\lVert n f_n \rVert_{\ell^2_n}\right)
\end{equation}

\subsubsection{Step 1: The stationary phase estimate}

We first consider the problem of proving that $\hat{f}$ is bounded.  Observe that $\hat{f}$ satisfies the equation
\begin{equation}\label{eqn:dt-f-hat}
    \partial_t \hat{f}(\xi,t) = \iint e^{it\phi(\xi,\eta,\sigma)}\cos(\xi - \eta - \sigma) {\hat f(\xi-\eta-\sigma,t)}\hat f(\eta, t) \overline{\hat f(-\sigma, t)}\;d\eta d\sigma
\end{equation}
where the variables $\xi, \eta, \sigma$ belong to the frequency space $\bbT = \bbR / 2\pi\bbZ$, and
\begin{equation}
    \phi(\xi,\eta,\sigma) = -\mathcal{F}\discrAvg(\xi) + \mathcal{F}\discrAvg(\eta) + \mathcal{F}\discrAvg(\xi - \eta - \sigma) - \mathcal{F}\discrAvg(-\sigma) = 8\sin \frac{\xi - \eta}{2} \cos \frac{\xi - \sigma}{2} \sin\frac{\eta + \sigma}{2}
\end{equation}
is the phase function, which describes how nonlinear interactions lead to the production of a frequency $\xi$ through the interactions of the frequencies $\xi-\eta-\sigma, \eta$ and $-\sigma$. The right-hand side of~\eqref{eqn:dt-f-hat} is an oscillatory integral, so we expect the leading order contributions to come from the stationary points of $\phi$ with respect to $\eta$ and $\sigma$.   A quick calculation shows that
\begin{equation}\label{eqn:phi-derivs}\begin{split}
    \partial_\xi \phi &=  4\sin \frac{\eta + \sigma}{2} \cos\left( \xi - \frac{\eta + \sigma}{2}\right)\\
    \partial_\eta \phi &= -4\cos\frac{\xi - \sigma}{2} \sin \left(\frac{\xi - \sigma}{2} - \eta\right)\\
    \partial_\sigma \phi &= -4 \sin \frac{\xi - \eta}{2} \cos\left(\frac{\xi - \eta}{2} - \sigma\right)
\end{split}\end{equation}
Holding $\xi \neq \pm \fhPi$ fixed, we see that in the $\eta$ and $\sigma$ variables $\phi(\xi,\eta,\sigma)$ has six stationary points given by
\begin{equation*}
\begin{array}{ll}
    P_1 = \left(\xi, -\xi\right) & P_2 = \left(\frac{\xi}{3} + \pi, \frac{\xi}{3}\right)\\
    P_3 = \left(\frac{\xi - \pi}{3}, \frac{\xi + 2\pi}{3}\right) & P_4 = \left(\frac{\xi + \pi}{3}, \frac{\xi - 2\pi}{3}\right)\\
    P_5 = \left(\pi-\xi, \xi + \pi\right) & P_6 = \left(\xi, \pi + \xi \right)
\end{array}
\end{equation*}
Thus, we have that
\begin{equation*}
    \partial_t \hat{f}(\xi,t) = \sum_{r = 1}^6 \frac{1}{t \sqrt{\det D^2 \phi(P_j)}} e^{i\frac{\pi}{4}\sign D^2 \phi(P_j) + it \phi(P_j)} \cos(\xi - \eta_j - \sigma_j) \hat{f}(\eta_j) \hat{f}(\xi - \eta_j - \sigma_j) \overline{\hat{f}(-\sigma_j)} + \btext{error}
\end{equation*}
The contributions from $P_5$ and $P_6$ cancel by symmetry.  The phases for the remaining stationary points are
\begin{equation*}\begin{split}
    \phi(P_1) =& 0\\
    \phi(P_2) =& -8 \sin^3 \left( \frac{\xi}{3} - \frac{\pi}{6}\right)\\
    \phi(P_3) =& 8 \cos^3 \frac{\xi}{3}\\
    \phi(P_2) =& -8 \sin^3 \left( \frac{\xi}{3} + \frac{\pi}{6}\right)
\end{split}
\end{equation*}
In particular, $|\phi(P_2)| \sim \left|\xi - \fhPi\right|^3$, $|\phi(P_4)| \sim \left| \xi + \fhPi\right|^3$, and $\phi(P_3)$ is nonvanishing over the entire frequency domain.  Thus, provided $\xi$ is at a distance at least $t^{-1/3}$ away from the degenerate frequencies, we expect that the oscillations in these terms will lead to addition cancellation,  preventing them from affecting the asymptotics.  Thus, only the point $P_1$ is asymptotically relevant, so we can write
\begin{equation*}
    \partial_t \hat{f}(\xi,t) = \frac{i \sgn(\cos \xi)}{4t}|\hat{f}(\xi,t)|^2 \hat{f}(\xi,t) + \btext{better terms}
\end{equation*}
Dropping all lower order terms leaves us with an ODE which can be explicitly integrated.  Doing so, we find that
\begin{equation*}
    \hat{f}(\xi,t) \approx \exp\left(\frac{i}{4} \int_1^t \frac{|\hat{f}(\xi,s)|^2}{s}\;ds\right) f_\infty(\xi)
\end{equation*}
which is compatible with the logarithmic phase shift given in~\Cref{thm:main-thm-AL}.

\subsubsection{Step 2: Behavior near the degenerate frequencies}

The above argument entirely breaks down near the frequencies $\xi = \pm \fhPi$.  There is no longer enough oscillatory cancellation for us to ignore the contribution from the points $P_2$ and $P_4$, and in addition the error term in the stationary phase approximation becomes unacceptably large due to the coalescence of the stationary points.  Thus, the above argument only gives us the bound
\begin{equation*}
    \partial_t \hat{f}(\xi,t) = O(\epsilon^3 t^{-1})
\end{equation*}
which is insufficient to control $\hat{f}$ for $t \gg \exp\left(C \epsilon^{-2}\right)$.  

To control $\hat{f}$ at these near-degenerate frequencies, we must understand and take advantage of the nonlinear behavior there.  First, we note that the problem in the above argument occurs at the degenerate stationary point $\xi = \eta = -\sigma = \pm \fhPi$.  Based on the group velocity relation, this corresponds to the region $n \approx \mp 2t$ in physical space where $u_n$ exhibits self-similar behavior.  Thus, we introduce the decomposition
\begin{equation*}
    u_n(t) = \mSS_n(t) + w_n(t)
\end{equation*}
where $\mSS_n(t)$ is the approximate self-similar solution defined by
\begin{equation}\label{eqn:mSS-def-intro}\begin{split}
    \widehat{\mSS}(\xi, t) =& \rho(t^{1/5} (\xi - \pi/2)) e^{-2i(\xi - \fhPi) t}\hat{\fs}(\xi - \pi/2,t;\hat{f}(\fhPi/2,t))\\
    &+ \rho(t^{1/5} (\xi + \pi/2)) e^{2i(\xi + \fhPi) t}\hat{\fs}(\xi + \pi/2,-t;\hat{f}(-\fhPi/2,t))
\end{split}\end{equation}
where $\fs(x,t;\alpha)$ is a self-similar solution to the complex mKdV equation~\eqref{eqn:cmkdv} with $\hat{\fs}(0,t;\alpha) = \alpha$ and $\rho$ is a smooth bump function supported in a small neighborhood of $0$.  Then, we can write
\begin{equation}\label{eqn:f-hat-degen-decomp}\begin{split}
    \partial_t \hat{f}\left(\pm \fhPi, t\right) =& \sum_{n \in \bbZ} e^{\mp i\fhPi n} \left(|u_n|^2 (u_{n+1} + u_{n-1})\right)\\
    =& \sum_{n \in \bbZ} e^{\mp i\frac{\pi}{2}n} |\mSS_n|^2(\mSS_{n+1} + \mSS_{n-1})\\
    &+ \sum_{n \in \bbZ} e^{\mp i \fhPi n} |u_n|^2 (w_{n+1} + w_{n-1})\\
    &+ \bsim
\end{split}\end{equation}
Now, for the first term, 
\begin{equation*}\begin{split}
    \sum_{n \in \bbZ} e^{\mp i\frac{\pi}{2}n} |\mSS_n|^2(\mSS_{n+1} + \mSS_{n-1}) =& \pm \int |\fs\left(x,t; \pm \hat{f}\left(\pm \fhPi,t\right)\right)|^2 \partial_x \fs\left(x,t;\left(\pm \fhPi,t\right)\right)\;dx + \bigO{t^{-6/5}}\\
    =& 0 + \bigO{t^{-6/5}}
\end{split}\end{equation*}
since $\fs$ is a self-similar solution to complex mKdV.  Thus, we will be able to control $\hat{f}(\pm \fhPi, t)$ provided we can control the second term of~\eqref{eqn:f-hat-degen-decomp}.

Since $\hat{w}$ vanishes at the degenerate frequencies $\xi = \pm \fhPi$, we expect $w_n$ to have better dispersive decay at the degenerate frequencies provided it also satisfies some weighted $\ell^2$ bound.  The final step in our argument is to show that this is indeed the case.

\subsubsection{Step 3: The weighted \texorpdfstring{$\ell^2$}{l2} bound for \texorpdfstring{$w_n$}{w\_n}}

Now, we use a Gr\"onwall argument to prove that $g_n = \invLinProp w_n$ satisfies the weighted $\ell^2$ bound
\begin{equation*}
    \lVert n g_n \rVert_{\ell^2} \lesssim \epsilon t^{\frac{1}{10}}
\end{equation*}
By Plancherel's theorem, the bound for $ng_n$ will follow if we can show $\partial_\xi \hat{g}$ grows at most like $t^{\frac{1}{10}}$ in $L^2(\bbT)$.  Observe that we have the bound
\begin{equation}\begin{split}
    \frac{1}{2}\partial_t \lVert \partial_\xi \hat g(\xi)\rVert_{L^2}^2 =& \Im\left\langle\partial_\xi \hat{g}, \frac{\partial_\xi}{\pi} \iint e^{it\phi(\xi,\eta,\sigma)}\cos(\xi - \eta - \sigma) \hat g(\xi-\eta-\sigma)\hat f(\eta) \overline{\hat f(-\sigma)}\;d\eta d\sigma \right\rangle_{L^2_\xi}\\
    &+ \left\langle\partial_\xi \hat{g},\hat{R}(\xi) \right\rangle_{L^2_\xi} + \bsim
\end{split}\end{equation}
where $\hat{R}(\xi)$ is the error term arising because $\mSS_n$ is not an exact solution to~\eqref{eqn:cmkdv}, and $\bsim$ denotes integral terms like the first one with one $\hat{f}$ factor replaced by $\widehat{\mSS}$.  We will focus on the integral term: the argument for the equation error is in many ways similar, but more involved.  This integral term can be written as
\begin{equation}\label{eqn:intro-duhamel-division}\begin{split}
    &\left\langle \partial_\xi g, \frac{1}{\pi} \iint e^{it\phi} \cos(\xi - \eta - \sigma) \partial_\xi \hat g(\xi-\eta-\sigma)\hat f(\eta) \overline{\hat f(-\sigma)}\;d\eta d\sigma \right\rangle\\
    \qquad&-t\left\langle \partial_\xi g, \iint e^{it\phi} \partial_\xi \phi \cos(\xi - \eta - \sigma) \hat g(\xi-\eta-\sigma)\hat f(\eta) \overline{\hat f(-\sigma)}\;d\eta d\sigma ds\right\rangle\\
    \qquad &+ \btext{better terms}
\end{split}\end{equation}
The first term initially appears problematic, since it can be written in physical space as
\begin{equation*}
    \langle Lw_n, (\Psi L w_n)|u_n|^2 \rangle_{\ell^2}
\end{equation*}
where
\begin{equation}\label{eqn:L-def}
    L := e^{it\Psi} n \invLinProp
\end{equation}
is the pushforward of the weight $n$ under $e^{it\Psi}$.  Sincethe $\Psi$ multiplier falls on a weighted term,  we cannot use the bilinear estimate directly to get extra decay.  However, by writing the inner product in the Fourier domain and using trigonometric identities it is possible, in essence, to move the portion of the $\cos$ multiplier that is responsible for the bilinear decay from $nw_n$ onto the $|u_n|^2$ term (see~\Cref{cor:bilin-sin-est}), recovering the $t^{-1}$ decay and letting us control this term.  Although this identity might appear miraculous, it should also be expected, since a similar identity holds for complex mKdV using integration by parts (see~\cite{stewartLongTimeDecay2021}).

To control the second term in~\eqref{eqn:intro-duhamel-division}, we use the method of space-time resonances.  This method was first developed developed by Germain, Masmoudi, and Shatah in~\cite{germainGlobalSolutions3D2008,germainGlobalSolution2D2012} and independently by Gustafson, Nakanishi, and Tsai in~\cite{gustafsonScatteringTheoryGross2009}, and has since been applied to a range of problems involving asymptotics for equations with a dispersive character~\cite{katoNewProofLongrange2011,ifrimGlobalBoundsCubic2015,germainGlobalExistenceEulerMaxwell2014,germainGlobalExistenceCoupled2011,ionescuGlobalSolutionsQuasilinear2014,ionescuEulerPoissonSystem2013,ionescuNonlinearFractionalSchrodinger2014,elgindiSharpDecayEstimates2015, ionescuGlobalAnalysisModel2016, ionescuGlobalRegularity2d2018}.  A related method, the method of testing with wave packets developed by Ifrim and Tataru, has also seen wide use~\cite{ifrimGlobalBoundsCubic2015,ifrimTwoDimensionalWater2016,harrop-griffithsLongTimeBehavior2016}.

The space-time resonance method can be seen as a generalization from the normal form method for PDEs introduced by Shatah~\cite{shatahNormalFormsQuadratic1985} and Klainerman's method of vector fields~\cite{klainermanUniformDecayEstimates1985}.  Since the second integral term in~\eqref{eqn:intro-duhamel-division} can be seen as an oscillatory integral in $s$, $\eta$, and $\sigma$, it is also natural to assume that the leading order contribution to $f_n$ occur for interactions where $\nabla_{s,\eta,\sigma}(s\phi) = 0$.  For this to happen, we must have
\begin{enumerate}
    \item $\phi(\xi,\eta,\sigma) = 0$, which corresponds to a resonances in the classical ODE sense of the word.
    \item $\nabla_{\eta,\sigma}\phi(\xi,\eta,\sigma) = 0$, which occurs when to the group velocities of the input frequencies $\xi - \eta - \sigma$, $\eta$, and $\sigma$ are all equal.
\end{enumerate}
Thus, the leading order contribution to the long time dynamics should come from interactions where $\phi = 0$ and $\nabla_{\eta,\sigma} \phi = 0$.  Not all of these interaction are asymptotically relevant: if the symbol $m(\xi,\eta,\sigma) = \cos(\xi - \eta - \sigma) \partial_\xi \phi$ co-vanishes with $\phi$ or $\nabla_{\eta,\sigma} \phi$, we can expect to get better decay.  This latter effect was first noted by Klainerman in~\cite{klainermanNullConditionGlobal1986} in the vector field context, where the symbol $m$ having these vanishing properties is called a null form.

Based on the above discussion, we define the set of time resonances $\mathcal{T}$ to be the set of all frequencies where $\phi$ vanishes to a higher order than $m$.  Similarly, we define the set of space resonances $\mathcal{S}$ to be the set of all frequencies where $\nabla_{\eta,\sigma} \phi$ vanishes to a higher order than $m$, and the set of space-time resonances, $\mathcal{R}$, to be the intersection of $\mathcal{T}$ and $\mathcal{S}$.  We calculate these sets to be
\begin{equation}\label{eqn:space-time-resonance-sets}\begin{split}
    \mathcal{T} &= \{\xi = \eta\} \cup \{\xi = \pi + \sigma\}\\
    \mathcal{S} &= \{\xi = 2\eta + \sigma\} \cap\{ \xi = \eta + 2\sigma + \pi\}\\
    \mathcal{R} &= \left\{\xi = \eta = \pi + \sigma = \pm \frac{\pi}{2}\right\}
\end{split}\end{equation}
with all equalities understood mod $2\pi$.  

Now, we introduce a partition of unity and write
\begin{equation*}
    \iint e^{is\phi} \partial_\xi \phi \cos(\xi - \eta - \sigma) \hat g(\xi-\eta-\sigma)\hat f(\eta) \overline{\hat f(-\sigma)}\;d\eta d\sigma = I_\mathcal{S} + I_\mathcal{T} + I_\mathcal{R}
\end{equation*}
where $I_\mathcal{S}$ contains is supported away from $\mathcal{S}$, $I_\mathcal{T}$ is supported away from $\mathcal{T}$, and $I_\mathcal{R}$ is supported within a distance $t^{-1/3}$ of $\mathcal{R}$.  The contribution from $I_\mathcal{R}$ can be controlled directly using the smallness of $\hat{f}$, $\hat{g}$, and $m$ near $\mathcal{R}$.  For $I_\mathcal{S},$ we can integrate by parts in $\eta$ and $\sigma$ to obtain terms of the form
\begin{equation*}\begin{split}
    I_\mathcal{S} = t^{-1}T(u_n, w_n, \overline{u_n}) + t^{-1}T(u_n, \Psi Lw_n, \overline{u_n}) + t^{-1}T(\overline{L\mSS_n}, \Psi w_n, \overline{u_n}) + \bsim
\end{split}\end{equation*}
where $T(\cdot,\cdot,\cdot)$ is a cubic pseudoproduct.  For the second pseudoproduct term, if we choose the support of $I_\mathcal{S}$ to be a small neighborhood of $\mathcal{T}$, then $|\xi - \eta - \sigma \pm \fhPi| \ll \max\{|\eta \pm \fhPi|,|\sigma \pm \fhPi|\}$, which effectively allows us to transfer the $\Psi$ multiplier from $Lw_n$ to one of the $u_n$ factors.  By adapting the argument for the bilinear bound $|u_n \Psi u_n| \lesssim t^{-1}$ to the pseudoproduct framework, we find that
\begin{equation*}
    \lVert T(u_n, \Psi Lw_n, \overline{u_n}) \rVert_{\ell^2} \lesssim \epsilon^2 t^{-1} \lVert ng_n \rVert_{\ell^2}
\end{equation*}
A similar argument applies for the first term, since the fact the $\hat{g}(\pm\fhPi) = 0$ gives us the Hardy-type inequality
\begin{equation*}
    \left\lVert \frac{1}{\Psi} g_n \right\rVert_{\ell^2} \lesssim \lVert n g_n \rVert_{\ell^2}
\end{equation*}
which lets us bound the first pseudoproduct term in effectively the same manner.  The final pseudoproduct term, which comes from writing $Lu_n = Lw_n + L\mSS_n$, can be controlled in a similar manner using decay estimates for the self-similar solution.  Thus, we obtain the bound
\begin{equation*}
    \int_1^t \langle Lw_n, I_\cR + I_\cS \rangle\;ds \lesssim \int_1^t \epsilon^2 s^{-1}\lVert ng_n \rVert_{\ell^2}^2\;ds
\end{equation*}
which is compatible with the slow growth of $\lVert ng_n \rVert_{\ell^2}$.

For $I_\cT$, we integrate by parts in time to obtain
\begin{equation*}\begin{split}
    \int_1^t \Re\langle Lw_n, I_\cT\rangle\;ds =& \int_1^t \Re\langle Lw_n, T(u_n, w_n, \overline{u_n})\rangle \;ds + \int_1^t s\Re\langle Lw_n, T(u_n, e^{it\Psi} \partial_s g_n, \overline{u_n})\rangle\;ds\\
    &+ \int_1^t s\Re\langle L\partial_s w_n, T(u_n, w_n, \overline{u_n})\rangle\;ds + \bsim
\end{split}\end{equation*}
The first term is similar to the pseudoproduct terms considered for $I_\cS$.  For the second term, we observe that
\begin{equation*}
    e^{it\Psi} \partial_s g_n = |u_n|^2 \Psi w_n + R_n + \bsim
\end{equation*}
A careful analysis also shows that
\begin{equation*}
    \lVert T(u_n, |u_n|^2 \Psi w_n, \overline{u_n}) \rVert_{\ell^2} \lesssim s^{-2} \lVert ng_n \rVert_{\ell^2}
\end{equation*}
and a more in-depth analysis shows that we can also control the term containing $R_n$.  Finally, for the term containing $\partial_s Lw_n$, we expand out the derivative using the equation.  The majority of the terms can be shown to decay rapidly.  The only remaining term is
\begin{equation*}
    J = \int_1^t s^2 \Re \langle I_\cT, T(u_n, w_n, \overline{u_n}) + \bsim\rangle\;ds
\end{equation*}
Integrating by parts in time, we get terms
\begin{equation}\begin{split}
    J =& \left.s^2 \Re\langle T(u_n, w_n, \overline{u_n}) + \bsim, T(u_n, w_n, \overline{u_n}) + \bsim \rangle \right|_{s=1}^{s=t}\\
    &- \int_1^t 2s \Re\langle T(u_n, w_n, \overline{u_n}) + \bsim, T(u_n, w_n, \overline{u_n}) + \bsim \rangle\;ds\\
    &- \int_1^t s^2 \Re\langle T(u_n, w_n, \overline{u_n}) + \bsim, \partial_s (T(u_n, w_n, \overline{u_n}) + \bsim )\rangle\;ds
\end{split}\label{eqn:J-ident}\end{equation}
In the last term, we recognize
\begin{equation*}
    \partial_s (T(u_n, w_n, \overline{u_n}) + \bsim ) = I_\cT + T(u_n, \partial_s w_n, \overline{u_n}) + \bsim
\end{equation*}
Notice that the first term gives us a term $-J$ on the right hand side of~\eqref{eqn:J-ident}.  Thus, moving this $-J$ term to the left-hand, we obtain the bound
\begin{equation*}
    J \lesssim \int_1^t \epsilon^2 s^{-1} \lVert ng_n \rVert_{L^2} \;ds + \epsilon^2 t^{1/5}
\end{equation*}
which gives the slow growth of $\lVert n g_n \rVert_{\ell^2}$.

\subsubsection{Step 4: Verifying the asymptotics}

Finally, we show that the previous estimates are enough to obtain the asymptotics from~\Cref{thm:main-thm-AL}. Observe that the bound on $ng_n$ together with the bound
\begin{equation*}
    \lVert n \invLinProp \mSS_n \rVert_{\ell^2} \lesssim \epsilon t^{1/6}
\end{equation*}
imply that 
\begin{equation*}
    \lVert nf_n \rVert_{\ell^2} \lesssim \epsilon t^{1/6}
\end{equation*}
so $\lVert f_n \rVert_X \lesssim \epsilon$.  This is already sufficient to prove that asymptotics in the region $|n| - 2|t| \gg t^{1/3}$, since the linear solution $u_n^L$ also has these asymptotics.  Similarly, the asymptotics in the region $2|t| - |n| \gg t^{1/3}$ follow from the linear estimates once we take into account the logarithmic phase correction for $\hat{f}$.  Finally, in the region $|n| - 2|t| = O(t^{1/3})$, we see that $w_n$ is of lower order than $\mSS_n$, which combined with the integrable decay of $\partial_t \hat{f}(\pm\fhPi,t)$ yields~\eqref{eqn:main-thm-al-ss-rgn}.

\subsection{Organization of the paper}

The plan for the rest of this paper is as follows.  We begin by defining notation and giving some basic properties of pseudoproduct operators in~\Cref{sec:notation-sec}.   Next, in~\Cref{sec:lin-est}, we present some refined estimates for the linear discrete Schr\"odinger equation.  We then give a definition for and some properties of the approximate self-similar solution in~\Cref{sec:approx-ss}. In~\Cref{sec:prof-red}, we show how the main theorem follows from a bootstrap argument involving bounds on $\lVert ng_n\rVert_{\ell^2}$, $\hat{f}$, and $\partial_t \hat{u}(\pm \fhPi,t)$.  Finally, we verify that the bootstrap holds by proving the weighted $\ell^2$ estimate for $g_n$ in~\Cref{sec:n-ell2}, the bounds for the modulation parameters $\partial_t \hat{u}(\pm \fhPi,t)$ in~\Cref{sec:degen-freq-decay} and the logarithmic phase rotation for $\hat{f}$ in~\Cref{sec:L-infty}.

\section{Preliminaries\label{sec:notation-sec}}

\subsection{Conventions and notation}

We consider the Fourier transform $\mathcal{F} : \ell^2(\bbZ) \to L^2(\bbT)$ given by
\begin{equation}
    \mathcal{F}(f_n)(\xi) := \hat{f}(\xi) = \frac{1}{\sqrt{2\pi}} \sum_{n \in \bbZ} f_n e^{in\xi}  
\end{equation}
where $\bbT = \bbR / (2\pi \bbZ)$ is the torus.  The inverse Fourier transform is given by
\begin{equation}
    (\mathcal{F}^{-1} g)_n = \check{g}_n = \frac{1}{\sqrt{2\pi}} \int_{\bbT} g(\xi) e^{-in\xi}\;d\xi
\end{equation}
We now define the Littlewood-Paley multipliers.  Let $\psi_j(x) = \psi(2^{-j}x)$ be a dyadic partition of unity, where $\psi: \bbR \to [0,1]$ is a smooth bump function supported on $\{\frac{1}{2} \leq |x| \leq 2\}$ such that
    \begin{equation*}
        \sum_{j =-\infty}^\infty \psi(2^{-j}x) = 1 \qquad\qquad \text{for all } x \in \bbR \setminus \{0\}
    \end{equation*}  
    Let us also consider a smooth $2\pi$-periodic function $\tilde{\psi}:\bbR \to [0,1]$ which is equal to $1$ for $|x| \leq 1$, vanishes for $|x- \pi| \leq 1$, and satisfies
    \begin{equation*}
        \tilde{\psi}(x) + \tilde{\psi}(x+\pi) = 1 \qquad\qquad \text{for all } x \in \bbR
    \end{equation*}
Using these functions, we define define the functions
\begin{equation}\label{eqn:tilde-psi-def}
    \tilde{\psi}_j(\xi) = \begin{cases}
                    \sum_{n \in \bbZ} \psi_j(\xi - 2\pi n) & j < 0\\
                    \tilde{\psi}(\xi) - \sum_{j < 0} \tilde{\psi}^{\pm}_j(\xi) & j = 0
        \end{cases}
\end{equation}
For a real number $a$, we also define
\begin{equation*}
    \tilde{\psi}_{\leq a} = \sum_{\ell \leq a} \tilde{\psi}_\ell
\end{equation*}
and let
\begin{equation*}
    \tilde{\psi}^{[a]}_j(\xi) = \begin{cases}
        \tilde{\psi}_j(\xi) & 2^{j+1} > a\\
        \tilde{\psi}_{\leq j}(\xi) & 2^j < a \leq 2^{j+1}\\
        0 & \text{else}
    \end{cases}
\end{equation*}

Since the $\tilde{\psi}_j$ are $2\pi$-periodic, so we can naturally consider them as functions on the torus.  Thus, we can define the Fourier projectors
\begin{equation*}
    Q^\pm_j = \begin{cases}
        \tilde{\psi}_j(D \mp \fhPi) & 2^j > t^{-1/3}\\
        \tilde{\psi}_{\leq j}(D \mp \fhPi) & 2^{j} \leq t^{-1/3} < 2^{j+1}\\
        0 & \text{else}
    \end{cases}
\end{equation*}
which act on sequences.  We will also make use of Littlewood-Paley projectors on $\bbR$ when we are discussing the approximately self-similar solutions.  These are given by
\begin{equation*}
    Q_j^\tcts =
    Q^\pm_j = \begin{cases}
        {\psi}_j(D) & 2^j > t^{-1/3}\\
        {\psi}_{\leq j}(D) & 2^{j} \leq t^{-1/3} < 2^{j+1}\\
        0 & \text{else}
    \end{cases}
\end{equation*}
We define the projectors $Q^\pm_{\leq j}$, $Q^\pm_{<j}$, and so on in the natural way.  In working with dispersive estimates, it will be useful to augment these frequency projectors with physical space multipliers.  We begin by defining $\chi^\pm_n(t)$ to be a partition of unity on $\bbZ$ such that $\chi_n^\pm(t) = 1$ for $\mp n < t$ and $\chi_n^\pm(t) = 0$ otherwise.  We also define $\chi^\tcts_k(x;t)$ to be the continuous bump function defined by
\begin{equation*}
    \chi_k^\tcts(x;t) = \begin{cases}
        \psi(x/(t2^{2k})) & 2^k > t^{-1/3}\\
        \sum_{2^\ell \leq t^{-1/3}} \psi(x/(t2^{2\ell})) & 2^{k} \leq t^{-1/3} < 2^{k+1}\\
        0 & 2^{k+1} \leq t^{-1/3}
    \end{cases}
\end{equation*}
Then, we define
\begin{equation*}
    \chi^\pm_{k;n}(t) = \tilde{\chi}^\pm_n(t) \chi^\tcts(n \pm 2t;t)
\end{equation*}
Henceforth, we will drop the index $n$ and the argument $t$ whenever they are clear from the context.  We also define the shorthand
\begin{equation*}\begin{split}
    \chi_{[k \pm N]}^\pm = \sum_{\ell = k -M}^{k+M} \chi_{\ell}^\pm\\
    Q_{[j \pm N]}^\pm = \sum_{\ell = j -M}^{j+M} Q_{\ell}^\pm
\end{split}\end{equation*}
    
\subsection{Multilinear harmonic analysis}

We define the pseudoproduct operator with symbol $M$ to be the operator by
\begin{equation*}
    T_M(f_n,g_n,h_n) = \mathcal{F}^{-1} \iint M(\xi, \eta, \sigma) \hat f(\xi - \eta - \sigma) \hat g(\eta) \hat h(\sigma)\;d\eta d\sigma
\end{equation*}
for $M: \bbT^3 \to \bbC$.  As the following lemma shows, these operators will have the same boundedness properties as those given by the H\"{o}lder inequality, provided $\check{M} \in \ell^1$.

\begin{lemma}\label{thm:multiplier-lemma}
    If $\check{M} \in \ell^1(\bbZ^3)$, then $T_M : \ell^{p_1} \times \ell^{p_2} \times \ell^{p_3} \to \ell^{q}$ provided $\frac{1}{p_1} + \frac{1}{p_2} + \frac{1}{p_3} = \frac{1}{q}$, $1 \leq p_1, p_2, p_3, q \leq \infty$.  Moreover, 
    \begin{equation}
        \lVert T_M(f_n,g_n,h_n) \rVert_{L^q} \leq \lVert \check{M}\rVert_{\ell^1(\bbZ^3)}\lVert f_n \rVert_{\ell^{p_1}}\lVert g_n \rVert_{\ell^{p_2}}\lVert h_n \rVert_{\ell^{p_3}}
    \end{equation}
\end{lemma}
\begin{proof}
    Using Fourier inversion, we find that
    \begin{equation}
        \widehat{T}_M(f_n,g_n,h_n)(\xi) = \iint M(\xi, \eta, \sigma) \sum_{n_1=-\infty}^\infty f_{n_1} e^{-in_1(\xi - \eta - \sigma)}\sum_{n_2=-\infty}^\infty  g_{n_2} e^{-in_2\eta}\sum_{n_3=-\infty}^\infty  h_{n_3}e^{-in_3\sigma}\;d\eta d\sigma
    \end{equation}
    Interchanging the order of summation and integration, we find that
    \begin{equation}
        \widehat{T}_M(f_n,g_n,h_n)(\xi) = \sum_{n_1, n_2, n_3 \in \bbZ} f_{n_1}g_{n_2}h_{n_3} e^{-i\xi n_1} \iint M(\xi, \eta, \sigma)e^{-i\eta(n_2 - n_1)} e^{-i\sigma(n_3 - n_1)} \;d\eta d\sigma
    \end{equation}
    If we now invert the Fourier transform in $\xi$, we get
    \begin{equation}\label{eqn:pseudoprod-phys-def}\begin{split}
        T_M(f_n,g_n,h_n)_m &= \sum_{n_1, n_2, n_3 \in \bbZ} f_{n_1}g_{n_2}h_{n_3}  \iiint e^{-i\xi (n_1 + m)} M(\xi, \eta, \sigma)e^{-i\eta(n_2 - n_1)} e^{i\sigma(n_3 - n_1)} \;d\eta d\sigma d\xi \\
        &= \sum_{n_1, n_2, n_3 \in \bbZ} f_{n_1}g_{n_2}h_{n_3} \check{M}_{n_1 + m, n_2 - n_1, n_3 - n_1}\\
        &= \sum_{k_1, k_2, k_3 \in \bbZ} f_{k_1 - m}g_{k_2 + k_1 - m}h_{k_3 + k_1 - m} \check{M}_{k_1, k_2, k_3}
    \end{split}\end{equation}
    The result now follows by taking the $\ell^q$ norm in $m$ and applying the Minkowski and H\"{o}lder inequalities:
    \begin{align*}
        \left\lVert T_M(f_n,g_n,h_n) \right\rVert_{\ell^q} =& \left\lVert \sum_{k_1, k_2, k_3 \in \bbZ} f_{k_1 - m}g_{k_2 + k_1 - m}h_{k_3 + k_1 - m} \check{M}_{k_1, k_2, k_3} \right\rVert_{\ell^q_m}\\
        \leq& \sum_{k_1, k_2, k_3 \in \bbZ} |\check{M}_{k_1, k_2, k_3}| \left\lVert f_{k_1 - m}g_{k_2 + k_1 - m}h_{k_3 + k_1 - m}  \right\rVert_{\ell^q_m}\\
        \leq& \lVert \check{M} \rVert_{\ell^1(\bbZ^3)} \lVert f \rVert_{\ell^{p_1}} \lVert g \rVert_{\ell^{p_2}}\lVert h \rVert_{\ell^{p_3}}\qedhere
    \end{align*}
\end{proof}
\begin{rmk}\label{rmk:freq-loc-symbol-bds}
    In our analysis, we will often consider symbols $m$ which are compactly supported in a region of frequency space with volume $O(2^{j + k + \ell})$ and satisfy the symbol bounds
    \begin{equation*}
        |\partial_\xi^a \partial_\eta^b \partial_\sigma^c m(\xi,\eta,\sigma)| \lesssim_{a,b,c} A 2^{-aj - bk - c\ell}
    \end{equation*}
    with an implicit constant independent of $m$.  For such symbols,
    \begin{equation}\label{eqn:m-check-decay-bds}
        \left|\check{m}_{n,k,r}\right| \lesssim A \frac{2^j}{(1 + |2^{j}n|)^N}\frac{2^{k}}{(1 + |2^{\ell}m|)^N}\frac{2^{\ell}}{(1 + |2^{\ell}r|)^N}
    \end{equation}
    which shows that $m$ satisfies the hypotheses of~\Cref{thm:multiplier-lemma}.
\end{rmk}

In general, pseudoproducts are not local operators.  However, it can be shown that (at least after localizing in frequency) pseudoproducts with smooth symbols are \emph{pseudolocal}, which roughly speaking means that they are local up to rapidly decaying error terms.
\begin{lemma}\label{thm:pseudolocality-lemma}
    Suppose that $f_n,g_n,$ and $h_n$ are functions such that $\supp f_n$ and $\supp h_n$ separated by a distance $R$.  Let $m_{j,k,\ell}$ be a symbol supported on a region of size $O(2^{j+k+\ell})$ and satisfying the symbol bounds
\begin{equation*}
    |\partial_{\xi}^a \partial_\eta^b \partial_\sigma^c m(\xi,\eta,\sigma)| \lesssim_{a,b,c} 2^{-aj - bk - c\ell}
\end{equation*}
  Then, for $\frac{1}{p_1} + \frac{1}{p_2} + \frac{1}{p_3} = \frac{1}{p}$, we have
\begin{equation}\label{eqn:pseudoloc-1}
    \lVert T_{m_{j,k,\ell}}(f_n,g_n,h_n)  \rVert_{\ell^p} \lesssim_N \jBra{2^{\min(j,k,\ell)} R}^{-N} \lVert f \rVert_{\ell^{p_1}} \lVert g \rVert_{\ell^{p_2}} \lVert h \rVert_{\ell^{p_3}}
\end{equation}
Similarly, if $\supp k_n$ and $\supp f_n$ are separated by a distance $R$, then
\begin{equation}\label{eqn:pseudoloc-2}
    \lVert k_nT_{m_{j,k,\ell}}(f_n,g_n,h_n)  \rVert_{\ell^p} \lesssim_N \jBra{2^{\min(j,k,\ell)} R}^{-N} \lVert f \rVert_{\ell^{p_1}} \lVert g \rVert_{\ell^{p_2}} \lVert h \rVert_{\ell^{p_3}} \lVert k \rVert_{\ell^{p_4}}
\end{equation}
for all $p, p_1, p_2, p_3$ such that $\frac{1}{p_1} + \frac{1}{p_2} + \frac{1}{p_3} + \frac{1}{p_4} = \frac{1}{p}$, and analogous results apply if we interchange the roles of $f_n$, $g_n$, and $h_n$.
\end{lemma}
\begin{proof}
    We will focus on proving~\eqref{eqn:pseudoloc-1}: the argument for~\eqref{eqn:pseudoloc-2} is similar.  Using the support assumptions on $f_n$ and $h_n$ and~\eqref{eqn:pseudoprod-phys-def}, we can write
    \begin{equation*}\begin{split}
        T_M(f_n,g_n,h_n)_m =& \sum_{\substack{k_1, k_2, k_3 \in \bbZ\\ |k_3| \geq R}} f_{k_1 - m}g_{k_2 + k_1 - m}h_{k_3 + k_1 - m} \check{M}_{k_1, k_2, k_3}
    \end{split}\end{equation*}
    Thus, taking an $\ell^p$ norm and using the bounds from~\eqref{eqn:m-check-decay-bds}, we see that
    \begin{equation*}\begin{split}
        \lVert T_M(f_n,g_n,h_n) \rVert_{\ell^p} \lesssim_N&  \sum_{\substack{n_1, n_2\\ |n_3| \geq R}} \lVert f_{n_1 - m}g_{n_2 + n_1 - m}h_{n_3 + n_1 - m}\rVert_{\ell^p_m} 2^{j+k+\ell} \jBra{2^j n_1}^{-10}\jBra{2^k n_2}^{-10}\jBra{2^\ell n_3}^{-10-N}\\
        \lesssim_N& \jBra{2^{\ell}R}^{-N} \lVert f \rVert_{\ell^{p_1}} \lVert g \rVert_{\ell^{p_2}} \lVert h \rVert_{\ell^{p_3}}\qedhere
    \end{split}\end{equation*}
\end{proof}

    A similar pseudolocality estimate also holds the Littlewood-Paley projectors:
    \begin{lemma}\label{lem:LP-pseudoloc-lemma}
        Suppose $f_n$ and $g_n$ are functions whose supports are separated by a distance $R$.  Then,
        \begin{equation*}
            \lVert f P_{\leq j} g \rVert_{\ell^p} + \lVert f P_{j} g \rVert_{\ell^p} + \lVert f P_{> j} g \rVert_{\ell^p} \lesssim_N \jBra{2^j R}^{-N}\lVert f \rVert_{\ell^{p_1}}\lVert f \rVert_{\ell^{p_2}}
        \end{equation*}
        for $\frac{1}{p} = \frac{1}{p_1} + \frac{1}{p_2}$.
    \end{lemma}
	
\section{Linear and multilinear estimates\label{sec:lin-est}}

\subsection{Linear estimates}
We now derive dispersive decay estimates for the Ablowitz-Ladik propagator.  Let $X$ be the Banach space defined by the norm
\begin{equation*}
    \lVert f \rVert_X := \sup_{t} \lVert \hat{f}(\xi,t) \rVert_{L^\infty} + t^{-1/6} \lVert n f_n \rVert_{\ell^2}
\end{equation*}
and define 
\begin{equation*}\begin{split}
    \lVert f \rVert_{X_j^\pm} =& \lVert Q_j^\pm f \rVert_{X}\\
    \lVert f \rVert_{X_{\leq \log t^{-1/3}}^\pm} =& \Bigl\lVert \sum_{2^j \leq t^{-1/3}} Q_{j}^\pm f \Bigr\rVert_{X}\\
\end{split}\end{equation*}    
Then, we have the following linear estimate:
\begin{thm}\label{thm:lin-decay-al}
    Let $f_n \in X$.  Then, for $t > 0$ $u_n(t) = \linProp f_n$ satisfies
    \begin{equation}\begin{split}\label{eqn:refined-linear-est}
        Q_{j}^{\pm} u_n(t) =& \sum_{r = 1}^2\frac{1}{\sqrt[4]{4t^2 - n^2}} e^{-i(-1)^r\sqrt{4t^2 - n^2}+i(-1)^r\frac{\pi}{4}} \hat{P_j^\pm f(\xi_r)} \bbOne_{|n| < 2t - O(t^{1/3})}\\
        &\qquad+ O(t^{-1/3} (t^{1/3} 2^j)^{-4/5})\lVert f \rVert_{X}\chi_{\sim j}^\pm(t)\\
        &\qquad+ O\left(t^{-1/3} \left(t^{1/3}2^{-j/3} (2^j + 2^{k})^{4/3}\right)^{-3/2} \right)\lVert f \rVert_{X_j^\pm}
    \end{split}\end{equation}
    where the $\xi_r$, are solutions\footnote{Strictly speaking, these equations have no solutions in $\bbT$ if $|n| \geq 2|t|$. Since the $\xi_r$ only appear in the formula if $|n| < 2t$, this is not an issue.} to $\sin \xi_r = n/2t$, $\sgn(\cos \xi_r) = (-1)^r$, and $2^{2k} \sim \frac{n}{2t} \mp 1$.   Moreover,
    \begin{equation}\label{eqn:low-freq-est}
        |Q_{\leq \log t^{-1/3}}^{\pm} u_n(t)| \lesssim t^{-1/3}(1 + t^{1/3} 2^{k})^{-2} \lVert f \rVert_{X_{\leq t^{-1/3}}}
    \end{equation}
\end{thm}
\begin{proof}
    We first consider~\eqref{eqn:refined-linear-est}.  Note that we can write
    \begin{equation*}
        Q^{\pm}_j u_n(t) = \frac{1}{\sqrt{2\pi}} \int_\bbT e^{it\phi(\xi)} \hat{f}^{\pm}_j(\xi) \;d\xi
    \end{equation*}
    where $\phi(\xi) = n/t - 2 \cos(\xi)$ and $\hat{f}^{\pm}_j = \mathcal{F}(Q^\pm_j f)$.  Let  $|\frac{n}{2t} \mp 1| \sim 2^{2k}$.  There are three cases depending on the relative sizes of $2^{j}$ and $2^k$.
    
    \paragraph{Case $2^k \ll 2^{j}$} In this case, we have $|\partial_\xi \phi| \sim 2^{2j}$ and $|\partial_\xi^2 \phi| \sim 2^j$, so, integrating by parts,
    \begin{equation*}\begin{split}
        \left|Q^\pm_j u_n(t)\right| &\leq t^{-1} \int_{\bbT} \left| \partial_\xi \frac{1}{\partial_\xi \phi}\right| |\hat f^\pm_j(\xi)|\;d\xi + t^{-1} \int_{\bbT} \left|\frac{1}{\partial_\xi \phi}\right| |\partial_\xi \hat f^\pm_j(\xi)|\;d\xi\\
        \lesssim& t^{-1} 2^{-2j} \lVert \hat{f}^{\pm}_j \rVert_{L^\infty} + t^{-1} 2^{-3/2j} \lVert \partial_\xi \hat{f}^\pm_j \rVert_{L^2}\\
        \lesssim& t^{-1/3} (t^{1/3} 2^j)^{-3/2} \lVert f \rVert_{X^\pm_j}
   \end{split}\end{equation*}
   as required.  
   
   \paragraph{Case $2^k \gg 2^{j}$} In this case, $|\partial_\xi \phi| \sim 2^{2k}$, and an integration by parts argument like above gives
   \begin{equation*}\begin{split}
       \left|Q^\pm_j u_n(t)\right| \lesssim& t^{-1} 2^{2j-4k} \lVert \hat{f}^{\pm}_j \rVert_{L^\infty} + t^{-1} 2^{j/2 - 2k} \lVert \partial_\xi \hat{f}^\pm_j \rVert_{L^2}\\
       \lesssim& t^{-1/3} (t^{1/3} 2^{-j/3} 2^{4/3k})^{-3/2})\lVert f \rVert_{X^\pm_j}
   \end{split}\end{equation*}
   
   \paragraph{Case $2^k \sim 2^{2j}$} In this case, the analysis is more delicate due to the vanishing of $\partial_\xi \phi$.  Let us write
   \begin{equation}
       Q^\pm_j u_n(t) = I_j + \sum_{r=1}^2\sum_{\ell = \ell_0}^{j-10} I_{j,\ell,r}
   \end{equation}
   where
   \begin{equation*}\begin{split}
       I_j =& \frac{1}{\sqrt{2\pi}} \int_{\bbT} (1 - \psi_{\leq j-10}^\pm(\xi - \xi_1) - \psi_{\leq j-10}^\pm(\xi - \xi_2)) e^{it\phi} \hat{f}^{\pm}_j(\xi)\;d\xi\\
       I_{j,\ell,r} =& \frac{1}{\sqrt{2\pi}} \int_{\bbT} \psi_{\ell}^\pm(\xi - \xi_k) e^{it\phi} \hat{f}^{\pm}_j(\xi)\;d\xi, \qquad \ell > \ell_0\\
       I_{j,\ell_0,r} =& \frac{1}{\sqrt{2\pi}} \int_{\bbT} \psi_{\leq \ell_0}^\pm(\xi - \xi_k) e^{it\phi} \hat{f}^{\pm}_j(\xi)\;d\xi\\
   \end{split}\end{equation*}
   and $\ell_0$ is defined such that $2^{\ell_0} \sim t^{-1/3}(t^{1/3} 2^j)^{-1/5}$.  We will show how the estimate the $I_{j,\ell,r}$ terms: the estimate for $I_j$ is similar.  First, suppose $\ell > \ell_0$.  Then, on the support of the integrand of $I_{j,\ell,r}$, we have that $|\partial_\xi \phi| \sim 2^{j + \ell}$, so integration by parts gives
   \begin{equation*}\begin{split}
       |I_{j,\ell,r}| \lesssim& t^{-1} 2^{-j-\ell} \lVert \hat{f}^\pm_j\rVert_{L^\infty} + t^{-1} 2^{-j - \ell/2} \lVert \partial_\xi \hat{f}^\pm_j \rVert_{L^2}\\
       \lesssim& (t^{-1} 2^{-j-\ell} + t^{-1} 2^{-j - \ell/2}) \lVert f \rVert_{X^\pm_j}
   \end{split}\end{equation*}
   For $\ell = \ell_0$, we introduce a further decomposition
   \begin{equation*}\begin{split}
       I_{j,\ell_0,r} =& \frac{1}{\sqrt{2\pi}} \int_{\bbT} \psi_{\leq \ell_0}(\xi - \xi_r) e^{it\phi} \left(\hat{f}^\pm_j(\xi) - \hat{f}^\pm_j(\xi_r)\right)\;d\xi\\
       &\qquad + \frac{\hat{f}^{\pm}_j(\xi_r)}{\sqrt{2\pi}} \int_{\bbT} \psi_{\leq \ell_0}(\xi - \xi_r) \left(e^{it\phi} - e^{it(\phi(\xi_r) + \phi''(\xi_r) \xi^2}\right)\;d\xi\\
       &\qquad + \frac{\hat{f}^{\pm}_j(\xi_r)}{\sqrt{2\pi}} e^{it\phi(\xi_k)} \int_{\bbT} \psi_{\leq \ell_0}(\xi - \xi_r) e^{it\phi''(\xi_r) \xi^2}\;d\xi\\
       &= I_{j,\ell_0,r}^1 + I_{j,\ell_0,r}^2 + I_{j,\ell_0,r}^3
   \end{split}\end{equation*}
   Using the Morrey-Sobolev embedding $\dot{H}^1 \to C^{1/2}$, we find that
   \begin{equation*}
       |I^1_{j,\ell_0,r}| \lesssim 2^{3/2\ell_0} t^{1/6} \lVert f \rVert_{X}
   \end{equation*}
   Moreover, since $\phi(\xi) = \phi(\xi_r) + \phi''(\xi_r) (\xi-\xi_r)^2 + O(2^j |\xi-\xi_r|^3)$, we have that
   \begin{equation*}
       |I^2_{j,\ell_0, r}| \lesssim t 2^j 2^{4\ell_0}
   \end{equation*}
   Finally, for $I^3_{j,\ell_0,r}$, we use the classical method of stationary phase.  Using the inclusion $\bbT \sim [0,2\pi] \to \bbR$, we find that
   \begin{equation*}\begin{split}
       I^3_{j,\ell_0,r} =& \frac{\hat{f}^{\pm}_j(\xi_k)}{\sqrt{2\pi}} e^{it\phi(\xi_r)} \int_{\bbR} \psi_{\leq \ell_0}(\xi - \xi_r) e^{it\phi''(\xi_r) (\xi-\xi_r)^2}\;d\xi\\
       =& 2^{\ell_0}\frac{\hat{f}^{\pm}_j(\xi_r)}{\sqrt{2\pi}} e^{it\phi(\xi_r)} \int_{\bbT} \psi_{\leq 0}(\xi) e^{it\phi''(\xi_r)2^{2\ell_0} \xi^2}\;d\xi\\
       =& \frac{\hat{f}^{\pm}_j(\xi_r)}{\sqrt{t|\phi''(\xi_r)|}} e^{it\phi(\xi_r) + i \frac{\pi}{4} \sgn(\phi''(\xi_r))} + O(t^{-3/2} 2^{-2\ell_0} 2^{-3/2j}\lVert f \rVert_{X})
   \end{split}\end{equation*}
   Recalling that $\cos \xi_k = (-1)^k \sqrt{4 - (n/t)^2}$ gives the desired result.
\end{proof}

From this estimate, we can readily derive the following bilinear decay estimate:	
\begin{cor}\label{cor:decay-estim}
	Let $u_n(t) = \linProp f_n$ with $f_n \in X$.  Then,
	\begin{equation*}
	    |u_n \Psi u_n| \lesssim t^{-1} \lVert f \rVert_X^2
	\end{equation*}
\end{cor}

We also have the following bilinear estimate, which will be useful later:
\begin{cor}\label{cor:bilin-sin-est}
    If $u_n(t) = \linProp f_n$, $f_n \in X$, then
    \begin{equation*}
        \left| \sin\left( \frac{D}{2}\right) |u_n|^2\right| \lesssim t^{-1} \lVert f \rVert_X^2
    \end{equation*}
\end{cor}
\begin{proof}
    Let us define $Q_j^0$ to be the frequency projector which projects to frequencies within a distance $\sim 2^j$ from $0$ on the torus for $2^j \geq 2t^{-1/3}$, and to project to all frequencies within a distance $\lesssim t^{-1/3}$ of $0$ for $2^{j-1} < t^{-1/3} \leq 2^j$.  Then, we can write
    \begin{equation*}
        \lVert \sin\left( \frac{D}{2}\right) |u_n|^2 \rVert_{\ell^\infty} \lesssim \sum_{t^{-\frac{1}{3}} \leq 2^j \lesssim 1} 2^j \lVert Q_j^0 |u_n|^2 \rVert_{\ell^\infty}
    \end{equation*}
    Performing a further dyadic decomposition in frequency, we can write
    \begin{equation*}\begin{split}
        Q_j^0 |u_n|^2 =& Q_j^0(Q_{\gtrsim j}^+ u_n \overline{Q_{\gtrsim j}^+ u_n})\\
        &+ Q_j^0(Q_{\sim j}^+ u_n \overline{Q_{\ll j}^+ u_n})\\
        &+ \bsime
    \end{split}\end{equation*}
    By~\Cref{thm:lin-decay-al}, the leading order contribution comes from the first term.  Thus,
    \begin{equation*}
        \lVert Q_j^0 |u_n|^2 \rVert_{\ell^\infty} \lesssim t^{-1} 2^{-j} \lVert f_n \rVert_{X}^2
    \end{equation*}
    which completes the proof.
\end{proof}

\section{The approximate self-similar dynamics}\label{sec:approx-ss}

From our heuristic calculations in the introduction, we expect that wave-packets near the degenerate frequencies $\pm \fhPi$ should behave approximately like solutions to the complex mKdV equation (up to phase modulation and translation).  Since the low frequency component of complex mKdV exhibits nonlinear self-similar behavior, let us consider a family of self-similar solutions given by $\fs(x,t;\alpha)$, where $\fs$ satisfies
\begin{equation}\label{eqn:cmkdv-self-sim}
    \left\{\begin{array}{c}
        (\partial_t + \frac{1}{3}\partial_x^3) \fs(x,t;\alpha) = 2|\fs(x,t;\alpha)|^2 \partial_x \fs(x,t;\alpha)\\
        \hat{\fs}(0,t;\alpha) = \alpha
    \end{array}\right.
\end{equation}
To uniquely specify $\fs$, we also impose the condition that $\fs$ decay at the same rate a solution to the complex mKdV equation with localized initial data.  These functions play a key role in our arguments, so we will review their key properties now.

\subsection{Properties of self-similar solutions}

By writing $\fs(x,t;\alpha) = t^{-1/3} v(t^{-1/3} x;\alpha)$, we see that $v$ satisfies
\begin{equation}\label{eqn:cmdkv-self-sim-prof-expanded}
    \left\{\begin{array}{c}
        \partial_y^3 v(y;\alpha) -  y\partial_y v(y;\alpha) - v(y;\alpha) = 6|v(y;\alpha)|^2 \partial_y v(y;\alpha)\\
        \hat{v}(0;\alpha) = \alpha
    \end{array}\right.
\end{equation}
where $y = t^{-1/3}x$ is the self-similar variable.  Since we are interested in solutions $\fs$ which behave like linear solutions, let us assume that we can treat the nonlinearity in~\eqref{eqn:cmdkv-self-sim-prof-expanded} perturbatively as $y \to \infty$, giving us the expression
\begin{equation*}
    v(y;\alpha) = c_1(\alpha) \operatorname{Ai}(y) + c_2(\alpha) \operatorname{Bi}(y) + c_3(\alpha) \operatorname{Gi}(y) + o(1) \qquad\qquad \text{as } y \to \infty
\end{equation*}
where $\operatorname{Ai}$ and $\operatorname{Bi}$ are the Airy functions of the first and second kind and $\operatorname{Gi}$ is Scorer's function~\cite{scorerNumericalEvaluationIntegrals1950}, which satisfies
\begin{equation*}
    (\partial_y^2 - y) \operatorname{Gi}(y;\alpha) = \frac{1}{\pi}
\end{equation*}
Based on our requirement that $v$ be bounded, we can immediately conclude that $c_2(\alpha) = 0 $, and if we require that $(\partial^2_y - y)v \in L^2$ (so $\fs$ is localized), we see that $c_3(\alpha) = 0$ as well.  Thus, we seek solutions $v$ which are asymptotic to some multiple of the Airy function $\operatorname{Ai}$ as $y \to \infty$.

For $\alpha$ sufficiently small and real, these solutions were studied in~\cite{correiaAsymptoticsFourierSpace2020} (in comparing the results, note that our $\alpha$ corresponds to their $c$).  There, it was shown $v$ solves
\begin{equation}\label{eqn:cmkdv-self-sim-prof}
    \left\{\begin{array}{c}
        yv(y;\alpha) -  \partial_y^2 v(y;\alpha) = 6 |v(y;\alpha)|^2v(y;\alpha)\\
        \hat{v}(0;\alpha) = \alpha\\
        \hat{v}(\cdot;\alpha) \text{ is continuous at } 0
    \end{array}\right.
\end{equation}
(Actually, \cite{correiaAsymptoticsFourierSpace2020} considers a version of~\Cref{eqn:cmkdv-self-sim,eqn:cmkdv-self-sim-prof} with slightly different different constants, but this is easily handled by rescaling).  Since~\eqref{eqn:cmdkv-self-sim-prof-expanded} is phase rotation invariant, we see that the solution to~\eqref{eqn:cmdkv-self-sim-prof-expanded} for complex $\alpha$ is simply given by $v(y;r e^{ita}) =e^{ita} v(y;r)$, so we can also use~\eqref{eqn:cmkdv-self-sim-prof} to define $v$ for $\alpha$ small and complex.  

We now turn to extracting bounds and decay estimates for $\fs$.  The results of~\cite{correiaAsymptoticsFourierSpace2020} imply that the function 
\begin{equation*} 
    \theta(x,t) = e^{\frac{t}{3} \partial_x^3} \fs(x,t) = e^{\frac{1}{3} \partial_x^3} v(x)
\end{equation*}
satisfies the bounds
\begin{align}
    \lVert \hat{\theta} \rVert_{L^\infty} \lesssim& |\alpha|\label{eqn:theta-L-inf-bdd}\\
    \lVert x\theta \rVert_{L^2} \lesssim& t^{1/6} |\alpha|^3\label{eqn:theta-L2-bdd}
\end{align}
The second inequality comes from the identity
\begin{equation}\label{eqn:cmkdv-self-sim-fourier-ident}
    e^{-\frac{t}{3}\partial_x^3}xe^{\frac{t}{3} \partial_x^3} \fs = (x - t \partial_x^2) \fs(x,t;\alpha) = 6|\fs(x,t;\alpha)|^2 \fs(x,t;\alpha)
\end{equation}
which can be read off of~\eqref{eqn:cmkdv-self-sim-prof} using the scaling of $\fs$.  The bounds~\eqref{eqn:theta-L-inf-bdd} and~\eqref{eqn:theta-L2-bdd} imply that $\fs$ obeys the same sort of linear decay estimates as solutions to the linear KdV equation.  In particular, for $2^j \geq t^{-1/3}$
\begin{equation}\label{eqn:fs-decay-bdds}
    |Q_{j}^\tcts \fs(x,t;\alpha)| \lesssim |\alpha|\begin{cases}
        t^{-1/2} 2^{-j/2} & -x \sim t 2^{2j}\\
        t^{-5/6} 2^{-3/2 j} & |x| \ll t 2^{2j}\\
        t^{-5/6} 2^{j/2} 2^{-2k} & |x| \sim t 2^{2k} \gg t 2^{2j}
    \end{cases}
\end{equation}
and
\begin{equation}\label{eqn:fs-decay-bdds-low}
    |Q_{\leq \log t^{-1/3}}^\tcts\fs(x,t;\alpha)| \lesssim |\alpha|\begin{cases}
        t^{-1/3} & |x| \lesssim t^{1/3}\\
        t^{-1} 2^{-2k} & |x| \sim t 2^{2k} \gg t^{1/3}
    \end{cases}
\end{equation}
see~\cite[Section 3]{stewartLongTimeDecay2021}.  In our work, it is necessary to allow the profile of the self-similar solutions to change in time.  Given a solution $\fs(x;\alpha)$, we define $D_\alpha \fs$ by
\begin{equation}\label{eqn:D-alpha-def}
    \partial_s \fs(x,t;\alpha(t)) = [D_\alpha \fs] \cdot \alpha'(s)
\end{equation}
We employ the somewhat unusual notation to highlight the fact that $[D_\alpha \fs]$ is \emph{not} the derivative of $\fs$ with respect to $\alpha$ in the sense of complex variables.  (Indeed, $\fs$ is not complex differentiable in $\alpha$.)  Instead, we identify $\alpha = (a + ib) \in \bbC$ with $(a,b) \in \bbR^2$ and define $[D_\alpha \fs]$ to be the derivative of the corresponding function from $\bbR^2$ to $\bbC$.  By using the formula for $\hat{\fs}(\xi;\alpha)$ given in~\cite{correiaAsymptoticsFourierSpace2020} together with the decay estimates for the linear Airy equation, we compute that
\begin{equation}\label{eqn:D-alpha-fs-bdds}\begin{split}
    \lVert [D_\alpha Q_{j}^\tcts \fs] \rVert_{L^\infty} \lesssim& t^{-1/2} 2^{-j/2} \ln(2 + t^{1/3} 2^j)\\
    \lVert \chi_{\ll j} [D_\alpha Q_{j}^\tcts \fs] \rVert_{L^\infty} \lesssim& t^{-5/6} 2^{-3/2j} \ln(2 + t^{1/3} 2^j)\\
    \lVert \chi_{k} [D_\alpha Q_{j}^\tcts \fs] \rVert_{L^\infty} \lesssim& t^{-5/6} 2^{j/2 - 2k} \ln(2 + t^{1/3} 2^j)\qquad k > j + 20\\
    \lVert [D_\alpha \fs] \rVert_{L^p} \lesssim& t^{-\frac{1}{3} + \frac{1}{3p}}
\end{split}\end{equation}
for $4 < p \leq \infty$, see~\cite[Section 4]{stewartLongTimeDecay2021}.
\begin{rmk}\label{rmk:D-alpha-unif}
    In the above inequalities, the implicit constants can be taken to be uniform in $\alpha$ for $|\alpha| \ll 1$.
\end{rmk}
By combining this estimate with~\cref{eqn:fs-decay-bdds,eqn:fs-decay-bdds-low,eqn:cmkdv-self-sim-fourier-ident}, we see that
\begin{equation}\label{eqn:D-alpha-L-fs-bdds}
    \lVert (x - t\partial_x^2)[D_\alpha \fs] \rVert_{L^2} \lesssim t^{1/6} |\alpha|^2
\end{equation}

\subsection{The approximate self-similar solutions}
Using these self-similar solutions, we define the approximate self-similar solution to Ablowitz-Ladik by
\begin{equation}\begin{split}
    \widehat{\mSS}(\xi, t; \alpha_+, \alpha_-) =& \rho(t^{1/5} (\xi - \pi/2)) e^{-2i(\xi - \fhPi) t}\hat{\fs}(\xi - \pi/2,t;\alpha_+)\\
    &+ \rho(t^{1/5} (\xi + \pi/2)) e^{2i(\xi + \fhPi) t}\hat{\fs}(\xi + \pi/2,-t;\alpha_-)
\end{split}\end{equation}
where $\rho$ is a cut-off function with $\rho(\xi) = 1$ for $|\xi| < 1/2$ and $\rho(\xi) = 0$ for $|\xi| > 1$.  Occasionally, it will be useful to consider the parts of $\mSS_n$ concentrated at the positive and negative frequencies separately, so we also write
\begin{equation}\label{eqn:mssn-pm-def}
    \mSS_n^+(t) = \mSS_n(t, \alpha_+, 0)\qquad\qquad\qquad \mSS_n^-(t) = \mSS_n(t, 0, \alpha_-)
\end{equation}
for the parts of the approximately self-similar solution localized near frequencies $\pm \fhPi$, respectively, and denote the corresponding self-similar solutions by
\begin{equation*}
    \fs^\pm(x,t) = \fs(x,t;\alpha_\pm)
\end{equation*}
\begin{rmk}
    The choice to cut-off at $|\xi \mp \fhPi| \sim t^{-1/5}$ can be motivated as follows: 
\end{rmk}

We now consider estimates for $\mSS_n$.  For $2^j, 2^k \lesssim |t|^{1/5}$ and $\nu \in \{+, -\}$, we see that
\begin{equation*}\begin{split}
    \chi_{k}^\nu \mSS_n =& (\nu i)^n \chi_k^\tcts(2n-\nu t,t) \fs(2n-\nu t, \nu t, \alpha_\nu)\\
    Q_{j}^\nu \mSS_n =& (\nu i)^n Q_j^\tcts\fs(2n-\nu t, \nu t, \alpha_\nu)\\
    \chi_{k}^\nu Q_{j}^\nu \mSS_n =& (\nu i)^n \chi_k^\tcts(2n-\nu t,t) Q_j^\tcts \fs(2n-\nu t, \nu t, \alpha_\nu)\\
\end{split}\end{equation*}
so we can obtain decay estimates for $\mSS_n$ from those for $\fs$ in a straightforward manner.  Since we will use these estimates extensively later, we record them here for later reference:
\begin{equation}\label{eqn:mssn-lin-ests}\begin{split}
    \lVert \mSS_n \rVert_{\ell^p} \lesssim& (|\alpha_{+}| + |\alpha_{-}|)  t^{\frac{1}{3p}-\frac{1}{3}}\\
    \lVert Q^{\pm}_j \mSS_n \rVert_{\ell^p} \lesssim& |\alpha^\pm|\begin{cases}
        t^{\frac{1}{p} - \frac{1}{2}} 2^{\left(\frac{2}{p}-\frac{1}{2}\right)j} & t^{-1/3} \leq 2^j \lesssim t^{-1/5}\\
        0 & \text{else}
    \end{cases}\\
    \lVert \chi_{k}^{\pm} \mSS_n \rVert_{\ell^p} \lesssim& (|\alpha_{+}| + |\alpha_{-}|)\begin{cases}
        t^{\frac{1}{p} - \frac{1}{2}} 2^{\left(\frac{2}{p}-\frac{1}{2}\right)j} & t^{-1/3} \leq 2^k \lesssim t^{-1/5}\\
        t^{\frac{1}{p} - \frac{5}{6}} 2^{\left(-\frac{3}{2} + \frac{2}{p}\right)j} & 2^k \gg t^{-1/5}\\
    \end{cases}\\
    \lVert \chi_{k}^\mu Q_j^\nu \mSS_n \rVert_{\ell^p} \lesssim& |\alpha_{\nu}| \begin{cases}
        0 & 2^j \gg t^{-1/5}\\
        t^{\frac{1}{p} - \frac{1}{2}} 2^{\left(\frac{2}{p}-\frac{1}{2}\right)j} & \mu = \nu, 2^j \sim 2^k\\
        t^{\frac{1}{p} - \frac{5}{6}} 2^{\left(\frac{2}{p} - 2\right)k} 2^{j/2} & \mu = \nu, 2^j \ll 2^k\\
        t^{\frac{1}{p} - \frac{5}{6}} 2^{\left(-\frac{3}{2} + \frac{2}{p}\right)j} & \text{else}\\
    \end{cases}
\end{split}\end{equation}	
An important property of the self-similar solutions $\fs$ is the identity
\begin{equation*}\label{eqn:self-sim-vf-identity}
    e^{-\frac{1}{3}\partial_x^3 t}x e^{\frac{1}{3}\partial_x^3 t}\fs =e^{-\frac{1}{3}\partial_x^3 t}x \theta =  2t|\fs|^2\fs
\end{equation*}
A similar identity holds for $\mSS_n$, which allows us to estimate $L\mSS_n$ pointwise in space:
\begin{lemma}\label{lem:LSS-expression}
    We can write
    \begin{equation*}
        L \mSS_n = \fL_n^+ + \fL_n^- + \fC_n^+ + \fC_n^-
    \end{equation*}
    where $\fL_n^\pm$ and $\fC_n^\pm$ have Fourier transforms supported on $|\xi \pm \fhPi| \lesssim t^{-1/5}$ and satisfy the estimates
    \begin{equation}\label{eqn:fL-ests}\begin{aligned}
        \lVert \fL_n^\pm \rVert_{\ell^2} \lesssim& |\alpha_\pm| &\qquad
        \lVert \fL_n^\pm \rVert_{\ell^\infty} \lesssim& |\alpha_\pm| t^{-1/5}\\
        \lVert \Psi \fL_n^\pm \rVert_{\ell^\infty} \lesssim& |\alpha_\pm| t^{-2/5}&\qquad
        \lVert Q_{\lesssim j}^\pm \fL_n^\pm \rVert_{\ell^\infty} \lesssim& |\alpha_\pm| t^{1/2} 2^{7/2j}\\
        \lVert \chi_{k}^\pm \fL_n^\pm \rVert_{\ell^\infty} \lesssim& |\alpha_\pm| t^{1/2} 2^{7/2k}&\qquad
        \lVert \chi^\pm_{k} \fL^\pm_n \rVert_{\ell^\infty} \lesssim& |\alpha_\pm| t^{-11/15}2^{-2k}               
        \qquad\qquad 2^k \gg t^{-1/5}\\
        \lVert \chi_{k}^\pm \Psi \fL_n^\pm \rVert_{\ell^\infty} \lesssim& |\alpha_\pm| t^{1/2} 2^{9/2k} &\qquad
        \lVert \chi_{k}^\pm \Psi \fL^\pm_n \rVert_{\ell^\infty} \lesssim& |\alpha_\pm| t^{-14/15} 2^{-2k}	   \qquad\qquad 2^k \gg t^{-1/5}
    \end{aligned}\end{equation}
    and
    \begin{equation}\label{eqn:fC-ests}\begin{aligned}
        \lVert \fC^\pm_n \rVert_{\ell^2} \lesssim& |\alpha_\pm|^3 t^{1/6}&\qquad
        \lVert \fC^\pm_n \rVert_{\ell^\infty} \lesssim& |\alpha_\pm|^3\\
        \lVert \Psi \fC^\pm_n \rVert_{\ell^\infty} \lesssim& |\alpha_\pm|^3t^{-1/3}&\qquad
        \lVert Q_j^\pm \fC^\pm \rVert_{\ell^\infty} \lesssim& |\alpha_\pm|^3 t^{-1/2} 2^{-3/2j}\\
        \lVert \chi_{k}^\pm \fC^\pm_n \rVert_{\ell^\infty} \lesssim& |\alpha_\pm|^3 t^{-3/2} 2^{-3/2 k}&\qquad
        \lVert \chi_{k}^\pm \Psi \fC^\pm_n \rVert_{\ell^\infty} \lesssim& |\alpha_\pm|^3 t^{-3/2} 2^{-k/2}
    \end{aligned}\end{equation}
\end{lemma}
\begin{proof}
    Using the identity
    \begin{equation*}
        \mSS_n(t, \alpha_{+}, \alpha_{-}) = \mSS_n(t, \alpha_{+}, 0) + \mSS_n(t, 0, \alpha_{-})
    \end{equation*}
    we see that it suffices to consider the contributions from near $\fhPi$ and $-\fhPi$ separately.  A direct computation shows that
    \begin{equation*}\begin{split}
        \cF(L \mSS_n(t;\alpha_+, 0)) =& e^{2it\cos \xi} i \partial_\xi \left(e^{-2it\cos \xi} e^{it\left(2\xi + \frac{\xi - \fhPi)^3}{3}\right)} \rho(|t|^{1/5} (\xi - \fhPi)) \hat{\theta}(\xi - \fhPi, t; \alpha_{+}) \right)\\
        =& i e^{it \fhPi} e^{2it\cos\xi} \partial_\xi \left( E(\xi,t) \rho(|t|^{1/5} (\xi - \fhPi))\right) \hat{\theta}(\xi - \fhPi, t; \alpha_{+})\\
        &+ i e^{it \fhPi} e^{2it\cos\xi}  \left( E(\xi,t) \rho(|t|^{1/5} (\xi - \fhPi))\right) \partial_\xi\hat{\theta}(\xi - \fhPi, t; \alpha_{+})\\
        =:& \cF\left( \fL^+_n + \fC^+_n\right)
    \end{split}\end{equation*}
    where $E(\xi,t) = \exp\left(-it2\left(\cos \xi + (\xi - \fhPi) - \frac{\xi - \fhPi)^3}{6}\right)\right)$.  For $\fL^+_n$, arithmetic simplifications yield
    \begin{equation*}\begin{split}
        \cF(\fL^+_n) =& e^{-2it\xi} 2t \left(-\sin \xi + 1 - \frac{\left(\xi - \fhPi\right)^2}{2}\right) \rho\left(|t|^{\frac{1}{5}} \left(\xi - \fhPi\right)\right)\hat\fs(\xi - \fhPi, t; \alpha_{+})\\
        &+ i e^{-2it\xi} |t|^{\frac{1}{5}} \rho'\left(|t|^{\frac{1}{5}} \left(\xi - \fhPi\right)\right)\hat\fs(\xi - \fhPi, t; \alpha_{+})\\
        =:& \hat{\rmI} + \hat{\rmII}
    \end{split}\end{equation*}
    For the first term, we note that, since $\sin(\xi) = -1 + \frac{\left(\xi - \fhPi\right)^2}{2} + O\left( \left(\xi - \fhPi\right)^4\right)$, the frequency localization of $\fL^+_n$ implies that
    \begin{equation*}\begin{split}
        \lVert \rmI \rVert_{\ell^2} =& \lVert \hat{\rmI} \rVert_{L^2}\\
        \lesssim& \lVert \hat{s}(\xi,t;\alpha_\pm) \rVert_{L^\infty_\xi}\\
        \lesssim& |\alpha_{\pm}|
    \end{split}\end{equation*}
    and
    \begin{equation*}\begin{split}
        \lVert \rmII \rVert_{\ell^2} =& \lVert \hat{\rmII} \rVert_{L^2}\\
        \lesssim& \lVert \hat{s}(\xi,t;\alpha_\pm) \rVert_{L^\infty_\xi}\\
        \lesssim& |\alpha_{\pm}|
    \end{split}
    \end{equation*}
    which gives the first $\ell^2$ estimate for $\fL_n^+$.  For the $\ell^\infty$ estimates, we note that $\rmI$ and $\rmII$ can be seen as Fourier multipliers applied to $\fs$.  Thus, using the $L^\infty$ decay bounds for $Q_j^\tcts \fs$, we find that
    \begin{equation*}\begin{split}
        \lVert Q_j^+ \rmI \rVert_{\ell^\infty} \lesssim& t 2^{4j} \lVert \fs \rVert_{L^\infty} \lesssim |\alpha_{+}| t^{1/2} 2^{7/2j}
    \end{split}\end{equation*}
    and, since $\rmII$ is supported on the region $|\xi - \fhPi| \sim |t|^{-1/5}$, 
    \begin{equation*}\begin{split}
        \lVert \rmII \rVert_{\ell^\infty} \lesssim& \sum_{2^j \sim t^{-1/5}} |t|^{1/5} \lVert Q_j^+ \fs \rVert_{L^\infty} \lesssim |\alpha_{+}| t^{-1/5}
    \end{split}\end{equation*}
    which is enough to give the  $\ell^\infty$ estimates for $\fL^\pm$, $\Psi \fL^\pm$ and $Q_{\leq j}^\pm \fL^\pm$.  A similar argument gives the remaining bounds for $\chi_k^\pm \fL_n^\pm$ and $\chi_k^\pm \Psi \fL_n^\pm$.
    
    Turning to $\fC_n^+$, the identity~\eqref{eqn:cmkdv-self-sim-fourier-ident} gives us that
    \begin{equation*}
        \fC_n^+ = 6it Q_{\lesssim \log |t|^{1/5}}^\tcts |\fs|^2\fs(n - 2t;\alpha_+)
    \end{equation*}
    Thus,
    \begin{equation*}\begin{split}
        \lVert \fC_n^+ \rVert_{\ell^2} \lesssim& t \lVert \fs(x,t;\alpha_+) \rVert_{L^6}^3
        \lesssim |\alpha_+|^3 t^{1/6}
    \end{split}
    \end{equation*}
    which is enough to get the $\ell^2$ bound in~\eqref{eqn:fC-ests}.  The remaining $\ell^\infty$ estimates in~\eqref{eqn:fC-ests} follow from the estimates for $\fs$. For instance, we have
    \begin{equation*}\begin{split}
        \lVert Q_j^+ \fC^+_n \rVert_{\ell^\infty} \lesssim& t\lVert Q_j^\tcts |\fs|^2\fs \rVert_{L^\infty}
        \lesssim |\alpha_{+}|^3 t^{-1/2} 2^{-3/2j}
    \end{split}\end{equation*}
    which is enough to give the first three $\ell^\infty$ estimates, and, by~\Cref{lem:LP-pseudoloc-lemma}
    \begin{equation*}\begin{split}
        \lVert \chi_{k}^+ Q_j^+ \fC^+_n \rVert_{\ell^\infty} \lesssim& t\lVert \chi_k^\tcts Q_j^\tcts |\fs|^2\fs \rVert_{L^\infty}\\
        \lesssim& |\alpha_{+}|^3 t^{-1/2} 2^{-3/2k}
    \end{split}\end{equation*}
    which gives the last two inequalities.
\end{proof}


\subsection{Equation error for the approximate self-similar solution\label{sec:ss-eqn-error}}

To justify the name `approximate self-similar solution' for $\mSS_n$, we should verify that the error
\begin{equation}\label{eqn:R-n-def}
    R_n = \frac{d}{dt} \mSS_n(t; \alpha_+(t), \alpha_-(t)) -i\discrAvg \mSS_n + i|\mSS_n|^2 \discrAvg \mSS_{n}
\end{equation}
is small in some sense.  In this section, we will derive $\ell^2$ bounds for the error term.  (We will also need weighted $\ell^2$ bounds at a later stage in the proof, but since these require space-time resonance ideas, we postpone them until~\Cref{sec:weighted-l2-eqn-err}).  We begin by introducing the decomposition
\begin{equation*}
    R_n = R_n^+ + R_n^- + R_n^{\text{cross}}
\end{equation*}
where
\begin{equation}\label{eqn:R-n-pm-def}
    R_n^\pm = \frac{d}{dt} \mSS_n^\pm(t) - i \Psi \mSS_n^\pm + i |\mSS_n^\pm|^2 \Psi \mSS_n^\pm
\end{equation}
and
\begin{equation}\label{eqn:R-n-cross-def}
    R_n^{\text{cross}} = |\mSS_n|^2\Psi \mSS_n - |\mSS_n^+|^2\Psi \mSS_n^+ - |\mSS_n^-|^2\Psi \mSS_n^-
\end{equation}
We will show how to obtain bounds for the terms containing $R_n^+$.  The terms containing $R_n^-$ are identical (up to changing signs in the appropriate places), and the terms contained in $R_n^\text{cross}$ are easier to control because $\mSS_n^+$ and $\mSS_n^-$ are concentrated on sets which become increasingly separated as $t \to \infty$, producing additional decay.  Notice that
\begin{equation*}\begin{split}
    \partial_t \hat{\mSS^+}(\xi + \fhPi) =& \partial_t\left(\rho(t^{1/5}\xi)e^{-2i\xi t}\right)\hat{\fs}(\xi, t; \alpha_+(t))\\
    &+\rho(t^{1/5}\xi)e^{-2i\xi t}[D_\alpha\hat{\fs}(\xi, t; \alpha_+(t))]\cdot \partial_t \alpha_+(t)\\
    &+\rho(t^{1/5}\xi)e^{-2i\xi t}\partial_t\hat{\fs}(\xi - \pi/2, t; \alpha_+(t))\\
\end{split}\end{equation*}
Using~\eqref{eqn:cmkdv-self-sim} to expand $\partial_t \hat{\fs}$ and substituting into the expression for $R_n^+$, we obtain terms which are linear in $\fs$, a modulation term containing $D_\alpha \fs$, and terms which are cubic in $\fs$.  Motivated by this, we introduce the further decomposition
\begin{equation}\label{eqn:Rn-decomp}
    R_n^+ = R_n^{\text{lin},+} + R_n^{\text{mod},+} + R_n^{\text{cubic},+}
\end{equation}
where
\begin{align}
    \begin{split}\hat{R}^{\text{lin},+}(\xi+\pi/2) =& \biggl(\frac{t^{-4/5}}{5}\xi \rho'(t^{1/5}\xi) - 2i\left(\cos (\xi+\pi/2) + \xi - \frac{\xi^3}{6}\right)\rho(t^{1/5}\xi) \biggr) e^{-2i\xi t}\hat{\fs}(\xi,t;\alpha_+)\end{split}\label{eqn:R-n-lin-def}\\
    \hat{R}^{\text{mod},+}(\xi+\pi/2) =& \rho(t^{1/5}\xi) e^{-2i\xi t} [D_{\alpha} \hat{\fs}(\xi,t;\alpha_+)] \cdot \partial_t \alpha_+\label{eqn:R-n-mod-def}\\
    \hat{R}^{\text{cubic},+}(\xi+\pi/2) =& i \mathcal{F}(|\mSS_n^+|^2 \Psi \mSS_n^+)(\xi+\pi/2) + 2 \rho(t^{1/5}(\xi) e^{-2i\xi t} \mathcal{F}(|\fs|^2\partial_x \fs)(\xi+\pi/2)\label{eqn:R-n-cubic-def}
\end{align}
with $[D_\alpha \hat{\fs}] \cdot \partial_t \alpha$ interpreted in the same sense as in~\eqref{eqn:D-alpha-def}.  We now handle each term separately.
\subsubsection{The linear error term}
Observe that $R_n^{\text{lin}, +}$ has a similar form to $\fL_n^+$.  Thus, by adapting the reasoning used to prove~\eqref{eqn:fL-ests}, we find that
\begin{equation}\label{eqn:R_n-lin-ests}\begin{split}
\lVert \chi_{\sim k}^+ Q_{\lesssim j}^+ R_n^{\text{lin},+}\rVert_{\ell^2} \lesssim& \begin{cases}
        \epsilon 2^{11/2k} & k \leq j+20\\
        \epsilon t^{-1/3} 2^{\frac{11}{2} j - k} & k > j + 20
    \end{cases}\\
    \lVert \chi_{\sim k} R_n^{\text{lin},+} \rVert_{\ell^2} \lesssim& \begin{cases}
        \epsilon  2^{11/2k} & 2^k \lesssim s^{-1/5}\\
        \epsilon t^{-\frac{43}{30}} 2^{-k} & 2^k \gg s^{-1/5} 
    \end{cases}
\end{split}\end{equation}

\subsubsection{The modulation error term}
Here, we simply use the $L^\infty$ estimate from~\eqref{eqn:D-alpha-fs-bdds} to obtain the bounds
\begin{equation}\label{eqn:R_n-mod-ests}\begin{split}
    \lVert Q_{\lesssim j}^+ R_n^{\text{mod},+} \rVert_{\ell^2} \lesssim& 2^{j/2} \ln(2 + t^{1/3} 2^{j}) |\partial_t \alpha_+|\\
    \lVert \chi_{\sim k}^+ Q_{\lesssim j}^+ R_n^{\text{mod},+} \rVert_{\ell^2} \lesssim&\begin{cases}
        2^{k/2} \ln(2 + t^{1/3} 2^{k}) |\partial_t \alpha_+| & k \leq j + 20\\
        t^{-1/3} 2^{3/2j - 2k}\ln(2 + t^{1/3} 2^{j}) |\partial_t \alpha_+| & k > j + 20
    \end{cases}\\
    \lVert \chi_{\sim k}^+ R_n^{\text{mod},+} \rVert_{\ell^2} \lesssim&\begin{cases}
        2^{k/2} \ln(2 + t^{1/3} 2^{k}) |\partial_t \alpha_+| & 2^k \lesssim t^{-1/5}\\
        t^{-\frac{19}{30}} 2^{- 2k}\ln(2 + t^{\frac{2}{15}}) |\partial_t \alpha_+| & 2^k \gg t^{-1/5}
    \end{cases}\\
    \lVert \chi_{< j - 20}^+ Q_{\sim j}^+ R_n^{\text{mod},+} \rVert_{\ell^2} \lesssim& t^{-\frac{1}{3}} 2^{-j/2} \ln(2 + t^{1/3} 2^{j})|\partial_t \alpha_+|
\end{split}\end{equation}

\subsubsection{The cubic error term}
Finally, we consider the cubic error terms.  Defining 
\begin{equation*}
    m^R = 2\left(\cos (\sigma + \fhPi) \rho(t^{\frac{1}{5}}\eta) \rho(t^{\frac{1}{5}}(-\xi + \eta +\sigma)) \rho(t^{\frac{1}{5}}\sigma) + \sigma\rho(t^{\frac{1}{5}}\xi)\right)
\end{equation*}
we see that we can write $\hat{R}$ in terms of a cubic pseudoproduct involving $\fs^+$:
\begin{equation*}
    \hat{R}^{\text{cubic},+}(\xi + \fhPi) = ie^{-2i\xi t} \hat{T}_{m^R}(\fs^+, \fs^+, \overline{\fs^+})
\end{equation*}
Introducing the further decomposition $m^R = m^{R,1} + m^{R,2} + m^{R,3}$ with
\begin{equation*}\begin{split}
    m^{R,1} =& 2\SSCutoff\bigl(1-\SSCutoff[\eta]\SSCutoff[\sigma]\SSCutoff[(\xi - \eta - \sigma)] \bigr) \sigma\\
    m^{R,2} =& 2(\SSCutoff- 1) \SSCutoff[\eta]\SSCutoff[\sigma]\SSCutoff[(\xi - \eta - \sigma)] \sigma\\
    m^{R,3} =& 2\SSCutoff[\eta]\SSCutoff[\sigma]\SSCutoff[(\xi - \eta - \sigma)](\sigma - \sin \sigma)
\end{split}\end{equation*}
we can write
\begin{equation}\label{eqn:R-n-cubic-division}\begin{split}
    \hat{R}^{\text{cubic},+}(\xi + \fhPi)  =& ie^{-2i\xi t}\hat{T}_{m^{R, 1}}(\fs^+, \fs^+, \overline{\fs^+}) + ie^{-2i\xi t}\hat{T}_{m^{R,2}}(\fs^+, \fs^+, \overline{\fs^+}) +ie^{-2i\xi t}\hat{T}_{m^{R,3}}(\fs^+, \fs^+, \overline{\fs^+})\\
    =:& \hat{R}^{c, 1}(\xi + \fhPi) + \hat{R}^{c, 2}(\xi + \fhPi) + \hat{R}^{c, 3}(\xi + \fhPi)
\end{split}\end{equation}
\begin{rmk}
    Note that in~\eqref{eqn:R-n-cubic-division}, the left-hand side is a function defined for $\xi \in \bbT$, while the right-hand side is defined for $\xi \in \bbR$.  However, each of the symbols $m^{R,a}$ is compactly supported in $(-\pi, \pi)$, so we interpret~\eqref{eqn:R-n-cubic-division} by associating each $\xi \in \bbT = \bbR / \bbZ$ with its representative in the fundamental domain $(-\pi,\pi]$ and interpreting the equality pointwise almost everywhere for such $\xi$.  Since we will only be interested in $\ell^2$-based estimates, Plancherel's theorem assures us that it makes no difference whether we work in $\bbT$ or $\bbR$.
\end{rmk}

For $\Rncp{1}$, notice that $m^{R,1}$ is supported on the region $\max\{|\eta|,|\sigma|,|\xi - \eta - \sigma|\} \gtrsim t^{-1/5}$, so we can write
\begin{equation*}
    \Rncp{1} = (i)^n \sum_{2^\ell \gtrsim s^{-1/5}} 2^\ell T_{m_\ell}(Q_{\sim \ell}^\tcts\fs, \fs, \overline{\fs})(n - 2t) + \bsim
\end{equation*}
Similarly, for~$\Rncp{2}$, the symbol $m^{R,2}$ is supported in the region $|\xi| \sim \max\{|\eta|, |\sigma|, |\xi - \eta - \sigma|\} \sim s^{-\frac{1}{5}}$, so we can write
\begin{equation*}
    \Rncp{2} = (i)^n \sum_{2^\ell \sim t^{-1/5}} 2^\ell Q_{\sim \ell}^\tcts {T}_{m_\ell}(Q_{\sim \ell}^\tcts \fs, \fs, \overline{\fs})(n-2t) + \bsim
\end{equation*}
By using the decay properties of $\fs$, we obtain the estimate
\begin{equation}\label{eqn:rncp-1-2-decay-ests}\begin{split}
    \lVert \chi_{k}^+ (\Rncp{1} + \Rncp{2}) \rVert_{\ell^2} \lesssim& \sum_{2^\ell \gtrsim t^{-1/5}} 2^\ell \lVert \chi_{k}^\tcts T_{m_\ell}(Q_{\sim \ell}^\tcts\fs, \fs, \overline{\fs})\rVert_{L^2}\\
    \lesssim& \lVert \chi_{k}^\tcts T_{m_\ell} (\chi_{\sim k}^\tcts Q_{\sim \ell}^\tcts\fs, \chi_{\sim k}^\tcts\fs, \chi_{\sim k}^\tcts\overline{\fs})\rVert_{L^2} + \bbetter\\
    \lesssim& \begin{cases}
        \epsilon^3 t^{-\frac{37}{30}} & 2^k \ll t^{-1/5}\\
        \epsilon^3 t^{-\frac{11}{10}} & 2^k \sim t^{-1/5}\\
        \epsilon^3 t^{-4/3} 2^{-k/2} & 2^k \gg t^{-1/5}
    \end{cases}
\end{split}\end{equation}
Turning to $\Rncp{3}$, we observe that the symbol $m^{R,3}$ is supported on the region $|\xi| + |\eta| +| \sigma| \lesssim s^{-1/5}$, and decays like $|\sigma|^3$.  Thus, we can write
\begin{equation*}
    (i)^n\Rncp{3} = \sum_{2^\ell \lesssim t^{-1/5}}2^{3\ell} {T}_{m_\ell}(\fs,\fs,\overline{\fs})(n-2t)
\end{equation*}
which gives the estimate
\begin{equation}\label{eqn:rncp-3-decay-bds}\begin{split}
    \lVert \chi^+_k \Rncp{3} \rVert_{\ell^2} \lesssim& \epsilon^3 t^{-8/5} 2^{-k/2}\\
    \lVert \Rncp{3} \rVert_{\ell^2} \lesssim& \epsilon^3 t^{-\frac{19}{15}}
\end{split}\end{equation}
Combining all of these estimates, we see that
\begin{equation}\label{eqn:R-cubic-decay-bds}\begin{split}
    \lVert R_n^{\text{cubic},+} \rVert_{\ell^2} \lesssim& \epsilon^3 t^{-\frac{11}{10}}\\
    \lVert \chi_k^+ R_n^{\text{cubic},+} \rVert_{\ell^2} \lesssim& \begin{cases}
        \epsilon^3 t^{-\frac{37}{30}} & 2^k \ll t^{-1/5}\\
        \epsilon^3 t^{-\frac{11}{10}} & 2^k \sim t^{-1/5}\\
        \epsilon^3 t^{-4/3} 2^{-k/2} & 2^k \gg t^{-1/5}
    \end{cases}
\end{split}\end{equation}

\section{Bootstrap and basic estimates\label{sec:prof-red}} 

\subsection{Reduction to profile estimates}
We are now in a position to begin proving~\Cref{thm:main-thm-AL}.  Let us define 
\begin{align*}
    \begin{aligned}
    f_n(t) =& \invLinProp u_n(t)\\
    \alpha_+(t) =& \hat{f}(\fhPi,t)\\
    \alpha_-(t) =& \hat{f}(-\fhPi,t)
    \end{aligned}\qquad\qquad\qquad
    \begin{aligned}
        w_n(t) =& u_n(t) - \mSS_n(t; \alpha_+(t), \alpha_-(t))\\
        g_n(t) =& \invLinProp  w_n(t)\\ 
        h_n(t) =& \invLinProp  \mSS_n(t; \alpha_+(t), \alpha_-(t))
        \end{aligned}
\end{align*}
Then,~\Cref{thm:main-thm-AL} follows from the following theorem:
\begin{thm}\label{thm:prof-thm-A-L}
    There exists an $\epsilon_0 > 0$ such that for all $\epsilon \leq \epsilon_0$, if $\lVert \jBra{n} u^*_n \rVert_{\ell^2} \leq \epsilon$, then the solution to~\eqref{eqn:AL-eqn-rotated} exists for all $t \geq 1$ and satisfies the bounds
    \begin{align}
        \lVert n g_n(t) \rVert_{\ell^2} \lesssim \epsilon t^{\frac{1}{10}}\label{eqn:ngn-prof-est}\\ 
        \lVert \hat{f}(t) \rVert_{L^\infty} \lesssim \epsilon\label{eqn:f-hat-prof-est}
    \end{align}
    Moreover,
    \begin{equation}\label{eqn:alpha-pm-prof-est}
        |\partial_t \alpha_\pm(t)| \lesssim \epsilon^3 t^{-\frac{16}{15}}
    \end{equation}
    and there exists a bounded function $f_\infty$ such that
    \begin{equation}\label{eqn:f-hat-prof-asymp}
        \hat{f}(\xi,t) = \exp\left(\frac{i\sgn(\cos(\xi)}{4} \int_1^t \frac{|\hat{f}(\xi,s)|^2}{s}\;ds\right) f_\infty(\xi) + \bigO{\epsilon^3 |t^{\frac{1}{3}}\cos\xi|^{-\stPhErr}}
    \end{equation}
\end{thm}

We now show how~\Cref{thm:main-thm-AL} follows from~\Cref{thm:prof-thm-A-L}:
\begin{proof}[Proof of~\Cref{thm:main-thm-AL} assuming~\Cref{thm:prof-thm-A-L}]
    Note that~\eqref{eqn:f-hat-prof-est} implies that $|\alpha_\pm(t)| \lesssim \epsilon$.  Thus, the estimates from~\Cref{lem:LSS-expression} combined with~\eqref{eqn:ngn-prof-est} imply that $\lVert n f_n \rVert_{\ell^2} \lesssim \epsilon t^{1/6}$.  It follows that $u_n(t)$ obeys the linear estimates given in~\Cref{thm:lin-decay-al}.  This immediately verifies the decay estimate~\eqref{eqn:main-thm-al-decay-rgn}.  The estimate~\eqref{eqn:main-thm-al-osc-rgn} also follows once we use~\eqref{eqn:f-hat-prof-asymp}.  Thus, it only remains to verify the estimate~\eqref{eqn:main-thm-al-ss-rgn}.  Suppose that $n \mp 2t \lesssim t^{-1/3}$.  Then, we can write
    \begin{equation*}\begin{split}
        |u_n(t) - \fs(n \mp 2t, t; \alpha_{\pm})| \leq& |u_n(t) - Q_{\lesssim \log t^{-1/5}}^\pm u_n(t)|\\ 
        &+ |\fs(n - 2t, t; \alpha_\pm(t)) - Q_{\lesssim \log t^{-1/5}} \mSS_n(t;\alpha_{+}(t),\alpha_{-}(t))|\\ 
        &+ |Q_{\lesssim \log t^{-1/5}} w| + |\fs(n - 2t, t; \alpha_{\pm}(t)) - \fs(n - 2t, t; \alpha_{\pm,\infty})|
    \end{split}\end{equation*}
    where $\alpha_{\pm,\infty} = \lim_{t \to \infty} \alpha_\pm(t)$ (the limit must exist by~\eqref{eqn:alpha-pm-prof-est}).   From the linear estimates for $u_n$, we see that
    \begin{equation*}
        \lVert u_n(t) - Q_{\lesssim \log t^{-1/5}}^\pm u_n(t)\rVert_{\ell^\infty} \lesssim \epsilon t^{-2/5}
    \end{equation*}
    Similar reasoning allows us to conclude that 
    \begin{equation*}
        \lVert \fs(n - 2t, t; \alpha_{\pm}(t)) - Q_{\lesssim \log t^{-1/5}}^\pm \mSS_n(t;\alpha_+(t),\alpha_-(t))\rVert_{\ell^\infty} \lesssim \epsilon t^{-2/5}
    \end{equation*}
    once we observe that $Q_{\lesssim \log t^{-1/5}}^\pm \mSS_n(t;\alpha_+(t),\alpha_-(t)) = Q_{\lesssim \log t^{-1/5}}^\tcts \fs(n \mp 2t,t; \alpha_\pm)$ and use the decay bounds given by~\eqref{eqn:fs-decay-bdds}.  For the $w_n$ term, we observe that $|\hat{g}(\xi,t)| \lesssim \lVert ng_n \rVert_{\ell^2} |\xi|^{1/2}$, so~\Cref{thm:lin-decay-al} gives us that
    \begin{equation*}\begin{split}
        \lVert Q_{\lesssim \log t^{-\frac{1}{5}}}^\pm w_n \rVert_{\ell^\infty} \lesssim& \lVert n g_n \rVert_{\ell^2} t^{-1/2}\\
        \lesssim& \epsilon t^{-2/5}
    \end{split}\end{equation*}
    Finally,~\eqref{eqn:alpha-pm-prof-est} implies that $|\alpha_\pm(t) - \alpha_{\pm,\infty}| \lesssim \epsilon^3 t^{-\frac{1}{15}}$, so we find that
    \begin{equation*}\begin{split}
        |\fs(n - 2t, t; \alpha_{\pm}(t)) - \fs(n - 2t, t; \alpha_{\pm,\infty})| \lesssim& \sum_{t^{-1/3} \leq 2^j \lesssim t^{-1/5}} \lVert D_\alpha \fs \rVert_{\ell^\infty} |\alpha_\pm(t) - \alpha_{\pm,\infty}|\\
        \lesssim& \epsilon t^{-2/5}
    \end{split}
    \end{equation*}
    by~\eqref{eqn:D-alpha-fs-bdds}.
\end{proof}

\subsection{The bootstrap argument}

It only remains to prove~\Cref{thm:prof-thm-A-L}.  We now show that~\Cref{thm:prof-thm-A-L} follows from a certain bootstrap estimate.  The remainder of the paper is then devoted to verifying that this bootstrap estimate holds.

We first show that we can construct a local solution to~\eqref{eqn:AL-eqn-rotated} with $\lVert \hat{f} \rVert_{L^\infty}$ and $\lVert n f_n \rVert_{\ell^2}$ finite.  Note that by assumption $u_n^* \in \ell^2$.  Thus, it is straightforward to use Picard iteration to construct a local solution $u_n(t) \in C([1,T],\ell^2)$ for some $T > 1$.  (In fact, we can take this solution to be global using a conservation law argument, but this is not important for us here.)  Moreover, using the embedding $\ell^2 \hookrightarrow \ell^\infty$, we see that
\begin{equation*}
    \lVert n |u_n|^2 \Psi u_n \rVert_{\ell^2}  \lesssim (1 + \lVert u_n(t) \rVert_{\ell^2}^2) \lVert n u_n \rVert_{\ell^2}
\end{equation*}
Thus, we can upgrade this solution to a solution in $C([1,T], \jBra{n}^{-1} \ell^2)$.  By using the commutator identity $[n, e^{it\Psi}] = 2it \Psi'e^{it\Psi}$ where $\Psi' a_n= a_{n-1} - a_{n+1}$, we see that $n f_n \in C([1,T];\ell^2)$, as well.  The Sobolev embedding then implies that $\hat{f} \in C([1,T], L^\infty)$.   We also note that for this solution, 
\begin{equation*}
    |\partial_t \alpha^{\pm}| \leq \lVert |u_n|^2 \Psi u_n \rVert_{\ell^1} \lesssim \lVert u_n \rVert_{\ell^2}^3 \lesssim \lVert \hat{f} \rVert_{L^\infty}^3
\end{equation*}
which can be used to show that $\partial_t \alpha_\pm \in C([1,T])$.

In particular, all of the above quantities are continuous in time, so by taking $T$ sufficiently close to $1$, we can guarantee that
\begin{equation}\label{eqn:bootstrap-hyps-al}\begin{split}
    \sup_{1 \leq t \leq T} \left( \lVert \hat{f} \rVert_{L^\infty} + t^{-1/6} \lVert n f_n \rVert_{\ell^2} \right) \leq& CM\epsilon\\
    \sup_{1 \leq t \leq T} t^{\frac{16}{15}}|\alpha_{\pm}(t)| \leq& CM^3 \epsilon^3 
\end{split}
\end{equation}
where $C$ is some constant and $M \gg 1$ is a large fixed constant.  Intuitively, $C$ represents the constants that occur in the strong bounds~\eqref{eqn:bootstrap-hyps-al-improved}, while $M$ has the effect of producing weaker bounds which we use to start the bootstrap.  Assuming that $\epsilon \ll M^{-2}$, we will show that $\lVert n g_n \rVert_{\ell^2} \leq C \epsilon t^{\frac{1}{10}}$ for $1 \leq t \leq T$ in~\Cref{sec:n-ell2}.  Using this bound together with the bootstrap hypotheses, we will then show in~\Cref{sec:degen-freq-decay} that
\begin{equation*}
    |\partial_t \alpha_\pm(t)| \leq CM^2 \epsilon^3 t^{-\frac{16}{15}}
\end{equation*}
for $1 \leq t \leq T$ and that $\sup_{1 \leq t \leq T}\lVert \hat{f}(\xi,t) \rVert_{L^\infty_\xi} \leq C\epsilon$ in~\Cref{sec:L-infty}.  Combining these estimates, we can see that we have the improved bounds
\begin{equation}\label{eqn:bootstrap-hyps-al-improved}\begin{split}
    \sup_{1 \leq t \leq T} \left( \lVert \hat{f} \rVert_{L^\infty} + t^{-1/6} \lVert n f_n \rVert_{\ell^2} \right) \leq& 2C\epsilon\\
    \sup_{1 \leq t \leq T} t^{\frac{16}{15}} |\alpha_{\pm}(t)| \leq& CM^2 \epsilon^3 
\end{split}
\end{equation}
It follows that the bootstrap estimates~\eqref{eqn:bootstrap-hyps-al} actually hold up to a time $T' > T$.  Since $T$ was arbitrary, a continuity argument then shows that~\eqref{eqn:bootstrap-hyps-al-improved} holds for all time, so the solution is global and satisfies~\cref{eqn:ngn-prof-est,eqn:f-hat-prof-est,eqn:alpha-pm-prof-est}.  Moreover, as byproduct of our proof that $\hat{f}(\xi,t)$ is bounded, we will prove~\eqref{eqn:f-hat-prof-asymp}, completing the proof of~\Cref{thm:prof-thm-A-L}.

\subsection{Basic estimates for \texorpdfstring{$w_n$}{w\_n}}

We conclude this section by recording estimates for $w_n$ that we will use extensively in the remainder of the paper.  To begin, we observe that 
\begin{equation*}
    \lVert \hat{g} \rVert_{L^\infty} \leq \lVert \hat{f} \rVert_{L^\infty} + \lVert \hat{h} \rVert_{L^\infty} \lesssim M\epsilon
\end{equation*}
In addition, by the Sobolev embedding, $\hat{g}$ is $\frac{1}{2}$-H\"older in $\xi$, and by construction $\hat{g}(0) = 0$, so
\begin{equation}\label{eqn:hat-g-Holder}
    |\hat{g}(\xi)| \lesssim |\xi|^{\frac{1}{2}} \lVert n g_n \rVert_{\ell^2}
\end{equation}
Thus,~\Cref{thm:lin-decay-al} gives us the following $\ell^\infty$ decay estimates for $w_n$:
\begin{equation}\label{eqn:w-ellinf-al}\begin{split}
    \lVert w_n \rVert_{\ell^\infty} \lesssim& t^{-1/2} \lVert n g_n \rVert_{\ell^2}\\
    \lVert \chi^\pm_k w_n \rVert_{\ell^\infty} \lesssim& t^{-1/2} \lVert n g_n \rVert_{\ell^2} c_k\\ 
    \lVert \chi^\pm_k \Psi w_n \rVert_{\ell^\infty} \lesssim& t^{-1/2} 2^k \lVert n g_n \rVert_{\ell^2} c_k\\ 
    \lVert Q^\pm_j \Psi w_n \rVert_{\ell^\infty} \lesssim& t^{-1/2} 2^j \lVert ng_n \rVert_{\ell^2} c_j\\ 
    \lVert \chi_k^\pm Q_{\leq j}^\pm w_n \rVert_{\ell^\infty} \lesssim& t^{-5/6} 2^{j - 2k} \lVert ng_n \rVert_{\ell^2}\qquad\qquad  \text{for } k > j + 10\\
    \lVert \chi_{< j - 10}^\pm Q_{j}^\pm w_n \rVert_{\ell^\infty} \lesssim& t^{-5/6} 2^{-j} \lVert ng_n \rVert_{\ell^2}\\
    \lVert \chi_k^\pm Q_{j}^\pm w_n \rVert_{\ell^\infty} \lesssim& t^{-5/6} 2^{j - 2k} \lVert ng_n \rVert_{\ell^2}\qquad\qquad  \text{for } k > j + 10\\
    \lVert \chi_k^\pm Q_{\leq j}^\pm \Psi w_n \rVert_{\ell^\infty} \lesssim& t^{-5/6} 2^{2j - 2k} \lVert ng_n \rVert_{\ell^2}c_k \qquad\qquad \text{for } k > j + 10
\end{split}\end{equation}
By using these estimates or~\eqref{eqn:hat-g-Holder}, we can derive the following $\ell^2$ estimates
\begin{equation}\label{eqn:w-ell2-al}\begin{split}
    \lVert Q_{\leq j}^\pm w_n \rVert_{\ell^2} \lesssim& 2^j \lVert ng_n \rVert_{\ell^2}\\ 
    \lVert Q^\pm_j w_n \rVert_{\ell^2} \lesssim& 2^j \lVert ng_n \rVert_{\ell^2} c_j\\ 
    \lVert \chi^\pm_k w_n \rVert_{\ell^2} \lesssim& 2^k \lVert ng_n \rVert_{\ell^2} c_k\\
    \lVert \chi_k^\pm Q_{\leq j}^\pm w_n \rVert_{\ell^2} \lesssim& t^{-1/3} 2^{j - k} \lVert ng_n \rVert_{\ell^2} \qquad\qquad k > j + 10\\ 
    \lVert \chi_{>j + 10}^\pm Q_j^\pm w_n \rVert_{\ell^2} \lesssim& t^{-1/3} \lVert ng_n \rVert_{\ell^2}\\ 
    \lVert (1 - \chi^\pm_{\sim j}) Q^\pm_j w_n \rVert_{\ell^2} \lesssim& t^{-1/3} \lVert ng_n \rVert_{\ell^2} c_j
\end{split}
\end{equation}
			
\section{The weighted \texorpdfstring{$\ell^2$}{l2} estimate\label{sec:n-ell2}}

\subsection{Set-up and basic estimates}

We now turn to the task of obtaining weighted $\ell^2$ estimates for $g_n$.  In order to close the bootstrap, we will prove an energy estimate of the form
\begin{equation}\label{eqn:n-gn-desired-est}\begin{split}
    \lVert n g_n (t) \rVert_{\ell^2}^2 \lesssim M^2\epsilon^4t^{\frac{1}{5}} + M^2\epsilon^2 \lVert ng_n \rVert_{\ell^2}^2 + \int_1^t M^2 \epsilon^2 s^{-1} \lVert ng_n (s)  \rVert_{\ell^2}^2 \;ds
\end{split}\end{equation}
which gives the desired estimate for $\epsilon \ll M^{-2}$ after applying Gr\"onwall's inequality.  Observe that
\begin{equation}\label{eqn:n-gn-deriv}\begin{split}
    \partial_t ng =& n \invLinProp (\partial_t - i\discrAvg) (u_n - \mSS_n) \\
    =& -i\invLinProp(|u_n|^2 \Psi' w_n + 2\Re(w_n \overline{u_n})\Psi' \mSS_n\\
    &-i \invLinProp \left(|u_n|^2 \Psi Lw_{n} + 2\Re(w_n \overline{u_n})\Psi L\mSS_{n}\right) \\
    &-i T_{\partial_\xi e^{it\phi}}(f_n, \Psi g_n, \overline{f}_n) + T_{\partial_\xi e^{it\phi}}(f_n, \Psi h_{n}, \overline{g}_n) + T_{\partial_\xi e^{it\phi}}(g_n, \Psi h_{n}, \overline{f}_n) \\
    &-i n\invLinProp R_n
\end{split}\end{equation}
where $\Psi' v_n = v_{n-1} - v_{n+1}$, $L$ is the operator defined in~\eqref{eqn:L-def}, and $R_n$ is the error term introduced in~\eqref{eqn:R-n-def}.  Using~\eqref{eqn:n-gn-deriv}, we find that
\begin{subequations}\begin{align}
    \frac{1}{2} \frac{d}{dt} \lVert ng_n \rVert_{\ell^2}^2 =& \Im \left \langle |u_n|^2 \Psi' w_{n} + 2\Re(w_n \overline{u_n})\Psi'\mSS_{n}, Lw_n \right\rangle \label{eqn:ng-deriv-commutator}\\
    & + \Im\left \langle |u_n|^2 \discrAvg Lw_{n}, Lw_n\right\rangle \label{eqn:ng-deriv-weighted}\\
    &+2\Im \left\langle\Re(w_n \overline{u_n}) \discrAvg L\mSS_{n}, Lw_n\right\rangle \label{eqn:ng-deriv-psi-L-SS}\\
    \begin{split}
    &+t\Re \left\langle T_{\partial_\xi\phi e^{it\phi}}(f_n, \discrAvg g_{n}, \overline{f}_n) + T_{\partial_\xi e^{it\phi}}(f_n, \discrAvg h_{n}, \overline{g}_n) + T_{\partial_\xi e^{it\phi}}(g_n, \discrAvg h_{n}, \overline{f}_n), ng_n \right\rangle \label{eqn:ng-deriv-str-terms}
    \end{split}\\
    &+ \Im \left\langle L R_n, Lw_n \right\rangle \label{eqn:ng-deriv-eqn-error}
\end{align}\end{subequations}

We will now estimate terms~\Cref{eqn:ng-deriv-commutator,eqn:ng-deriv-weighted,eqn:ng-deriv-psi-L-SS}.  The term~\eqref{eqn:ng-deriv-str-terms} will be estimated in~\cref{sec:weighted-l2-str} and~\eqref{eqn:ng-deriv-eqn-error} will be estimated in~\cref{sec:weighted-l2-eqn-err}.  For~\eqref{eqn:ng-deriv-commutator}, we can use the linear estimates for $u_n$ and the estimates~\eqref{eqn:mssn-lin-ests} for $\mSS_n$ together with the  $\ell^2$ estimates for $w_n$ given in~\eqref{eqn:w-ell2-al} to find that
\begin{equation*}\begin{split}
    |\eqref{eqn:ng-deriv-commutator}| \lesssim& \left(\lVert |u_n|^2 \Psi'w_{n} \rVert_{\ell^2} + \lVert u_n \overline{w_n} \Psi'\mSS_{n}\rVert_{\ell^2}\right) \lVert Lw_n \rVert_{\ell^2}\\
    \lesssim& \lVert ng_n \rVert_{\ell^2}\left( \left(\sum_{k} \lVert \chi_{k}^+ |u_n|^2 \Psi'w_{n} \rVert_{\ell^2}^2\right)^{1/2} + \left( \sum_{k} \lVert \chi_{k}^+ u_n \overline{w_n} \Psi'\mSS_{n}\rVert_{\ell^2}^2\right)^{1/2}\right) + \bsim\\
    \lesssim& M^2 \epsilon^2 t^{-1}\lVert ng_n \rVert_{\ell^2}^2
\end{split}\end{equation*}
which gives the required estimate after integrating in time.  For the term~\eqref{eqn:ng-deriv-weighted}, we must take advantage of cancellations.  Using the Fourier transform and defining $q =  \mathcal{F}(|u_n|^2)$, $r = \mathcal{F}(Lw)$, we find that
\begin{equation*}\begin{split}
    \eqref{eqn:ng-deriv-weighted} =& \frac{1}{2}\sum_{n \in \bbZ} |u_n|^2 \left( \discrAvg Lw_n \overline{Lw}_n - Lw_n \discrAvg \overline{Lw}_n\right) \\
    =& \frac{1}{4\pi} \int_{\bbT^2} q(\eta) \left(\cos(\xi - \eta) - \cos(\xi)\right) r(\xi - \eta) \overline{r}(\xi) \;d\xi d\eta\\
    =& \frac{1}{2\pi}\int_{\bbT^2} \sin \frac{\eta}{2} q(\eta) \left(\sin \frac{\xi}{2} \cos \frac{\xi - \eta}{2} + \cos\frac{\xi}{2} \sin \frac{\xi - \eta}{2}\right) r(\xi - \eta) \overline{r}(\xi) \;d\xi d\eta\\
    =& \sum_{n \in \bbZ} \sin(D/2) |u_n|^2 2\Re \left(\sin(D/2)Lw_n \cos(D/2) \overline{Lw}_n\right)\\
    \lesssim& M^2 \epsilon^2 t^{-1}\lVert ng_n \rVert_{\ell^2}^2
\end{split}\end{equation*}
where on the last line we have used the enhanced bilinear decay from~\Cref{cor:bilin-sin-est}.  Finally, to estimate~\eqref{eqn:ng-deriv-psi-L-SS}, we use~\cref{lem:LSS-expression} to write $L \mSS_n = \fL_n^{+} + \fL_n^{-} + \fC_n^+ + \fC_n^-$.  We will show how to bound the terms with `$+$' apexes: the other terms are similar.  Using~\eqref{eqn:fL-ests} together with~\eqref{eqn:w-ell2-al}, we find
\begin{equation*}\begin{split}
    \lVert \chi_{\lesssim \log t^{-1/5}}^\pm w_n \overline{u_n} \Psi \fL_n^+ \rVert_{\ell^2}\lesssim&
    \sum_{t^{-1/3} \lesssim 2^k \lesssim t^{-1/5}} \lVert \chi_{k}^\pm w_n \rVert_{\ell^2} \lVert \chi_{\sim k}^\pm u_n\rVert_{\ell^\infty} \lVert\Psi \fL^{+}_n\rVert_{\ell^\infty}\\
    \lesssim& M \epsilon^2 t^{-1}\lVert ng_n \rVert_{\ell^2}    
\end{split}\end{equation*}
and
\begin{equation*}\begin{split}
    \lVert \chi_{\gg \log t^{-\frac{1}{5}}}^+\chi_{\gg \log t^{-\frac{1}{5}}}^- w_n \overline{u_n} \Psi \fL_n^+ \rVert_{\ell^2}\lesssim& \lVert w_n \rVert_{\ell^\infty} \lVert \chi_{\gg \log t^{-\frac{1}{5}}}^+\chi_{\gg \log t^\frac{1}{5}}^- u_n \rVert_{\ell^\infty} \lVert \chi_{\gg \log t^{-\frac{1}{5}}}^+ \Psi \fL^+_n \rVert_{\ell^2}\\
    \lesssim& M \epsilon^2 t^{-\frac{17}{15}}\lVert ng_n \rVert_{\ell^2}
\end{split}\end{equation*}
which gives us the bound for these terms.  For the term involving $\fC_n^+$, we use~\eqref{eqn:fC-ests} to derive the bound
\begin{equation*}\begin{split}
    \lVert w_n \overline{u_n} \Psi \fC_n^+ \rVert_{\ell^2} \lesssim& \left( \sum_{k} \lVert \chi_{k}^+   w_n \overline{u_n} \Psi \fC_n^+ \rVert_{\ell^2}^2\right)^{1/2}\\
    \lesssim& M \epsilon^4 t^{-1}\lVert ng_n \rVert_{\ell^2}
\end{split}\end{equation*}
Combining the estimates for the $\fL^\pm_n$ and $\fC^\pm_n$ terms, we find that
\begin{equation*}
    |\eqref{eqn:ng-deriv-psi-L-SS}| \lesssim M \epsilon^2(1 + \epsilon^2) t^{-1}\lVert ng_n \rVert_{\ell^2}^2
\end{equation*}
which is better than what we require by a factor of $M$.

\subsection{\label{sec:weighted-l2-str}Estimates for the space-time resonant terms}
\subsubsection{Division in frequency space} 

To control~\eqref{eqn:ng-deriv-str-terms}, we must take advantage of the resonance structure of $\phi$ to gain back an extra power of $t$.  Recalling the definition of space-time resonance sets given in~\eqref{eqn:space-time-resonance-sets}, we define $\cR_1$ to be the set of all points within distance $1$ of the space-time resonance set $\cR$, and write
\begin{equation*}\begin{split}
    \cR_1 =& \{|\xi - \fhPi|^2 + |\eta - \fhPi|^2 +|\sigma + \fhPi|^2 \leq 1\} \cup \{|\xi + \fhPi|^2 + |\eta + \fhPi|^2 +|\sigma - \fhPi|^2 \leq 1\}\\
    =:& \cR_1^+ \cup \cR_1^-
\end{split}\end{equation*}
Let $\chi_\mathcal{S}(\xi,\eta,\sigma;t)$, $\chi_\mathcal{T}(\xi,\eta,\sigma;t)$, and $\chi_\mathcal{R}(\xi,\eta,\sigma;t)$ form a smooth partition of unity on $\bbT^3$ such that for $t \geq 1$:
\begin{enumerate}[(i)]
    \item $\left.\chi_\cS\right|_{\bbT^3 \setminus \cR_1}$ and $\left.\chi_\cT\right|_{\bbT^3 \setminus \cR_1}$ are constant and $\left.\chi_\cR\right|_{\bbT^3 \setminus \cR_1} = 0$
    \item $\left.\chi_\cS\right|_{\bbT^3 \setminus \cR_1}$ is supported outside a neighborhood of $\cS$ and $\left.\chi_\cT\right|_{\bbT^3 \setminus \cR_1}$ is supported outside a neighborhood of $\cT$
    \item Within $\cR_1^{\pm}$, we have that for $\bullet = \cS, \cT, \cR$ 
    \begin{equation*}
        \chi_\bullet(\xi,\eta,\sigma; t) = \tilde{\chi}_\bullet(t^{1/3}(\xi \mp \fhPi),t^{1/3}(\eta \mp \fhPi), t^{1/3}(\sigma \pm \fhPi))
    \end{equation*}
    where $\tilde{\chi}_\cS, \tilde{\chi}_\cT, \tilde{\chi}_\cR$ form a partition of unity on $\bbR$ such that $\tilde{\chi}_\cS$ and $\tilde{\chi}_\cT$ are zero homogeneous outside a ball of radius $1/2$, and vanish within a ball of radius $1/4$, and $\tilde\chi_{\cT}$ vanishes in a neighborhood of $\tilde{\cT} = \{\xi = \eta\}\cup\{\xi = \sigma\} \subset \bbR^3$, and $\tilde\chi_\cS$ vanishes in a neighborhood of $\tilde{\cS} := \{\eta = \sigma = \xi/3\} \subset \bbR^3$.
\end{enumerate}
Using this partition of unity, we write
\begin{subequations}\begin{align}
        \eqref{eqn:ng-deriv-str-terms} =& t\Re \bigl\langle T_{\partial_\xi \phi \chi_\cR e^{it\phi}}(f_n, \discrAvg g_{n}, \overline{f}_n) + T_{\partial_\xi \phi \chi_\cR e^{it\phi}}(f_n, \discrAvg h_{n}, \overline{g}_n) + T_{\partial_\xi \phi \chi_\cR e^{it\phi}}(g_n, \discrAvg h_{n}, \overline{f}_n), ng_n \bigr\rangle
    \label{eqn:space-time-res-term}\\
    \begin{split}
        &\cramped{+ t\Re \bigl\langle T_{\partial_\xi \phi \chi_\cS e^{it\phi}}(f_n, \discrAvg g_{n}, \overline{f}_n) + T_{\partial_\xi \phi \chi_\cS e^{it\phi}}(f_n, \discrAvg h_{n}, \overline{g}_n) + T_{\partial_\xi \phi \chi_\cS e^{it\phi}}(g_n, \discrAvg h_{n}, \overline{f}_n), ng_n \bigr\rangle}
   \end{split}\label{eqn:space-non-res-term}\\
   \begin{split}
        &\cramped{+t\Re \bigl\langle T_{\partial_\xi \phi \chi_\cT e^{it\phi}}(f_n, \discrAvg g_{n}, \overline{f}_n) + T_{\partial_\xi \phi \chi_\cT e^{it\phi}}(f_n, \discrAvg h_{n}, \overline{g}_n) +
        T_{\partial_\xi\phi \chi_\cT e^{it\phi}}(g_n, \discrAvg h_{n}, \overline{f}_n), ng_n \bigr\rangle}
        \end{split}\label{eqn:time-non-res-term}
\end{align}\end{subequations}

\subsubsection{Space-time resonant estimates}

For the term~\eqref{eqn:space-time-res-term}, we have that
\begin{equation*}\begin{split}
    |\eqref{eqn:space-time-res-term}| \lesssim& t \lVert n g_n \rVert_{\ell^2} \left( \lVert T_{\partial_\xi \phi \chi_\cR e^{it\phi}}(f_n, \discrAvg g_{n}, \overline{f}_n)\rVert_{\ell^2} + \lVert T_{\partial_\xi \phi \chi_\cR e^{it\phi}}(f_n, \discrAvg h_{n}, \overline{g}_n)\rVert_{\ell^2}\right. \\
    &\qquad\qquad\qquad\left.+ \lVert T_{\partial_\xi \phi \chi_\cR e^{it\phi}}(g_n, \discrAvg h_{n}, \overline{f}_n)\rVert_{\ell^2}\right)
\end{split}\end{equation*}
We will show how to estimate $\lVert T_{\partial_\xi \phi \chi_\cR e^{it \phi}}(f_n, \Psi g_n, \overline{f}_n)\rVert_{\ell^2}$.  The estimates for the other terms are analogous.  Defining $m_\cR^{\pm}(\xi,\eta,\sigma;t) = \partial_\xi \phi \chi_\cR e^{it \phi} \cos(\xi -\eta - \sigma)$, Plancherel's theorem and the support condition for $\chi_\cR$ yield
\begin{equation*}\begin{split}
    \lVert T_{\partial_\xi \phi \chi_\cR e^{it \phi}}(f_n, \Psi g_n, \overline{f}_n) \rVert_{\ell^2} =& \left\lVert\int_{\bbT^2} m_\cR \hat{f}(\eta) \hat{g}(\xi - \eta - \sigma) \overline{\hat{f}}(-\sigma) \;d\eta d\sigma \right\rVert_{L^2_\xi}\\
    \lesssim&  \left\lVert m_\cR(\xi,\eta,\sigma;t) \hat{f}(\eta) \hat{g}(\xi - \eta - \sigma) \overline{\hat{f}}(-\sigma) \right\rVert_{L^2_\xi L^1_{\eta,\sigma}} + \bsim\\
    \lesssim& \left\lVert |\xi - \eta - \sigma - \fhPi|^{1/2} m_\cR  \right\rVert_{L^2_\xi L^1_{\eta,\sigma}} \lVert \hat{f} \rVert_{L^\infty}^2 \lVert \hat{g} \rVert_{C^{1/2}}\\
    \lesssim& M^2 \epsilon^2 t^{-2}\lVert ng_n \rVert_{\ell^2}
\end{split}\end{equation*}
It follows that
\begin{equation*}
    |\eqref{eqn:space-time-res-term}| \lesssim M^2 \epsilon^2 t^{-1} \lVert ng_n \rVert_{\ell^2}^2
\end{equation*}
which is consistent with~\eqref{eqn:n-gn-desired-est}.

\subsubsection{Space non-resonant estimates}

We now show how to control the terms~\eqref{eqn:space-non-res-term}.  We can write
\begin{equation*}
    \eqref{eqn:space-non-res-term} = \Re \left\langle \mathcal{N}^\mathcal{S}(f,g,h), ng_n \right\rangle
\end{equation*}
where
\begin{equation*}
     \mathcal{N}^\mathcal{S}(f,g,h) = t T_{\partial_\xi \phi \chi_\cS e^{it\phi}}(f_n, \discrAvg g_{n}, \overline{f}_n) + T_{\partial_\xi \phi \chi_\cS e^{it\phi}}(f_n, \discrAvg h_{n}, \overline{g}_n) + T_{\partial_\xi \phi \chi_\cS e^{it\phi}}(g_n, \discrAvg h_{n}, \overline{f}_n)
\end{equation*}
By Cauchy-Schwarz, the desired bound will follow if we can show that
\begin{equation*}
    \lVert \mathcal{N}^\mathcal{S}(f,g,h) \rVert_{\ell^2} \lesssim M^2\epsilon^2(1 + \epsilon^2) t^{-1} \lVert ng_n \rVert_{\ell^2}
\end{equation*}
To obtain this estimate, we integrate by parts in frequency using the identity 
\begin{equation*}
    e^{it\phi} = \frac{\deles \phi}{it|\deles \phi|^2} \cdot \deles e^{it\phi}
\end{equation*}
Defining 
\begin{equation*}
    \mu^\mathcal{S} = \partial_\xi \phi \frac{\deles \phi}{i|\deles \phi|^2} \chi_\cS e^{it\phi}
\end{equation*}
to be the vector-valued symbol with components $\mu^{\mathcal{S},\eta}$ and $\mu^{\mathcal{S},\sigma}$, we find that
\begin{equation}\label{eqn:m-S-first-div}\begin{split}
    \mathcal{N}^\mathcal{S}(f,g,h) =& - T_{\deles \cdot \mu^\cS \cos(\xi - \eta - \sigma)}(u_n, w_n, \overline{u}_n)\\
    &- T_{\mu^{\cS,\sigma}}(u_n, \discrAvg w_n, \overline{L\mSS}_n)\\
    &- T_{\mu^{\cS,\sigma}}(u_n, \discrAvg L w_n, \overline{u}_n)\\
    &- T_{\mu^{\cS,\sigma}}(u_n, \discrAvg \mSS_n, \overline{Lw}_n)\\
    &+ \bsime
\end{split}\end{equation}
Where we have used the fact that $u_n = w_n + \mSS_n$ to express$\mathcal{N}^\cS(f,g,h)$ in terms of pseudoproducts which contain exactly one $w_n$ or $Lw_n$ factor and do not include any factors of $Lu_n$. Dividing dyadically in frequency, we can write
\begin{subequations}\begin{align}
    \mathcal{N}^\cS(f,g,h) =& \sum_{2^j \gtrsim t^{-1/3}} T_{m^\cS_j}(u_n, w_n \overline{u}_n)\label{eqn:m-S-j-no-Psi}\\
    &\qquad + T_{m^\cS_j}(u_n, \discrAvg L w_n,  \overline{u}_n)\label{eqn:m-S-j-Psi-L-w}\\
    &\qquad + T_{m^\cS_j}(u_n, \discrAvg u_n,  \overline{Lw}_n)\label{eqn:m-S-j-Psi-SS-L-w}\\
    &\qquad + T_{m^\cS_j}(u_n, \discrAvg L \mSS_n,  \overline{w}_n)\label{eqn:m-S-j-Psi-L-SS}\\
    &\qquad + T_{m^\cS_j}(u_n, \discrAvg w_n,  \overline{L\mSS}_n)\label{eqn:m-S-j-Psi-w-L-SS}\\
    &\qquad + T_{m^\cS_j}(L\mSS_n, \discrAvg u_n,  \overline{w}_n)\label{eqn:m-S-j-Psi-u-L-SS}\\
    &\qquad + T_{m^\cS_j}(u_n,  \mSS_n,  \overline{w}_n)\label{eqn:m-S-j-u-SS-w}\\
    &+\bsime \nonumber
\end{align}\end{subequations}
where the $m^\cS_j$ represents a generic symbol (the value of which can change from line to line), which satisfies the support condition
\begin{equation}\label{eqn:m-S-j-supp-cond}
    \supp(m^\cS_j) \subset \left\{|\xi - \eta| \ll 2^j \text{ or }|\xi - \sigma + \pi| \ll 2^j\right\} \cap \{\cos^2 \xi + \cos^2 \eta + \cos^2 \sigma \sim 2^{2j}\}
\end{equation}
and the symbol bound
\begin{equation}\label{eqn:m-S-j-symbol-bound}
    |\partial_{\xi,\eta,\sigma}^\alpha m^\cS_j| \lesssim 2^{-|\alpha|j}
\end{equation}
In particular, of the support of $m_j^\cS$, $|\xi - \eta - \sigma \pm \fhPi| \lesssim \max \{|\eta \pm \fhPi|, |\sigma \pm \fhPi|\}$, so we can write
\begin{equation}\label{eqn:m-S-j-supp-division}\begin{split}
    T_{m^\cS_j}(p_n,q_n,\overline{r_n}) =& T_{m^\cS_j}(Q_{\ll j}^{+}p_n, Q_{\lesssim j}^{+}q_n, Q_{\sim j}^{+}\overline{r_n}) + T_{m^\cS_j}(Q_{\sim j}^{+}p_n, Q_{\lesssim j}^{+}q_n, Q_{\ll j}^{+}\overline{r_n)} \\
    &+ Q_{\sim j}^{+} T_{m^\cS_j}(Q_{\sim j}^{+}p_n, Q_{\sim j}^{+}q_n, Q_{\sim j}^{+}\overline{r_n})+ \bsime
\end{split}\end{equation}

We begin by considering the terms~\eqref{eqn:m-S-j-no-Psi} and~\eqref{eqn:m-S-j-u-SS-w}.  By using the support condition~\eqref{eqn:m-S-j-supp-cond} and noting that $\mSS_n$ and $u_n$ satisfy the same sorts of decay estimates, we find that
\begin{subequations}\begin{align}
    \lVert \cramped{\eqref{eqn:m-S-j-no-Psi} + \eqref{eqn:m-S-j-u-SS-w}} \rVert_{\ell^2} \lesssim& \bigg\lVert \sum_{j} Q_{\sim j}^+ T_{m^\cS_j} (Q_{\sim j}^+ u_n, Q_{\sim j}^+ w_n, \overline{Q_{\sim j}^+ u_n}) \bigg\rVert_{\ell^2} \label{eqn:u-w-u-1}\\
    +& \bigg\lVert \sum_{j} \chi_{[j \pm 30]}^+ T_{m^\cS_j} (\chi_{[j \pm 30]}^+ Q_{\ll j}^+ u_n, \chi_{[j \pm 30]}^+ Q_{\lesssim j}^+ w_n, \chi_{[j \pm 20]}^+\overline{Q_{\sim j}^+ u_n}) \bigg\rVert_{\ell^2} \label{eqn:u-w-u-2}\\
    +& \bigg\lVert \sum_{j}  T_{m^\cS_j} ( Q_{\ll j}^+ u_n,  Q_{\lesssim j}^+ w_n, (1 - \chi_{[j \pm 20]}^+) \overline{Q_{\sim j}^+ u_n}) \bigg\rVert_{\ell^2} \label{eqn:u-w-u-3}\\
    +& \bigg\lVert \sum_{j} Q_{\sim j}^+ T_{m^\cS_j} (\chi_{> j - 40}^+ Q_{\ll j}^+ u_n, \chi_{> j - 30}^+ Q_{\ll j}^+ \mSS_n, \overline{Q_{\sim j}^+ w_n}) \bigg\rVert_{\ell^2} \label{eqn:u-w-u-4}\\
    +& \bigg\lVert \sum_{j}\sum_{k \leq j - 30} Q_{\sim j}^+ T_{m^\cS_j} (\chi_{\sim k}^+ Q_{\ll j}^+ u_n, \chi_{k}^+ Q_{\ll j}^+ \mSS_n, \chi_{\sim k}^+\overline{Q_{\sim j}^+ w_n}) \bigg\rVert_{\ell^2} \label{eqn:u-w-u-5}\\
    +& \bsime \nonumber
\end{align}
\end{subequations}
The terms~\cref{eqn:u-w-u-1,eqn:u-w-u-2,eqn:u-w-u-4} can all be controlled using the almost orthogonality of the terms in the sum to obtain
\begin{equation*}
    \eqref{eqn:u-w-u-1} + \eqref{eqn:u-w-u-2} + \eqref{eqn:u-w-u-4} \lesssim M^3 \epsilon^4 t^{-4/5} 
\end{equation*}
For~\eqref{eqn:u-w-u-3}, we use the improved decay of $Q_{\sim j}^+ u_n$ away from the support of $\chi_{[j\pm 20]}^+$ to obtain
\begin{equation*}\begin{split}
    \eqref{eqn:u-w-u-3} \lesssim&  M^2 \epsilon^2 t^{-1}\lVert ng_n \rVert_{\ell^2}
\end{split}
\end{equation*}
For the last term, 
\begin{equation*}\begin{split}
    \eqref{eqn:u-w-u-5} \lesssim& \sum_{2^j \gtrsim t^{-1/3}}\sum_{2^k \leq 2^{j - 30}} \lVert \chi_{\sim k}^+ Q_{\ll j}^+ u_n \rVert_{\ell^\infty} \lVert \chi_{\sim k}^+ Q_{\ll j}^+ \mSS_n \rVert_{\ell^\infty} \lVert \chi_{k} Q_{\sim j}^+ w_n \rVert_{\ell^2}\\
    \lesssim& M\epsilon^2  t^{-4/3}  \sum_{2^j \gtrsim t^{-1/3}} \sum_{t^{-1/3} \lesssim 2^k \lesssim 2^j}2^{-j} \lVert n g_n \rVert_{\ell^2}\\
    \lesssim& M \epsilon^2 t^{-4/3} \sum_{2^j \gtrsim t^{-1/3}} 2^{-j} \log(t^{1/3} 2^j) \lVert n g_n \rVert_{\ell^2}\\
    \lesssim& M \epsilon^2 t^{-1} \lVert ng_n \rVert_{\ell^2}
\end{split}
\end{equation*}

Turning now to the terms~\eqref{eqn:m-S-j-Psi-L-w} and~\eqref{eqn:m-S-j-Psi-SS-L-w} which containing $Lw_n$ factors, we can write
\begin{subequations}\begin{align}
    \lVert \eqref{eqn:m-S-j-Psi-L-w} + \eqref{eqn:m-S-j-Psi-SS-L-w}\rVert_{\ell^2} \lesssim& \bigg\lVert \sum_j Q_{\sim j}^+ T_{m^\cS_j}(Q_{\sim j}^+ u_n, Q_{\sim j}^+ \Psi L w_n, \overline{Q_{\sim j}^+ u_n}) \bigg \rVert_{\ell^2}\label{eqn:Lw-terms-1}\\
    &+ \bigg\lVert \sum_j \chi_{[j\pm 30]}^+ T_{m^\cS_j}(\chi_{[j\pm 30]}^+ Q_{\ll j}^+ u_n, Q_{\sim j}^+ \Psi L w_n, \chi_{[j\pm 20]}^+\overline{Q_{\lesssim j}^+ u_n}) \bigg \rVert_{\ell^2}\label{eqn:Lw-terms-2}\\
    &+ \bigg\lVert \sum_j T_{m^\cS_j}( Q_{\ll j}^+ u_n, Q_{\lesssim j}^+ \Psi L w_n, (1 - \chi_[j\pm 20]^+)\overline{Q_{\sim j}^+ u_n}) \bigg \rVert_{\ell^2}\label{eqn:Lw-terms-3}\\
    &+ \bigg\lVert \sum_j Q_{\sim j}^+ T_{m^\cS_j}(\chi_{> j - 40}^+ Q_{\ll j}^+ u_n, \chi_{> j - 30}^+ Q_{\ll j}^+ \Psi u_n, \overline{Q_{\sim j}^+ L w_n}) \bigg \rVert_{\ell^2}\label{eqn:Lw-terms-4}\\
    &+ \bigg\lVert \sum_j \sum_{k \leq j - 30}  \chi_{\sim k} Q_{\sim j}^+ T_{m^\cS_j}(\chi_{\sim k}^+ Q_{\ll j}^+ u_n, \chi_{k}^+ Q_{\ll j}^+ \Psi u_n,\chi_{\sim k}^+ Q_{\sim j}^+\overline{ L w_n}) \bigg \rVert_{\ell^2}\label{eqn:Lw-terms-5}\\
    &+\bsime \nonumber
\end{align}
\end{subequations}
The arguments for~\cref{eqn:Lw-terms-1,eqn:Lw-terms-2,eqn:Lw-terms-3,eqn:Lw-terms-4} are similar to those for~\cref{eqn:u-w-u-1,eqn:u-w-u-2,eqn:u-w-u-3,eqn:u-w-u-4}, and for~\eqref{eqn:Lw-terms-5}, the almost orthogonality of $\chi_{\sim k}^+ Q_{\sim j}^+\overline{L w_n}$ in $j$ and $k$ gives us the bound
\begin{equation*}\begin{split}
    \eqref{eqn:Lw-terms-5} \lesssim& M^3\epsilon^3 t^{-1}\lVert ng_n \rVert_{\ell^2}
\end{split}
\end{equation*}

It only remains to control the terms~\eqref{eqn:m-S-j-Psi-L-SS},~\eqref{eqn:m-S-j-Psi-u-L-SS}, and~\eqref{eqn:m-S-j-Psi-w-L-SS} containing an $LSS_n$ factor.  Using~\Cref{lem:LSS-expression}, we can rewrite these terms as
\begin{equation*}
    \eqref{eqn:m-S-j-Psi-L-SS} + \eqref{eqn:m-S-j-Psi-u-L-SS} + \eqref{eqn:m-S-j-Psi-w-L-SS} = \rmI + \rmII +\bsime
\end{equation*}
where $\rmI$ contains all of the terms with an $\fL_n^+$ factor and $\rmII$ contains the terms with a $\fC_n^+$ factor.  Let us first consider to bound for $\rmI$.  There, we can write
\begin{subequations}\begin{align}
    \lVert \rmI \rVert_{\ell^2} \leq& \biggl\lVert \sum_j Q_{\sim j}^+ T_{m_j^\cS}(Q_{\sim j}^+ u_n, Q_{\sim j}^+ \Psi w_n, \overline{Q_{\sim j}^+ \fL_n^+}) \biggr\rVert_{\ell^2} \label{eqn:fL-term-1}\\
    &+ \biggl\lVert \sum_j T_{m_j^\cS}(Q_{\ll j}^+ u_n, Q_{\lesssim j}^+ \Psi w_n, \overline{Q_{\sim j}^+ \fL_n^+}) \biggr\rVert_{\ell^2} \label{eqn:fL-term-2}\\
    &+ \biggl\lVert \sum_j T_{m_j^\cS}(Q_{\sim j}^+ u_n, Q_{\lesssim j}^+ \Psi w_n, \overline{Q_{\ll j}^+ \fL_n^+}) \biggr\rVert_{\ell^2} \label{eqn:fL-term-3}\\
    &+ \biggl\lVert \sum_j T_{m_j^\cS}(Q_{\ll j}^+ u_n, Q_{\lesssim j}^+ \Psi \fL^+_n, \overline{Q_{\sim j}^+ w}_n) \biggr\rVert_{\ell^2} \label{eqn:fL-term-4}\\
    &+ \biggl\lVert \sum_j T_{m_j^\cS}(Q_{\sim j}^+ u_n, Q_{\lesssim j}^+ \Psi \fL^+_n, \overline{Q_{\ll j}^+ w}_n) \biggr\rVert_{\ell^2} \label{eqn:fL-term-5}\\
    &+ \biggl\lVert \sum_j T_{m_j^\cS}(Q_{\sim j}^+ \fL^+_n, Q_{\lesssim j}^+ \Psi u_n, \overline{Q_{\ll j}^+ w}_n) \biggr\rVert_{\ell^2} \label{eqn:fL-term-6}\\
    &+ \biggl\lVert \sum_j T_{m_j^\cS}(Q_{\ll j}^+ \fL^+_n, Q_{\lesssim j}^+ \Psi u_n, \overline{Q_{\sim j}^+ w}_n) \biggr\rVert_{\ell^2} \label{eqn:fL-term-7}\\
    &+ \bsime\nonumber
\end{align}
\end{subequations}
Recalling that $Q_{\sim j}^+ \fL^+_n$ vanishes for $2^j \gg t^{-1/5}$ and decays rapidly at low frequencies by~\eqref{eqn:fL-ests}, we immediately obtain the estimate
\begin{equation*}\begin{split}
    \eqref{eqn:fL-term-1} \lesssim& M\epsilon^2 \sum_{2^j \lesssim t^{-1/5}} 2^{5j} \lVert ng_n \rVert_{\ell^2}\\
    \lesssim& M^2 \epsilon^2 t^{-1}\lVert ng_n \rVert_{\ell^2}
\end{split}\end{equation*}
For the others terms containing a $Q_{\sim j}^+ \fL^+_n$ factor, we must also decompose dyadically in space.  Due to the rapid decay of $Q_{\lesssim j} u_n$ and $Q_{\lesssim j} w_n$ in the region $|x| \gg t 2^{2j}$, we see that the leading order terms are given by
\begin{equation*}\begin{split}
    \eqref{eqn:fL-term-2} + \eqref{eqn:fL-term-6} \lesssim& \biggl\lVert \sum_j \sum_{k < j + 20} T_{m_j^\cS}(\chi_{k}^+ Q_{\ll j}^+ u_n, \chi_{\sim k}^+Q_{\lesssim j}^+ \Psi w_n, \overline{Q_{\sim j}^+ \fL_n^+}) \biggr\rVert_{\ell^2}\\
    &+ \biggl\lVert \sum_j \sum_{k < j + 20} T_{m_j^\cS}(Q_{\sim j}^+ \fL_n^+ , \chi_{\sim k}^+Q_{\lesssim j}^+ \Psi  u_n, \chi_k^+ \overline{Q_{\ll j}^+ w_n}) \biggr\rVert_{\ell^2} + \bbetter\\
    \lesssim& M^2\epsilon^2 t^{-1} \lVert ng_n \rVert_{\ell^2}
\end{split}
\end{equation*}
For~\eqref{eqn:fL-term-3}, we split into the cases $2^j \lesssim t^{-1/5}$, $2^{j} \gg t^{-1/5}$ and divide dyadically in space to write
\begin{equation*}\begin{split}
    \eqref{eqn:fL-term-3} \leq& \sum_{2^j \lesssim t^{-1/5}} \lVert T_{m_j^\cS}(Q_{\sim j}^+ u_n, Q_{\lesssim j}^+ \Psi w_n, \overline{Q_{\ll j}^+ \fL_n^+}) \rVert_{\ell^2}\\
    &+ \sum_{2^j \gg t^{-1/5}} \sum_{2^k \lesssim t^{-1/5}} \lVert T_{m_j^\cS}(Q_{\sim j}^+ u_n, \chi_{\sim k}^+Q_{\lesssim j}^+ \Psi w_n, \chi_k^+\overline{Q_{\ll j}^+ \fL_n^+}) \rVert_{\ell^2}\\
    &+ \sum_{2^j \gg t^{-1/5}} \sum_{t^{-1/5} \ll 2^k \leq 2^{j+30}} \lVert T_{m_j^\cS}(Q_{\sim j}^+ u_n, \chi_{\sim k}^+Q_{\lesssim j}^+ \Psi w_n, \chi_k^+\overline{Q_{\ll j}^+ \fL_n^+}) \rVert_{\ell^2}\\
    &+ \sum_{2^j \gg t^{-1/5}} \sum_{k > j + 30} \lVert T_{m_j^\cS}(\chi_{> j + 20}^+Q_{\sim j}^+ u_n, \chi_{> j + 20}^+ Q_{\lesssim j}^+ \Psi w_n, \chi_{> j + 30}^+\overline{Q_{\ll j}^+ \fL_n^+}) \rVert_{\ell^2}\\
    \lesssim& M^2 \epsilon^2 t^{-1} \lVert ng_n \rVert_{\ell^2}
\end{split}
\end{equation*}
The argument for~\eqref{eqn:fL-term-5} is analogous, so it only remains to control the terms~\eqref{eqn:fL-term-4} and~\eqref{eqn:fL-term-7}.  We will give the argument for~\eqref{eqn:fL-term-7}; the argument for~\eqref{eqn:fL-term-4} is easier, since $\Psi$ falls on the $\fL_n^+$ factor, which has better frequency localization properties than $u_n$.  We have that
\begin{equation*}\begin{split}
    \eqref{eqn:fL-term-7} \leq& \sum_{2^j \lesssim t^{-1/5}} \lVert T_{m_j^\cS}(Q_{\ll j}^+ \fL_n^+, Q_{\lesssim j}^+ \Psi u_n, \overline{Q_{\sim j}^+ w_n}) \rVert_{\ell^2}\\
    &+ \sum_{2^j \gg t^{-1/5}} \sum_{t^{-1/3} \lesssim 2^k \lesssim t^{-1/5}} \lVert T_{m_j^\cS}(\chi_k^+ \fL_n^+, \chi_{\sim k}^+Q_{\lesssim j}^+ \Psi u_n, \chi_{\sim k}^+\overline{Q_{\sim j}^+ w_n}) \rVert_{\ell^2}\\
    &+ \sum_{2^j \gg t^{-1/5}} \sum_{2^k \gg t^{-1/5}} \lVert T_{m_j^\cS}(\chi_k^+ \fL_n^+, \chi_{\sim k}^+Q_{\lesssim j}^+ \Psi u_n, \chi_{\sim k}^+\overline{Q_{\sim j}^+ w_n}) \rVert_{\ell^2}\\
    &+ \bbetter\\
    \lesssim& M\epsilon^2 t^{-1} \lVert ng_n \rVert_{\ell^2}
\end{split}
\end{equation*}

We now turn to the term $\rmII$.  Here, we can write
\begin{subequations}\begin{align}
    \lVert \rmII \rVert_{\ell^2} \leq& \biggl \lVert \sum_{j} T_{m_j^\cS}(u_n, \Psi \fC_n^+, \overline{w_n}) \biggr\rVert_{\ell^2} \label{eqn:fC-term-1}\\
    &+ \biggl\lVert \sum_j T_{m_j^\cS}(Q_{\ll j}^+ u_n, Q_{\lesssim j}^+ \Psi w_n, \overline{Q_{\sim j}^+ \fC_n^+}) \biggr\rVert_{\ell^2} \label{eqn:fC-term-2}\\
    &+ \biggl\lVert \sum_j T_{m_j^\cS}(Q_{\sim j}^+ u_n, Q_{\lesssim j}^+ \Psi w_n, \overline{Q_{\ll j}^+ \fC_n^+}) \biggr\rVert_{\ell^2} \label{eqn:fC-term-3}\\
    &+ \biggl\lVert \sum_j T_{m_j^\cS}(Q_{\ll j}^+ \fC_n^+, Q_{\lesssim j}^+ \Psi u_n, \overline{Q_{\sim j}^+ w_n^+}) \biggr\rVert_{\ell^2} \label{eqn:fC-term-4}\\
    &+ \biggl\lVert \sum_j T_{m_j^\cS}(Q_{\sim j}^+ \fC_n^+, Q_{\lesssim j}^+ \Psi u_n, \overline{Q_{\ll j}^+ w_n^+}) \biggr\rVert_{\ell^2} \label{eqn:fC-term-5}\\
    &+ \bsim \notag
\end{align}
\end{subequations}
Since $\Psi \fC_n^+$ satisfies decay bounds similar to those for $\mSS_n$ and $u_n$, it is straightforward to obtain bounds for~\eqref{eqn:fC-term-1} by modifying the argument used to control the terms~\eqref{eqn:m-S-j-no-Psi} and~\eqref{eqn:m-S-j-u-SS-w}.  A similar argument applies to the terms~\eqref{eqn:fC-term-2} and~\eqref{eqn:fC-term-5}, since we can use the frequency localization to transfer the $\Psi$ multiplier onto the $Q_{\sim j}^+ \fC_n^+$ factor.  Thus, it only remains to control the terms~\eqref{eqn:fC-term-3} and~\eqref{eqn:fC-term-4}.  For~\eqref{eqn:fC-term-3}, we divide dyadically in frequency and use almost orthogonality to find that
\begin{equation*}\begin{split}
    \eqref{eqn:fC-term-3} \lesssim& \sum_j \sum_{k < j - 20} \lVert T_{m_j^\cS}(\chi_{\sim k}^+ Q_{\sim j}^+ u_n, \chi_{\sim k}^+ Q_{\lesssim j}^+ \Psi w_n, \chi_k^+ \overline{Q_{\ll j}^+ \fC_n^+}) \rVert_{\ell^2}\\
    &+ \biggl\lVert \sum_j \sum_{k < j - 20}\chi_{[j \pm 30]}^+ T_{m_j^\cS}(\chi_{[j \pm 30]}^+ Q_{\sim j}^+ u_n, \chi_{[j \pm 30]}^+ Q_{\lesssim j}^+ \Psi w_n, \chi_{[j \pm 20]}^+ \overline{Q_{\ll j}^+ \fC_n^+}) \biggr\rVert_{\ell^2}\\
    &+ \bbetter\\
    \lesssim& M \epsilon^4 t^{-1} \lVert ng_n \rVert_{\ell^2}
\end{split}
\end{equation*}
Similarly, we have that
\begin{equation*}\begin{split}
    \lVert \eqref{eqn:fC-term-4} \rVert_{\ell^2} \lesssim& \sum_j \sum_{k < j - 30} \lVert T_{m_j^\cS}(\chi_k^+ Q_{\ll j} \fC_n^+, \chi_{\sim k}^+ Q_{\lesssim j}^+ \Psi u_n, \chi_{\sim k}^+ \overline{Q_{\sim j}^+ w_n}) \rVert_{\ell^2}\\
    &+ \biggl\lVert \chi_{[j \pm 40]}^+ T_{m_j^\cS}(\chi_{[j \pm 30]}^+ Q_{\ll j} \fC_n^+,  Q_{\lesssim j}^+ \Psi u_n, \overline{Q_{\sim j}^+ w_n}) \biggr\rVert_{\ell^2}\\
    &+ \bbetter\\
    \lesssim& M\epsilon^4 t^{-1} \lVert ng_n \rVert_{\ell^2}
\end{split}
\end{equation*}
which completes the argument for the space non-resonant terms.

\subsubsection{Time non-resonant estimates}\label{sec:ng_n-tnr}

We now consider the time non-resonant terms~\eqref{eqn:time-non-res-term}.  If we write
\begin{equation*}
    \mu^\cT_s = \cos(\xi - \eta - \sigma) \frac{\partial_\xi \phi \chi^\cT_s}{\phi}
\end{equation*}
then we can integrate by parts in time to obtain
\begin{subequations}\begin{align}
   \int_1^t \eqref{eqn:time-non-res-term} \;ds =& \left. s \left\langle T_{\mu^\cT_s}(u_n, w_n, u_n), e^{is\discrAvg} n g_n(s) \right\rangle \right|_{s = 1}^{s=t} \label{eqn:tnr-bdy}\\
   &+ \int_1^t \left\langle T_{\mu^\cT_s}(u_n, w_n, u_n), e^{is\discrAvg} n g_n \right\rangle\;ds \label{eqn:tnr-no-s}\\
   &+ \int_1^t s \left\langle T_{\partial_s \mu^\cT_s}(u_n, w_n, u_n), e^{is\discrAvg} n g_n(s) \right\rangle\;ds \label{eqn:tnr-deriv-hits-symbol}\\
   &+ \int_1^t s \left\langle T_{\mu^\cT_s}(u_n, e^{is\discrAvg} \partial_s g_n, u_n), e^{is\discrAvg} n g_n(s) \right\rangle\;ds \label{eqn:tnr-deriv-hits-function}\\
   \begin{split}&+ \int_1^t s \left\langle T_{\mu^\cT_s}(u_n, w_n, \overline{u_n}) + T_{\mu^\cT_s}(u_n, \mSS_n, \overline{w}_n)\right.\\
   &\qquad\qquad\qquad \left.+ T_{\mu^\cT_s}(w_n, \mSS_n, \overline{u}_n), e^{is\discrAvg} n \partial_s g_n\right\rangle\;ds \end{split}\label{eqn:tnr-deriv-hits-IP}\\
   &+ \bsime\nonumber
\end{align}\end{subequations}

To help with bounding the terms, we will write $m_j$ for a time dependent symbol supported on $|\cos \xi| + |\cos \eta| + |\cos \sigma| \sim 2^j$ which satisfies the symbol bounds
\begin{equation*}
    |\partial_{\xi,\eta,\sigma}^\alpha m_j| \lesssim 2^{-|\alpha| j}
\end{equation*}
uniformly in time.  Using these symbols, we can write 
\begin{equation}\label{eqn:m-ct-decomp}
    \mu_s^\cT = \sum_{2^j \geq s^{-1/3}} m_j
\end{equation}
We now turn to the estimates.  We will estimate~\cref{eqn:tnr-bdy,eqn:tnr-no-s,eqn:tnr-deriv-hits-symbol,eqn:tnr-deriv-hits-function} in this section, and defer the estimate for~\eqref{eqn:tnr-deriv-hits-IP} until~\Cref{sec:final-ngn}.  For~\eqref{eqn:tnr-bdy}, we see by Cauchy-Schwarz that
\begin{equation*}
    \left|\eqref{eqn:tnr-bdy}\right| \lesssim t \left\lVert T_{\mu_s^\cT}(u_n, w_n, \overline{u_n})\right\rVert_{\ell^2} \left\lVert n g_n \right\rVert_{\ell^2} + \{\text{similar}\}
\end{equation*}
where $\{\text{similar}\}$ denotes the contribution from $s = 1$.  Thus, it suffices to bound the pseudoproduct $T_{\mu_s^\cT}(u_n, w_n, \overline{u_n})$ in $L^2$.  By~\eqref{eqn:m-ct-decomp}, we have
\begin{subequations}\begin{align}
    T_{\mu_s^\cT}(u_n, w_n, \overline{u_n}) =&  \sum_{2^j \geq s^{-1/3}} Q^+_{\sim j} T_{m_j}(Q^+_{\ll j} u_n, Q^+_{\sim j} w_n, \overline{Q^+_{\ll j} u_n})\label{eqn:tnr-pseudoprod-decomp-a}\\
    &+ \sum_{2^j \geq s^{-1/3}} T_{m_j}(Q^+_{\sim j} u_n, Q^+_{\lesssim j} w_n, \overline{Q^+_{\ll j} u_n})\label{eqn:tnr-pseudoprod-decomp-b}\\
    &+ \sum_{2^j \geq s^{-1/3}} T_{m_j}(Q^+_{\sim j} u_n, Q^+_{\lesssim j} w_n, \overline{Q^+_{\sim j} u_n})\label{eqn:tnr-pseudoprod-decomp-c}\\
    &+ \bsim\notag
\end{align}\end{subequations}
The terms~\eqref{eqn:tnr-pseudoprod-decomp-a} and~\eqref{eqn:tnr-pseudoprod-decomp-b} have the same form as those we obtained when considering~\eqref{eqn:m-S-j-no-Psi}.  The term~\eqref{eqn:tnr-pseudoprod-decomp-c} is new.  To handle it, we write
\begin{align*}
    \eqref{eqn:tnr-pseudoprod-decomp-c} =& \sum_{2^j \geq s^{-1/3}} \chi_{[j\pm 40]}T_{m_j}(Q^+_{\sim j} u_n, \chi_{[j\pm 30]} Q^+_{\lesssim j} w_n, \overline{Q^+_{\sim j} u_n})\\
    &+ \sum_{2^j \geq s^{-1/3}} T_{m_j}((1-\chi_{[j\pm 20]})Q^+_{\sim j} u_n, (1-\chi_{[j\pm 30]}) Q^+_{\lesssim j} w_n, (1-\chi_{[j\pm 20]})\overline{Q^+_{\sim j} u_n})\\
    &+ \bbetter\notag
\end{align*}
The second term is easily handled using the decay estimates for $(1-\chi_{[j\pm 20]})Q^+_{\sim j} u_n$, and the first term can be controlled using almost orthogonality.  Combining the estimates for~\cref{eqn:tnr-pseudoprod-decomp-a,eqn:tnr-pseudoprod-decomp-b,eqn:tnr-pseudoprod-decomp-c}, we find that
\begin{equation}\label{eqn:T-mu-s-bd}
    \lVert T_{\mu_s^\cT}(u_n, w_n, \overline{u_n}) \rVert_{L^2} \lesssim M^2\epsilon^2 s^{-1}\lVert ng_n \rVert_{\ell^2}
\end{equation}
and hence
\begin{equation*}
    |\eqref{eqn:tnr-bdy}| \lesssim M^2\epsilon^2 \lVert ng_n(t) \rVert_{\ell^2}^2
\end{equation*}
and
\begin{equation*}
    |\eqref{eqn:tnr-no-s}| \lesssim \int_1^t M^2\epsilon^2 s^{-1}\lVert ng_n(s) \rVert_{\ell^2}^2 \;ds
\end{equation*}
Furthermore, observing that we have the decomposition
\begin{equation*}
    s\partial_s \mu^\cT_s = \sum_{2^j \gtrsim s^{-1/3}} m_j
\end{equation*}
we see that~\eqref{eqn:tnr-deriv-hits-symbol} satisfies the same bounds as~\eqref{eqn:tnr-no-s}.  We now turn to~\eqref{eqn:tnr-deriv-hits-function}.  Using~\eqref{eqn:n-gn-deriv}, we can write
\begin{equation*}\begin{split}
    \left|\eqref{eqn:tnr-deriv-hits-function}\right| \lesssim& \int_1^t s\lVert T_{\mu^\cT_s}(u_n, |u_n|^2 \discrAvg w_n, \overline{u_n})\rVert_{\ell^2} \lVert n g_n \rVert_{\ell^2}\;ds\\
    &+ \int_1^t s\lVert T_{\mu^\cT_s}(u_n, w_n \overline{u}_n \discrAvg \mSS_n, \overline{u_n})\rVert_{\ell^2} \lVert n g_n \rVert_{\ell^2}\;ds\\
    &+ \int_1^t s\lVert T_{\mu^\cT_s}(u_n, R_n, \overline{u_n})\rVert_{\ell^2} \lVert n g_n \rVert_{\ell^2}\;ds\\
    &+\bsim
\end{split}\end{equation*}
so it suffices to show that
\begin{equation}\label{eqn:tnr-dhf-12-bound}
    \lVert T_{\mu^\cT_s}(u_n, |u_n|^2 \discrAvg w_n, \overline{u_n})\rVert_{\ell^2} + \lVert T_{\mu^\cT_s}(u_n, w_n \overline{u}_n \discrAvg \mSS_n, \overline{u_n})\rVert_{\ell^2} \lesssim M^4 \epsilon^4 s^{-2} \lVert ng_n \rVert_{\ell^2}
\end{equation}
and
\begin{equation}\label{eqn:tnr-dhf-3-bound}
    \lVert T_{\mu^\cT_s}(u_n, R_n, \overline{u_n})\rVert_{\ell^2} \lesssim M^2 \epsilon^2 (\epsilon + M^3\epsilon^3) s^{-\frac{19}{10}}
\end{equation}
since this implies that
\begin{equation*}
    \left|\eqref{eqn:tnr-deriv-hits-function}\right| \lesssim \int_1^t M^2\epsilon^2 s^{-1} \lVert ng_n(s)\rVert_{\ell^2}^2 \;ds + M^2 \epsilon^4 t^{1/5}
\end{equation*}
as required.  We begin with the bound for the term containing $|u_n|^2 \Psi w_n$.  From~\eqref{eqn:m-ct-decomp}, we have that
\begin{subequations}\begin{align}
    T_{\mu^\cT_s}(u_n, |u_n|^2 \discrAvg w_n, \overline{u_n}) 
    =& \sum_{2^j \geq s^{-1/3}} T_{m_j}(Q^+_{\sim j} u_n, |u_n|^2 \discrAvg w_n, \overline{Q^+_{\sim j} u_n})\label{eqn:tnr-dhf-1-1}\\
    &+ \sum_{2^j \geq s^{-1/3}} Q_{\sim j}^+ T_{m_j}\left(Q^+_{\ll j} u_n, Q_{\sim j}^+\left(|u_n|^2 \discrAvg w_n\right), \overline{Q^+_{\ll j} u_n}\right)\label{eqn:tnr-dhf-1-2}\\
    &+ \bsime\nonumber
\end{align}\end{subequations}
For the first term, using the dispersive estimates for $u_n$ and $w_n$ gives the bound
\begin{equation}\label{eqn:uu-Psi-w-bound}
    \lVert \chi_{\sim k}^+ |u_n|^2 \Psi w_n \rVert_{\ell^2} \lesssim M^2\epsilon^2 s^{-1} 2^k \lVert ng_n \rVert_{\ell^2} c_k
\end{equation}
which lets us write
\begin{equation*}\begin{split}
    \left\lVert \eqref{eqn:tnr-dhf-1-1} \right\rVert_{\ell^2}  \lesssim&\Bigl\lVert \sum_{2^j \geq s^{-1/3}} \chi_{[j-40,j+40]}^+ T_{m_j}( Q_{\sim j}^+ u_n, \chi_{[j-30,j+30]}^+ |u_n|^2 \Psi w_n, \chi_{[j-40,j+40]}^+\overline{Q_{\lesssim j}^+ u_n}) \Bigr\rVert_{\ell^2}\\
    &+\sum_{2^j \geq s^{-1/3}} \lVert T_{m_j}(\chi_{< j - 20}^+ Q_{\sim j}^+ u_n, \chi_{< j - 30}^+ |u_n|^2 \Psi w_n, \overline{Q_{\lesssim j}^+ u_n}) \rVert_{\ell^2}+ \bbetter\\
    \lesssim& M^4\epsilon^4 s^{-2}\lVert ng_n \rVert_{\ell^2}
\end{split}\end{equation*}
For the second term, note that $\chi_{\sim k}^+ Q^+_{\sim j} (|u_n|^2 \discrAvg w_n)$ obeys better bounds than~\eqref{eqn:uu-Psi-w-bound} when ${2^k \ll 2^j}$, since the frequency localization prevents the leading order interaction $|Q^+_{\sim k} u_n|^2 \discrAvg Q^+_{\sim k} w_n$:
\begin{equation}\label{eqn:uu-Psi-w-bound-refined}\begin{split}
    \bigl\lVert \chi_{k}^+ Q^+_{\sim j} \left( |u_n|^2 \discrAvg w_n \right) \bigr\rVert_{\ell^2} \lesssim& \bigl\lVert \chi_{k}^+ Q^+_{\sim j} ( \chi_{\sim k}^+ Q^+_{\gtrsim j} u_n \overline{u_n} \discrAvg w_n ) \bigr\rVert_{\ell^2}\\
    +& \bigl\lVert \chi_{k}^+ Q^+_{\sim j} ( \chi_{\sim k}^+ Q^+_{\ll j} u_n Q^+_{\gtrsim j} \overline{u_n} \discrAvg w_n ) \bigr\rVert_{\ell^2}\\
    &+ \bigl\lVert \chi_{k}^+ Q^+_{\sim j} ( \chi_{\sim k}^+ |Q^+_{\ll j} u_n|^2 Q^+_{\sim j}\discrAvg w_n ) \bigr\rVert_{\ell^2}\\
    &+ \bsime\\
    \lesssim& M^2\epsilon^2 s^{-4/3} 2^{-3/2j + 3/2 k} \lVert  n g_n \rVert_{\ell^2} c_k + M^2 \epsilon^2 s^{-4/3} \lVert n g_n \rVert_{\ell^2} c_j
\end{split}\end{equation}
On the other hand, for $2^k \gtrsim 2^j$, $\chi_{\sim k}^+ Q^+_{\sim j} (|u_n|^2 \discrAvg w_n)$ obeys the same bounds given by~\eqref{eqn:uu-Psi-w-bound}.  Applying this to~\eqref{eqn:tnr-dhf-1-2} and using almost orthogonality yields
\begin{equation*}\begin{split}
    \lVert \eqref{eqn:tnr-dhf-1-2} \rVert_{\ell^2} \lesssim& \biggl\lVert \sum_{\substack{2^j \geq s^{-1/3}\\ s^{-1/3} \leq 2^k < 2^{j-30}}}  \chi_{\sim k}^+ Q_{\sim j}^+ T_{m_j}(\chi_{\sim k}^+ Q_{\ll j}^+ u_n, \chi_k^+ Q_{\sim j}^+ (|u_n|^2 \Psi w_n), \chi_{\sim k}^+ \overline{Q_{\ll j}^+ u_n})\biggr\rVert_{\ell^2}\\
    +& \biggl\lVert \sum_{2^j \geq s^{-\frac{1}{3}}}  T_{m_j}(\chi_{[j-40, j+40]}^+ Q_{\ll j}^+ u_n, \chi_{[j-30, j+30]}^+ Q_{\sim j}^+ (|u_n|^2 \Psi w_n), \chi_{[j-40, j+40]}^+ \overline{Q_{\ll j}^+ u_n})\biggr\rVert_{\ell^2}\\
    +& \bbetter\\
    \lesssim& M^4\epsilon^4 s^{-7/3}\lVert n g_n \rVert_{\ell^2}\Biggl( \sum_{\substack{2^j \geq s^{-1/3}\\ s^{-1/3} \leq 2^k < 2^{j-30}}} 2^{k-3j} + 2^{-2k} c_j^2\Biggr)^{1/2}\\
    &+ M^4\epsilon^4 s^{-2} \lVert ng_n \rVert_{\ell^2} \biggl( \sum_{2^j \geq s^{-1/3}} c_j^2\biggr)^{1/2}\\
    \lesssim& M^4\epsilon^4 s^{-2}\lVert ng_n \rVert_{\ell^2}
\end{split}\end{equation*}

The argument for the term containing the $w_n\overline{u_n} \discrAvg \mSS_n$ factor is similar once we establish analogs for~\eqref{eqn:uu-Psi-w-bound} and~\eqref{eqn:uu-Psi-w-bound-refined}.  If we localize only in space, we find that
\begin{equation}\label{eqn:wu-Psi-SS-bound}
    \lVert \chi_{k}^\pm  w_n \overline{u_n} \discrAvg \mSS_n \rVert_{\ell^2} \lesssim M\epsilon^2 s^{-1} 2^k \lVert n g_n \rVert_{\ell^2} c_k
\end{equation}
which is slightly better than~\eqref{eqn:uu-Psi-w-bound}, since it is linear in the bootstrap parameter $M$.  If localize first in frequencies and then in space, then for $2^k \ll 2^j$ the argument from~\eqref{eqn:uu-Psi-w-bound-refined} yields that
\begin{equation}\label{eqn:wu-Psi-SS-bound-refined}\begin{split}
    \lVert \chi_{k}^\pm  Q_{\sim j}^+ \left(w_n \overline{u_n} \discrAvg \mSS_n\right) \rVert_{\ell^2} \lesssim& \left\lVert \chi_{k}^\pm  Q_{\sim j}^+\left(  w_n \overline{u_n} Q_{\gtrsim j}^+ \discrAvg \mSS_n\right)\right\rVert_{\ell^2} \\
    &+ \left\lVert \chi_{k}^\pm  Q_{\sim j}^+\left(\chi_{\sim k}^\pm   w_n Q_{\gtrsim j}^+ \overline{u_n} \discrAvg Q_{\ll j}^+ \mSS_n\right) \right\rVert_{\ell^2}\\
    &+ \left\lVert \chi_{k}^\pm Q_{\sim j}^+\left(\chi_{\sim k}^\pm Q^+_{\sim j} w_n Q_{\ll j}^+ \overline{u_n} \discrAvg Q_{\ll j}^+ \mSS_n\right)\right\rVert_{\ell^2}\\
    &+ \bsime\\
    \lesssim& M\epsilon^2 s^{-4/3} 2^{k/2-j/2} \lVert n g_n \rVert_{\ell^2} c_k + M\epsilon^2 s^{-4/3} 2^{k - j} \lVert n g_n \rVert_{\ell^2} c_j
\end{split}\end{equation}
which can be used similarly to~\eqref{eqn:uu-Psi-w-bound-refined}.  This completes the estimate~\eqref{eqn:tnr-dhf-12-bound}.  It only remains to prove~\eqref{eqn:tnr-dhf-3-bound}.  Recall that $R_n$ is given by~\eqref{eqn:R-n-def}.  As in~\Cref{sec:ss-eqn-error}, we will focus on the terms containing $R_n^+$.  Using~\eqref{eqn:Rn-decomp}, we find that
\begin{subequations}\begin{align}
    \sum_{2^j \geq s^{-1/3}} T_{m_j}(u_n, R_n, \overline{u_n}) =& \sum_{2^j \geq s^{-1/3}} T_{m_j}(u_n, R_n^{\text{lin},+}, \overline{u_n})\label{eqn:tnr-dhf-3-a}\\
    &+\sum_{2^j \geq s^{-1/3}} T_{m_j}(u_n, R_n^{\text{mod},+}, \overline{u_n})\label{eqn:tnr-dhf-3-b}\\
    &+\sum_{2^j \geq s^{-1/3}} T_{m_j}(u_n, R_n^{\text{cubic},+}, \overline{u_n})\label{eqn:tnr-dhf-3-c}\\
    &+ \bsime\notag
\end{align}\end{subequations}
By observing that $Q_{\sim j} R_n^{\text{lin},+}$ vanishes for $2^j \gg s^{-1/5}$ and using~\eqref{eqn:R_n-lin-ests}, we find that
\begin{equation*}\begin{split}
    \lVert \eqref{eqn:tnr-dhf-3-a} \rVert_{\ell^2} \lesssim& \sum_{\substack{s^{-1/3} \leq 2^j \lesssim s^{-1/5}\\ 2^k \geq s^{-1/3}}} \lVert \chi_{k}^+ T_{m_j}(\chi_{\sim k}^+u_n, \chi_{\sim k}^+ Q_{\lesssim j}^+ R_n^{\text{lin},+}, \chi_{\sim k}^+\overline{u_n})\rVert_{\ell^2}\\
    &+ \sum_{{2^j \gg s^{-1/5}}} \lVert T_{m_j}(Q^+_{\sim j} u_n, \chi_{[j - 20, j+20]}^+R_n^{\text{lin},+}, \chi_{[j - 30, j+30]}^+\overline{u_n})\rVert_{\ell^2}\\
    &+ \btext{similar or easier terms}\\
    \lesssim& M^2\epsilon^3 s^{-\frac{19}{10}}
\end{split}\end{equation*}
which is sufficient.  Turning to~\eqref{eqn:tnr-dhf-3-b}, we perform a further decomposition in frequency to obtain
\begin{subequations}\begin{align}
    \sum_{j} T_{m_j}(u_n, R_n^{\text{mod},+}, \overline{u_n}) =& \sum_{2^j \lesssim s^{-1/5}} T_{m_j}(u_n, Q_{[j-30, j+30]}^+ R_n^{\text{mod},+}, \overline{u_n})\label{eqn:tnr-r-mod-1}\\
    &+ \sum_{2^j \lesssim s^{-1/5}} T_{m_j}(Q_{\sim j}^+ u_n,  Q_{< j - 30}^+ R_n^{\text{mod},+}, \overline{u_n}) \label{eqn:tnr-r-mod-2}\\
    &+ \bsime\notag
\end{align}
\end{subequations}
For~\eqref{eqn:tnr-r-mod-1}, we see (based on the decay of $u_n$) that the worst terms will come from the region where $|n| \approx 2t$, so a quick calculation yields
\begin{equation*}\begin{split}
    \lVert \eqref{eqn:tnr-r-mod-1} \rVert_{\ell^2} \lesssim& \sum_{s^{-1/3} \lesssim 2^j \lesssim s^{-1/5}} \lVert T_{m_j}(\chi_{< j - 60}^+ u_n, \chi_{< j - 50}^+Q_{[j-30, j+30]}^+ R_n^{\text{mod},+}, \chi_{< j - 60}^+ \overline{u_n}) \rVert_{\ell^2}\\
    &+\bbetter\\
    \lesssim& M^5\epsilon^5 s^{-\frac{31}{15}} \sum_{2^j \gtrsim s^{-1/3}} 2^{-j/2} \ln(2 + s^{1/3} 2^{j})\\
    &\lesssim M^5 \epsilon^5 s^{-\frac{19}{10}}
\end{split}
\end{equation*}
which is in line with~\eqref{eqn:tnr-dhf-3-bound}.  Similarly, for~\eqref{eqn:tnr-r-mod-2}, we find that
\begin{equation*}\begin{split}
    \lVert \eqref{eqn:tnr-r-mod-2} \rVert_{\ell^2} \lesssim& \sum_{s^{-1/3} \lesssim 2^j \lesssim s^{-1/5}} \lVert T_{m_j}(Q_{\sim j}^+ u_n, \chi_{> j - 30}^+Q_{< j - 50}^+ R_n^{\text{mod},+}, \chi_{> j}^+\overline{u_n}) \rVert_{\ell^2}\\
    &+ \sum_{s^{-1/3} \lesssim 2^j \lesssim s^{-1/5}} \lVert T_{m_j}(\chi_{< j - 20} Q_{\sim j}^+ u_n, \chi_{< j - 30}^+Q_{< j - 30}^+ R_n^{\text{mod},+}, \chi_{> j}^+\overline{u_n}) \rVert_{\ell^2}\\
    &+\bbetter\\
    \lesssim& M^5\epsilon^5 s^{-\frac{12}{5}} \sum_{2^j \gtrsim s^{-1/3}} 2^{-3/2j} \ln(2 + s^{1/3} 2^{j}) + M^5\epsilon^5 s^{-\frac{67}{30}} \sum_{2^j \gtrsim s^{-1/3}} 2^{-j/2} \ln(2 + s^{1/3} 2^{j})\\
    \lesssim& M^5 \epsilon^5 s^{-\frac{19}{10}}
\end{split}
\end{equation*}

Finally, bound~\eqref{eqn:tnr-dhf-3-c}.  Based on the frequency localization of the symbols $m_j$, we can write
\begin{subequations}\begin{align}
    \sum_{j} T_{m_j}(u_n, R_n^{\text{cubic},+}, \overline{u_n}) =& \sum_{j} Q_{\sim j}^+T_{m_j}(Q_{\ll j}^+u_n, Q_{\sim j}^+R_n^{\text{cubic},+}, Q_{\ll j}^+\overline{u_n})\label{eqn:tnr-rncp-1-a}\\
    &+ \sum_{j} T_{m_j}(Q_{\sim j}^+u_n, R_n^{\text{cubic},+}, \overline{u_n})\label{eqn:tnr-rncp-1-b}\\
    &+ \bsim\notag
\end{align}\end{subequations}
We will show how to bound~\eqref{eqn:tnr-rncp-1-a}: the argument for~\eqref{eqn:tnr-rncp-1-b} is simpler, since the frequency localization of $u_n$ gives additional decay.  By introducing a decomposition in the spatial variables, we can write
\begin{equation*}\begin{split}
    \eqref{eqn:tnr-rncp-1-a} =& \sum_{j,k} \chi_{\sim k}^+ Q_{\sim j}^+T_{m_j}(\chi_{\sim k}^+ Q_{\ll j}^+u_n, \chi_{k}^+ Q_{\sim j}^+R_n^{\text{cubic},+}, \chi_{\sim k}^+ Q_{\ll j}^+\overline{u_n}) + \bbetter
\end{split}\end{equation*}
Thus, by using almost orthogonality in $j$ and $k$ together with the decay estimates~\eqref{eqn:R-cubic-decay-bds}, we see that
\begin{equation*}\begin{split}
    \lVert \eqref{eqn:tnr-rncp-1-a} \rVert_{\ell^2} \lesssim& \left( \sum_{j,k} \lVert T_{m_j}(\chi_{\sim k}^+Q_{\ll j}^+u_n, \chi_{k}^+ Q_{\sim j}^+R_n^{\text{cubic},+}, \chi_{\sim k}^+ Q_{\ll j}^+\overline{u_n}) \rVert_{\ell^2}^2 \right)^{\frac{1}{2}} + \bbetter\\
    \lesssim& M^2\epsilon^2 s^{-1} \left( \sum_{2^k \lesssim s^{-1/3}} 2^{-2k}\lVert \chi_{k}^+ R_n^{\text{cubic},+} \rVert_{\ell^2}^2 \right)^{\frac{1}{2}} + \bbetter\\
    \lesssim& M^2 \epsilon^5 s^{-\frac{19}{10}}
\end{split}
\end{equation*}
as required.

\subsection{\label{sec:weighted-l2-eqn-err}Estimates for the error \texorpdfstring{$LR_n$}{LR\_n}}
We now consider the estimates for the term~\eqref{eqn:ng-deriv-eqn-error}.  We will show that
\begin{equation}
    \left|\int_1^t \langle L R_n, Lw_n \rangle_{\ell^2}\;ds - \eqref{eqn:m-R-2-d-tnr-f}\right| \lesssim \int_1^t \epsilon s^{-\frac{9}{10}}\lVert ng_n \rVert_{\ell^2}\;ds\label{eqn:n-Rn-desired-bound}
\end{equation}
where~\eqref{eqn:m-R-2-d-tnr-f} is a term we defer until~\Cref{sec:final-ngn}.  To prove~\eqref{eqn:n-Rn-desired-bound}, we will again employ the decomposition given in~\cref{eqn:R-n-cross-def,eqn:R-n-lin-def,eqn:R-n-mod-def,eqn:R-n-cubic-def}.  We will focus on estimating the terms $n\invLinProp[t]R^{\text{lin},+}$, $n\invLinProp[t]R^{\text{mod},+}$, and $n\invLinProp[t]R^{\text{cubic},+}$, as the other terms are simpler or easier.

\subsubsection{The estimates for the linear and modulation terms}  We begin by considering $n\invLinProp[t]R^{\text{lin},+}$.  By using Plancherel's theorem and translating in Fourier space, we find that
\begin{subequations}\begin{align}
    \lVert L R_n^{\text{lin},+} \rVert_{\ell^2} =& \lVert n \invLinProp  R_n^{\text{lin},+} \rVert_{\ell^2}\notag\\
    =& \left\lVert \partial_\xi\left( e^{-2it\cos(\xi + \fhPi)} \widehat{R^{\text{lin},+}}(\xi + \fhPi)\right) \right\rVert_{L^2}\nonumber\\
    \lesssim& \left\lVert \partial_{\xi} \left( \SSCutoff E(\xi,t) \left(\sin \xi - \xi + \xi^3 /6\right) \hat{\theta}\right) \right\rVert_{L^2}\label{eqn:R-n-lin-I}\\
    &+ t^{-4/5}\left\lVert \partial_{\xi} \left( \xi \rho'(t^{1/5} \xi) E(\xi,t) \hat{\theta}\right) \right\rVert_{L^2}\label{eqn:R-n-lin-II}
\end{align}\end{subequations}
where 
\begin{equation*}
    E(\xi,t) = \exp\left(2i \left(\sin \xi - \xi + \xi^3/6 \right) t\right)
\end{equation*}  Noting that $|\partial_\xi E| \lesssim t \xi^4$, we find that 
\begin{equation*}\begin{split}
    \eqref{eqn:R-n-lin-I} \lesssim& t^{1/5} \lVert \xi^5 \rho'(t^{1/5} \xi) \hat{\theta} \rVert_{L^2} + t \lVert \xi^9 \SSCutoff \hat{\theta}\rVert_{L^2}\\
    &+ \lVert \xi^4 \SSCutoff \hat{\theta} \rvert_{L^2} + \lVert \xi^5 \SSCutoff \partial_\xi \hat{\theta} \rVert_{L^2}\\
    \lesssim& t^{-9/10} \lVert \hat{\theta} \rVert_{L^\infty} + \lVert \SSCutoff[D] |\fs|^2\fs \rVert_{L^2}\\
    \lesssim& \epsilon(1 + \epsilon^2)t^{-9/10}
\end{split}\end{equation*}
By a similar argument
\begin{equation*}\begin{split}
    \eqref{eqn:R-n-lin-II} \lesssim& t^{-4/5} \lVert \left(\xi t^{1/5} \rho''(t^{1/5} \xi) + \rho'(t^{1/5} \xi)\right) \hat{\theta} \rVert_{L^2} + t^{{1/5}} \lVert \xi^5 \rho'(t^{1/5} \xi) \hat{\theta} \rVert_{L^2}\\
    &+ t^{-4/5} \lVert \xi \rho'(t^{1/5} \xi) \partial_x \hat{\theta} \rVert_{L^2}\\
    \lesssim& \epsilon\left(1 + \epsilon^2 \right) t^{-9/10}
\end{split}\end{equation*}
Hence,
\begin{equation*}
    \lVert n \invLinProp  R_n^{\text{lin},+} \rVert_{\ell^2} \lesssim \epsilon t^{-9/10}
\end{equation*}
which is acceptable.  A similar argument gives the same bound for $n\invLinProp R_n^{\text{mod},+}$ once we use estimates~\cref{eqn:D-alpha-L-fs-bdds,eqn:D-alpha-fs-bdds}:
\begin{equation*}\begin{split}
    \left\lVert n \invLinProp  R_n^{\text{mod},+} \right\rVert_{\ell^2} =& \left\lVert \partial_{\xi} \left(\SSCutoff E(\xi,t) \left[D_\alpha \hat{\fs}(\xi;\alpha_+)\right] \cdot \partial_t \alpha_+\right) \right\rVert_{L^2}\\
    \lesssim& \lVert L D_\alpha \fs \rVert_{L^2}|\partial_t \alpha_+| + t^{{1/5}} \lVert \rho'(t^{1/5} \xi) D_\alpha \hat{\fs}(\xi,\alpha_+)\rVert_{L^2} |\partial_t \alpha_+|\\
    &+ \lVert \rho(t^{\frac{1}{5}}\xi) E'(\xi,t) D_{\alpha} \hat{\fs}(\xi, \alpha_+) \rVert_{L^2} |\partial_t \alpha_+|\\
    \lesssim& M^3\epsilon^5 t^{-9/10}
\end{split}\end{equation*}

\subsubsection{The estimates for the cubic term} 
The cubic term is the most difficult.  We begin by rewriting the term as
\begin{equation*}
    \invLinProp R^{\text{cubic},+}_n = \invLinProp \Rncp{1} + \invLinProp \Rncp{2} + \invLinProp \Rncp{3}
\end{equation*}
where $\Rncp{m}$, $m=1,2,3$ given by~\eqref{eqn:R-n-cubic-division}.
Taking the Fourier transform, we find that
\begin{equation}\label{eqn:R-cubic-nonlin-div}\begin{split}
    \mathcal{F}\left(n \Rncp{m}\right)(\xi + \fhPi) =& -\partial_\xi E e^{-it\xi^3} \hat{T}_{m^{R,m}}(\fs, \fs, \overline{\fs})\\
    &- E e^{-it\xi^3/3} \hat{T}_{\partial_\xi m^{R,m}}(\fs, \fs, \overline{\fs})\\
    &- tE e^{-it\xi^3/3} \hat{T}_{m^{R,m}}( \fs, |\fs|^2\fs, \overline{\fs})\\
    &- tE e^{-it\xi^3/3} \hat{T}_{\partial_\xi \phi_{\tmkdv}m^{R,m}}(\fs, \fs, \overline{\fs})\\
    &=: \rmI_m + \rmII_m + \rmIII_m + \rmIV_m
\end{split}\end{equation}
where $\phi_{\tmkdv}: \bbR^3 \to \bbR$ is the phase function function for the mKdV equation, given by
\begin{equation*}
    \phi_{\tmkdv}(\xi,\eta,\sigma) = \frac{-\xi^3 + \eta^3 + (\xi - \eta - \sigma)^3 + \sigma^3}{3} = - (\xi - \eta)(\xi - \sigma)(\eta + \sigma)
\end{equation*}
we now show how to control the terms $\rmI_m$, $\rmII_m$, $\rmIII_m$ and $\rmIV_m$ for $m=1,2,3$.

\paragraph{\textbf{Case $m=3$}:}  The rapid decay of $m^{R,3}$ at low frequencies allows us to get estimates with faster time decay than we need to bound~\eqref{eqn:ng-deriv-eqn-error}.  Recall that we have the decomposition
\begin{equation*}
    m^{R,3} = \sum_{2^j \lesssim t^{-{1/5}}} 2^{3j} m_j
\end{equation*}
which gives us the bounds
\begin{equation*}\begin{split}
    \lVert \rmI_3 \rVert_{L^2_\xi} \lesssim& \smashoperator[r]{\sum_{2^j \lesssim t^{-{1/5}}}} t2^{7j} \lVert T_{m_j}(\fs, \fs, \overline{\fs}) \rVert_{L^2_x}\\
    \lesssim& \smashoperator[r]{\sum_{2^j \lesssim t^{-{1/5}}}} t2^{7j} \lVert \fs \rVert_{L^6_x}^3\\
    \lesssim& \epsilon^3 t^{-\frac{37}{30}}\\
    \lVert \rmIII_3 \rVert_{L^2_\xi} \lesssim& \smashoperator[r]{\sum_{2^j \lesssim t^{-1/5}}} t2^{3j} \lVert T_{m_j}(\fs, |\fs|^2\fs, \overline{\fs}) \rVert_{L^2}\\
    \lesssim& \smashoperator[r]{\sum_{2^j \lesssim t^{-1/5}}} t2^{3j} \lVert \fs \rVert_{L^{10}_x}^5\\
    \lesssim& \epsilon^5 t^{-\frac{11}{10}}
\end{split}\end{equation*}
Similarly, using the decomposition
\begin{equation*}\begin{split}
    \partial_\xi m^{R,3} =& \smashoperator[r]{\sum_{2^k \sim t^{-{1/5}}}} 2^{2k} m_k\\
\end{split}\end{equation*}
we find that
\begin{equation*}\begin{split}
    \lVert \rmII_3 \rVert_{L^2_\xi} \lesssim& \sum_{2^k \sim t^{-{1/5}}} 2^{2k} \lVert T_{m_k}(\fs, \fs, \overline{\fs}) \rVert_{L^2_x}\\
    \lesssim& \epsilon^3 t^{-\frac{37}{30}}
\end{split}\end{equation*}
Finally, since $\partial_\xi \phi_{\tmkdv}$ is quadratic in the phase variables,
\begin{equation*}
    \partial_{\xi} \phi_{\tmkdv} m^{R,3} = \smashoperator[r]{\sum_{2^j \lesssim t^{-1/5}}} 2^{5j} m_j
\end{equation*}
so
\begin{equation*}\begin{split}
    \lVert \rmIV_3 \rVert_{L^2_\xi} \lesssim& \smashoperator[r]{\sum_{2^j \lesssim t^{-1/5}}} t2^{5j} \lVert T_{m_j}(Q_{\sim j}^\tcts\fs, \fs, \overline{\fs}) \rVert_{L^2_x} + \bsim\\
    \lesssim& \smashoperator[r]{\sum_{2^j \lesssim t^{-1/5}}} t2^{5j} \lVert T_{m_j}(\chi_{\geq j - 20}^\tcts Q_{\sim j}^\tcts\fs, \chi_{\geq j - 30}^\tcts \fs, \chi_{\geq j - 30}^\tcts \overline{\fs}) \rVert_{L^2_x}\\
    &+ \smashoperator[r]{\sum_{2^j \lesssim t^{-1/5}}} t2^{5j} \lVert T_{m_j}(\chi_{< j - 20}^\tcts Q_{\sim j}^\tcts\fs,  \fs,  \overline{\fs}) \rVert_{L^2_x}\\
    &+ \bsime\\
    \lesssim& \epsilon^3 t^{-\frac{11}{10}}
\end{split}\end{equation*}
All of these estimates are consistent with~\eqref{eqn:n-Rn-desired-bound}.

\paragraph{\textbf{Case $m=2$}:} Recall that $m^{R,2}$ is supported in the region
\begin{equation}\label{eqn:m-R-2-supp-cond}
    \left\{(\xi,\eta,\sigma) : |\xi|, \max(|\eta|, |\sigma|, |\xi - \eta - \sigma|)  \sim t^{-1/5} \right\}
\end{equation}
Thus, using the estimate $|\partial_\xi E| \lesssim t 2^{4j}$ together with the bounds of $\fs$ from~\eqref{eqn:fs-decay-bdds} and~\eqref{eqn:fs-decay-bdds-low}, we find that
\begin{equation*}\begin{split}
    \lVert \rmI_2 \rVert_{L^2_\xi} \lesssim& \smashoperator[r]{\sum_{2^j \sim t^{-1/5}}}t 2^{5j} \lVert T_{m_j}(Q_{\sim j}^\tcts \fs, \fs, \overline{\fs})\rVert_{L^2} + \bsim\\
    \lesssim& \smashoperator[r]{\sum_{2^j \sim t^{-1/5}}}t 2^{5j} \lVert T_{m_j}(\chi^\tcts_{\geq j - 20}Q_{\sim j}^\tcts \fs, \chi^\tcts_{\geq j - 30}\fs, \chi^\tcts_{\geq j - 30}\overline{\fs})\rVert_{L^2}\\
    &+ \smashoperator[r]{\sum_{2^j \sim t^{-1/5}}}t 2^{5j} \lVert T_{m_j}(\chi^\tcts_{< j - 20}Q_{\sim j}^\tcts \fs, \fs, \overline{\fs})\rVert_{L^2} + \bsime\\
    \lesssim& \epsilon^3 \sum_{2^j \sim t^{-1/5}} 2^{9/2j} + \epsilon^3 \sum_{2^j \sim t^{-1/5}} t^{-2/9} 2^{23/6 j}\\
    \lesssim& \epsilon^3 t^{-\frac{9}{10}}
\end{split}
\end{equation*}
The estimate for $\rmII_2$ is analogous, since
\begin{equation*}
    \partial_\xi m^{R,2} = \sum_{2^j \sim t^{-1/5}} m_j = \sum_{2^j \sim t^{-1/5}} t 2^{5j} m_j
\end{equation*}
Turning to~$\rmIII_2$, we write
\begin{subequations}\begin{align}
    \rmIII_2 =& \sum_{2^j \sim t^{-1/5}} t2^j \hat{T}_{m_j}(\fs, Q_{\sim j}^\tcts (|\fs|^2\fs), \overline{\fs})\label{eqn:III-2-a}\\
    &+ \sum_{2^j \sim t^{-1/5}} t2^j \hat{T}_{m_j}(Q_{\sim j}^\tcts\fs, |\fs|^2\fs, \overline{\fs})\label{eqn:III-2-b}\\
    &+\bsim \notag
\end{align}
\end{subequations}
The argument for~\eqref{eqn:III-2-b} is similar to the one used to control $\rmI_2$:
\begin{equation*}\begin{split}
    \eqref{eqn:III-2-b} \lesssim& \smashoperator[r]{\sum_{2^j \sim t^{-1/5}}} \lVert T_{m_j}(\chi_{\geq j - 20}^\tcts \fs, \chi_{\geq j - 30}^\tcts |\fs|^2\fs, \chi_{\geq j - 30}^\tcts \overline{\fs}) \rVert_{L^2}\\
    &+ \smashoperator[r]{\sum_{2^j \sim t^{-1/5}}} \lVert T_{m_j}(\chi_{< j - 20}^\tcts \fs, |\fs|^2\fs,\overline{\fs}) \rVert_{L^2}\\
    \lesssim& \epsilon^5 t^{-\frac{9}{10}}
\end{split}
\end{equation*}
For~\eqref{eqn:III-2-a}, we note that the frequency localization of $Q_{\sim j}^\tcts(|\fs|^2\fs)$ gives it better decay near $x = 0$ compared to the same term with no frequency localization.  In particular, using~\Cref{lem:LP-pseudoloc-lemma} we have that for $k < j - 50$,
\begin{equation*}\begin{split}
    \lVert \chi_{k}^\tcts Q_{\sim j}^\tcts(|\fs|^2\fs) \rVert_{L^2} =& \lVert \chi_{k}^\tcts Q_{\sim j}^\tcts( \chi_{[k-20, k+20]}^\tcts |\fs|^2Q_{[j - 20, j + 20]}^\tcts\fs)\rVert_{L^2} +\bsime \\
    \lesssim& \epsilon^3 t^{-4/3} 2^{-3/2j}
\end{split}\end{equation*}
(cf.~\cite[Section 3.2]{stewartLongTimeDecay2021}).  Thus, we have the bound
\begin{equation*}\begin{split}
    \lVert \eqref{eqn:III-2-a} \rVert_{L^2} \lesssim& \sum_{2^j \sim t^{-1/5}} \sum_{k < j - 50} t 2^j \lVert T_{m_j}(\chi_{\sim k}^\tcts \fs, \chi_{k}^\tcts Q_{\sim j}^\tcts(|\fs|^2\fs), \chi_{\sim k}^\tcts \overline{\fs}) \rVert_{L^2}\\
    &+ \sum_{2^j \sim t^{-1/5}} t 2^j \lVert T_{m_j}(\chi_{\geq j - 40}^\tcts \fs, \chi_{\geq j - 50}^\tcts Q_{\sim j}^\tcts(|\fs|^2\fs), \chi_{\geq j - 40}^\tcts \overline{\fs}) \rVert_{L^2}\\
    \lesssim& \epsilon^5 t^{-\frac{9}{10}}
\end{split}
\end{equation*}

The only remaining term is $\rmIV_2$.  To control this term, we must take advantage of the space-time resonance structure in $\phi_{\tmkdv}$.  After taking advantage of the vanishing of $\partial_\xi \phi$, we find that the space-time resonances are
\begin{equation}\label{eqn:mKdV-str-def}\begin{split}
    \mathcal{S}_{\tmkdv} =& \{\eta = \sigma = \xi/3\}\\
    \mathcal{T}_{\tmkdv} =& \{\xi = \eta\} \cup \{\xi = \sigma\}\\
    \mathcal{R}_{\tmkdv} =& \{(0,0,0)\}
\end{split}\end{equation}
Note that the space-time resonance set $\mathcal{R}_\tmkdv$ is separated from the support of $m^{R,2}$, so we only need to consider the space and time resonances.  Let $\chi^\mathcal{S}_{\tmkdv}$ and $\chi^\mathcal{T}_{\tmkdv}$ be a smooth, zero-homogeneous partition of unity on $\bbR^3 \setminus \{0\}$ such that $\supp \chi^\mathcal{S}_{\tmkdv} \cap \mathcal{S} = \supp \chi^\mathcal{T}_{\tmkdv} \cap \mathcal{T} = \emptyset$.  Then, we can write
\begin{subequations}\begin{align}
    \rmIV_2 =& -t E \hat{T}_{\partial_\xi \phi_{\tmkdv} \chi^\mathcal{S}_{\tmkdv} e^{it\phi_{mKdV}} m^{R,2}}(\theta, \theta, \overline{\theta})\label{eqn:m-R-2-d-snr}\\
    &- t E \hat{T}_{\partial_\xi \phi_{\tmkdv} \chi^\mathcal{T}_{\tmkdv} e^{it\phi_{mKdV}} m^{R,2}}(\theta, \theta, \overline{\theta})\label{eqn:m-R-2-d-tnr}
\end{align}\end{subequations}
We first consider the space non-resonant terms~\eqref{eqn:m-R-2-d-snr}.  By integrating by parts in frequency using the identity
\begin{equation*}
    \frac{1}{it|\deles \phi_{\tmkdv}|^2}\deles \phi_{\tmkdv} \cdot \deles e^{it\phi_{\tmkdv}} = e^{it\phi_{\tmkdv}}
\end{equation*}
and noting that $\mathcal{F}^{-1}(Ee^{-it\xi^3/3}) = \invLinProp $, we find that
\begin{subequations}\begin{align}
    \mathcal{F}^{-1} \eqref{eqn:m-R-2-d-snr} =& \invLinProp  T_{\deles \cdot \left(\deles \phi_{\tmkdv} \frac{\partial_\xi \phi_\tmkdv m^{R,2} \chi^\mathcal{S}_\tmkdv}{|\deles \phi_{\tmkdv}|^2}\right)}(\fs, \fs, \fs)\label{eqn:m-R-2-d-snr-symb-deriv}\\
    &+ \invLinProp  T_{\frac{\partial_{\eta}\phi_\tmkdv \partial_\xi \phi_{\tmkdv m^{R,2} \chi^\mathcal{S}_{\tmkdv}}}{|\deles \phi_\tmkdv|^2}}(t|\fs|^2\fs, \fs, \fs)\label{eqn:m-R-2-d-snr-arg-deriv-1}\\
    &+ \bsim\notag
\end{align}\end{subequations}
Now, the symbol for the paraproduct in~\eqref{eqn:m-R-2-d-snr-symb-deriv} can be written as
\begin{equation*}
    \deles \cdot \left(\deles \phi_{\tmkdv} \frac{\partial_\xi \phi_\tmkdv m^{R,2} \chi^\mathcal{S}_\tmkdv}{|\deles \phi_{\tmkdv}|^2}\right) = \sum_{2^j \sim t^{-1/5}} m_j
\end{equation*}
so we can estimate the corresponding pseudoproduct using the same argument we used from $\rmII_2$. Similarly,~\eqref{eqn:m-R-2-d-snr-arg-deriv-1} can be estimated in the same manner as $\rmIII_2$.  Thus, it only remains to control the time non-resonant term~\eqref{eqn:m-R-2-d-tnr}.  In order to control~\eqref{eqn:ng-deriv-eqn-error}, we must bound
\begin{equation*}
    2 \Re \int_1^t \left\langle \eqref{eqn:m-R-2-d-tnr}, ng_n \right\rangle\;ds = -2\Re \int_{1}^t s \left\langle \mathcal{F}^{-1} \left[ E \hat{T}_{\mu_s^\tmkdv\phimKdV e^{is\phimKdV}}(\theta, \theta, \overline{\theta})\right], n g_n \right\rangle\;ds
\end{equation*}
where $\mu_s^\tmkdv$ is the time-dependent symbol given by
\begin{equation*}
    \mu_s^\tmkdv(\xi,\eta,\sigma) = \frac{\partial_\xi \phimKdV \chi^\cT_{\tmkdv} m^{R,2}}{\phimKdV}
\end{equation*} 
Integrating by parts in time, we find that
\begin{subequations}\begin{align}
    2 \Re \int_1^t \left\langle \eqref{eqn:m-R-2-d-tnr}, ng_n \right\rangle\;ds =& 2\Im \left. s\left\langle e^{-is\Psi} T_{\mu_s^\tmkdv}(\fs, \fs, \overline{\fs}) , n g_n\right\rangle\;ds\right|_{s = 1}^{s = t}\label{eqn:m-R-2-d-tnr-a}\\
    &+ 2\Im \int_1^t \left\langle e^{-is\Psi} T_{\mu_s^\tmkdv}(\fs, \fs, \overline{\fs}) , n g_n\right\rangle\;ds\label{eqn:m-R-2-d-tnr-b}\\
    &+ 2\Im \int_1^t s\left\langle \mathcal{F}^{-1}\left[\partial_s E \hat{T}_{\mu_s^\tmkdv e^{is\phimKdV}}(\theta, \theta, \overline{\theta}) \right], n g_n\right\rangle\;ds\label{eqn:m-R-2-d-tnr-c}\\
    &+ 2\Im \int_1^t s\left\langle e^{-is\Psi} T_{\partial_s\mu_s^\tmkdv }(\fs, \fs, \overline{\fs}) , n g_n\right\rangle\;ds\label{eqn:m-R-2-d-tnr-d}\\
    &+ 2\Im \int_1^t s\left\langle e^{-is\Psi} T_{\mu_s^\tmkdv}(\fs, \fs, s|\fs|^2\overline{\fs}), n g_n\right\rangle\;ds\label{eqn:m-R-2-d-tnr-e}\\
    &+ 2\Im \int_1^t s\left\langle e^{-is\Psi} T _{\mu_s^\tmkdv}(\fs, \fs,  \overline{\fs}), n \partial_s g_n\right\rangle\;ds\label{eqn:m-R-2-d-tnr-f}\\
    &+ \bsim\notag
\end{align}\end{subequations}
We now show how to bound the terms~\cref{eqn:m-R-2-d-tnr-a,eqn:m-R-2-d-tnr-b,eqn:m-R-2-d-tnr-c,eqn:m-R-2-d-tnr-d,eqn:m-R-2-d-tnr-e}.  The term~\eqref{eqn:m-R-2-d-tnr-f} will be considered in~\Cref{sec:final-ngn}.  Starting with~\eqref{eqn:m-R-2-d-tnr-a}, we note that the desired bound follows if we can show the pseudoproduct term decays sufficiently rapidly in $L^2$.  We have the estimate
\begin{equation}\label{eqn:mu-tmkdv-est}\begin{split}
    \lVert T_{\mu_s^\tmkdv }(\fs,\fs, \overline{\fs}) \rVert_{L^2} \leq& \sum_{2^j \sim s^{-1/5}} \lVert T_{m_k}(Q_{\sim j}^{\tcts} \fs, \fs, \overline{\fs})\rVert_{L^2} + \bsim\\
    \lesssim& \epsilon^3 s^{-\frac{9}{10}}
\end{split}\end{equation}
This immediately gives us the bounds
\begin{equation*}\begin{split}
    |\eqref{eqn:m-R-2-d-tnr-a}| \lesssim& \epsilon^3 t^{1/10} \lVert n g_n \rVert_{\ell^2}\\
    |\eqref{eqn:m-R-2-d-tnr-b}| \lesssim& \int_1^t \epsilon^3 s^{-9/10} \lVert ng_n \rVert_{\ell^2} \;ds
\end{split}\end{equation*}
which are compatible with~\eqref{eqn:n-gn-desired-est}.  The same argument also lets us control~\eqref{eqn:m-R-2-d-tnr-d}, since $s\partial_s \mu_s^\tmkdv$ satisfies the same support and symbol conditions as $\mu_s^\tmkdv$, and the estimate for~\eqref{eqn:m-R-2-d-tnr-c} is similar once we observe that $|s \partial_s E| \lesssim 1$ on the support of the pseudoproduct.
Turning to~\eqref{eqn:m-R-2-d-tnr-e}, we recall that
\begin{equation*}
    \partial_s \theta = e^{s\partial_x^3/3} \partial_x(|\fs|^2 \fs)
\end{equation*}
So we can write
\begin{equation*}
     sT_{\mu_s^\tmkdv e^{is\phimKdV}}(\theta, \theta, \partial_s \overline{\theta}) = e^{s\partial_x^3/3} \sum_{2^j \sim s^{-1/5}} s2^j T_{m_j}(\fs, \fs, |\fs|^2\overline{\fs})
\end{equation*}
These terms have the same form as those considered in $\rmII_2$, so we have the bound
\begin{equation*}
    |\eqref{eqn:m-R-2-d-tnr-d} | \lesssim  \int_1^t \epsilon^5 s^{-9/10} \lVert n g_n \rVert_{\ell^2}\;ds
\end{equation*}

\paragraph{\textbf{Case $m=1$}:} We now consider the bounds for~\eqref{eqn:R-cubic-nonlin-div} with $m = 1$.  Note that the symbol $m^{R,1}$ is supported on the region ${|\xi| \lesssim t^{-1/5} \lesssim \max(|\eta|, |\sigma|, |\xi - \eta - \sigma|)}$.  Thus, we can write
\begin{equation}\label{eqn:m-R-1-symbol-decomp}
    m^{R,1} = m^{R,1}_{\text{lo}} + m^{R,1}_{\text{hi,hi,lo}} + m^{R,1}_{\text{hi,hi,hi}}
\end{equation}
where the symbol $m^{R,1}_{\text{lo}}$ is supported in the region where $\max\{|\xi - \eta - \sigma|, |\eta|, |\sigma|\} \sim t^{-1/5}$, the symbol $m^{R,1}_{\text{hi,hi,lo}}$ is supported in the region where two of the input  frequencies are much greater than the third, that is, ${\max (|\xi - \eta - \sigma|, |\eta|, |\sigma|) \gg \min (|\xi - \eta - \sigma|, |\eta|, |\sigma|)}$, and the remaining symbol $m^{R,1}_{\text{hi,hi,hi}}$ is supported in the region where $|\xi-\eta-\sigma| \sim |\eta| \sim |\sigma| \gg t^{-1/5}$.

We begin with the argument for $\rmI_1$.  Using the decomposition~\eqref{eqn:m-R-1-symbol-decomp}, we can naturally decompose $\rmI_1$ as
\begin{equation*}\begin{split}
    \rmI_1 =& \partial_\xi E e^{-i\xi^3t} \hat{T}_{m^{R,1}_{\text{lo}}}(\fs, \fs, \overline{\fs}) + \partial_\xi E e^{-i\xi^3t} \hat{T}_{m^{R,1}_{\text{hi,hi,lo}}}(\fs, \fs, \overline{\fs}) + \partial_\xi E e^{-i\xi^3t} \hat{T}_{m^{R,1}_{\text{hi,hi,hi}}}(\fs, \fs, \overline{\fs})\\
    &= \rmI_1^{\text{lo}} + \rmI_1^{\text{hi,hi,lo}} + \rmI_1^{\text{hi,hi,hi}}
\end{split}\end{equation*}
The argument for $\rmI_1^{\text{lo}}$ is essentially the same as the one for $\rmI_2$, giving the bound
\begin{equation*}
    |\rmI_1^\text{lo}| \lesssim \epsilon^3 t^{-\frac{9}{10}}
\end{equation*}
We now turn to the other two terms.  For $\rmI_1^{\text{hi,hi,lo}}$, we write
\begin{equation*}
    \rmI_1^{\text{hi,hi,lo}} = \sum_{2^k \gg 2^j \gtrsim t^{-1/5}} t^{1/5} 2^k T_{m_{t^{-1/5}, j,k}}(\fs, \fs, \overline{\fs}) + \bsime
\end{equation*}
where $m_{t^{-1/5}, j, k}$ is supported on the region $|\xi| \lesssim t^{-1/5}$, $|\eta| \sim 2^j$, $|\sigma| \sim 2^k$, $|\xi - \eta - \sigma| \sim 2^k$ and satisfies the symbol bound
\begin{equation*}
    |\partial_\xi^a \partial_\eta^b \partial_\sigma^c m_{t^{-1/5}, k, j}| \lesssim_{a,b,c} t^{a/5} 2^{-bk} 2^{-cj}
\end{equation*}
In particular, $ m_{t^{-1/5}, j, k}$ satisfies the hypotheses of~\Cref{thm:pseudolocality-lemma}, so the main contribution from the pseudoproduct terms in comes from the regions $x \sim t 2^{2j}$ and $x \sim t 2^{2k}$:
\begin{equation*}\label{eqn:T-m-t-1-5-k-j-est}
\begin{split}
    \lVert T_{m_{t^{-\frac{1}{5}}, k, j}}(\fs, \fs, \overline{\fs}) \rVert_{L^{2}} \lesssim& \lVert T_{m_{t^{-1/5}, j, k}}(\chi_{[j-20,j+20]}^\tcts Q^\tcts_{\sim j}\fs, \chi_{[j-30,j+30]}^\tcts Q^\tcts_{\sim k}\fs, \chi_{[j-30,j+30]}^\tcts Q^\tcts_{\sim k}\overline{\fs}) \rVert_{L^{2}}\\
    &+  \lVert T_{m_{t^{-1/5, j, k}}}(\chi_{[k - 30, k+30]}^\tcts Q^\tcts_{\sim j}\fs, Q^\tcts_{\sim k} \fs, \chi_{[k - 20, k+20]}^\tcts Q^\tcts_{\sim k}\overline{\fs}) \rVert_{L^2}\\
    &+ \bbetter\\
    \lesssim& \epsilon^3 t^{-4/3} 2^{j/2}2^{-2k}
\end{split}\end{equation*}
Summing, we find that
\begin{equation*}
    \lVert \rmI_1^{\text{hi,hi,lo}} \rVert_{L^2} \lesssim \epsilon^3 t^{-\frac{31}{30}}
\end{equation*}

For the pseudoproduct $\rmI_1^{\text{hi,hi,hi}}$, the input frequencies $\eta$, $\sigma$ and $\xi - \eta - \sigma$ all have comparable magnitude.  A quick calculation shows that dividing dyadically in frequency space and estimating using~\Cref{thm:lin-decay-al} produces bounds which are too weak to give the desired estimate.  To resolve the problem, we introduce a more granular division in frequency.  This has two advantages. First, it lets us take better advantage of the low frequency cancellations from the $\partial_\xi E$ factor.  Second, it gives us some additional linear decay from the fact that all of the input frequencies cannot be exactly equal, which causes the terms of the pseudoproduct to concentrate in different places in physical space.  This second effect can be quantified using the following lemma:
\begin{lemma}\label{lem:fs-loc-est}
    Let $\chi_{\ell,m}(x) = \chi \left(\frac{x}{t2^{2\ell}} - m^2\right)$ be a bump function supported on $x =  t2^{2\ell}(m + O(1))^2$, and let $Q_{\ell,n}^\tcts$ be an operator which smoothly projects to frequencies $|\xi - n 2^\ell| \lesssim 2^\ell$.  Then, if $||n| - |m|| \gg 1$, 
    \begin{equation*}
        |\chi_{\ell,m} Q_{\ell,n}^\tcts \fs| \lesssim \epsilon t^{-5/6} |n^2 - m^2|^{-1} 2^{-3/2 \ell}
    \end{equation*}
\end{lemma}
\begin{proof}
    Recall that $\fs(x,t) = e^{-\frac{t}{3}\partial_x^3} \theta(x,t)$, with $\lVert \hat{\theta} \rVert_{L^\infty} \lesssim \epsilon$ and $\lVert x \theta \rVert_{L^2} \lesssim \epsilon^3 t^{1/6}$.  Thus, defining 
    \begin{equation*}
        \phi_x(\xi) = \frac{\xi^3}{3} + \xi \frac{x}{t}
    \end{equation*}
    and integrating by parts, we obtain
    \begin{equation*}\begin{split}
        Q_{\ell,n}^\tcts \fs(x,t) =& \frac{1}{\sqrt{2\pi}} \int e^{it\phi_x} \widehat{Q^\tcts_{\ell,n}}(\xi) \hat{\theta}(\xi)\;d\xi\\
        =& -\frac{i}{\sqrt{2\pi}t} \int \frac{\partial_\xi^2 \phi_x(\xi) }{[\partial_\xi \phi_x(\xi)]^2}e^{it\phi_x(\xi)} \widehat{Q^\tcts_{\ell,n}}(\xi) \hat{\theta}(\xi)\;d\xi\\
        &+ \frac{i}{\sqrt{2\pi}t} \int \frac{1}{\partial_\xi \phi_x(\xi)}e^{it\phi_x(\xi)} \partial_\xi \widehat{Q^\tcts_{\ell,n}}(\xi) \hat{\theta}(\xi)\;d\xi\\
        &+ \frac{i}{\sqrt{2\pi}t} \int \frac{1}{\partial_\xi \phi_x(\xi)}e^{it\phi_x(\xi)}  \widehat{Q^\tcts_{\ell,n}}(\xi) \partial_\xi \hat{\theta}(\xi)\;d\xi
    \end{split}\end{equation*}
    Now, since $||n| - |m|| \gg 1$, we find that
    \begin{equation*}\begin{split}
        |\partial_\xi^2 \phi_x(\xi)| \sim& 2^\ell \jBra{n} \ll 2^{\ell} |n^2 - m^2|\\
        |\partial_\xi \phi_x(\xi)| \sim& 2^{2\ell} |n^2 - m^2|\\
    \end{split}
    \end{equation*}
    combining these bounds with the bounds for $\theta$, we find that
    \begin{equation*}\begin{split}
        \lVert \chi_{\ell,m} Q_{\ell,n}^\tcts \fs(x,t) \rVert_{L^\infty_x} \lesssim& \epsilon \frac{t^{-5/6}}{|n^2-m^2|} \left( \frac{t^{-1/6} \jBra{n}}{|n^2 - m^2| 2^{2\ell}} + \frac{t^{-1/6}}{2^{2\ell}} + \epsilon^22^{-3/2\ell} \right)\\
        \lesssim& \epsilon t^{-5/6} |n^2 - m^2|^{-1} 2^{-3/2\ell}\qedhere
    \end{split}\end{equation*}
\end{proof}
For integers $\vec{n} = (n_1, n_2, n_3) \in \bbZ^3$, let $m^{\vec{n}}_{\ell}$ denote a generic symbol supported on the region
\begin{equation*}
    \left\{|\xi|, |\eta - 2^\ell n_1|, |\sigma - 2^\ell n_2|, |\xi - \eta - \sigma - 2^\ell n_3| \lesssim 2^\ell\right\}
\end{equation*} 
and satisfying $|\partial_{\xi,\eta,\sigma}^\alpha m^{\vec{n}}_{\ell}| \lesssim 2^{-|\alpha| \ell}$.  Then, we can write
\begin{subequations}\begin{align}
    \rmI_1^{\text{hi,hi,hi}} =& \sum_{\crampedsubstack{t^{-1/3} \lesssim 2^{\ell} \lesssim t^{-1/5}\\2^{k+\ell} \gg t^{-1/5}}}\sum_{\vec{n} \in S_k}t 2^{5\ell+k} \hat{T}_{m^{\vec{n}}_{\ell}}(\fs, \fs, \overline{\fs})\notag\\
    =& \sum_{\crampedsubstack{t^{-1/3} \lesssim 2^{\ell} \lesssim t^{-1/5}\\2^{k+\ell} \gg t^{-1/5}}}\sum_{\vec{n} \in S_k^{12,-}}t 2^{5\ell+k} \hat{T}_{m^{\vec{n}}_{\ell}}(\fs, \fs, \overline{\fs}) \label{eqn:n1-like-n2}\\
    &+ \sum_{\crampedsubstack{t^{-1/3} \lesssim 2^{\ell} \lesssim t^{-1/5}\\2^{k+\ell} \gg t^{-1/5}}}\sum_{\vec{n} \in S_k^{\text{rem}}}t 2^{5\ell+k} \hat{T}_{m^{\vec{n}}_{\ell}}(\fs, \fs, \overline{\fs})\label{eqn:all-n-different}\\
    &+\bsim\notag
\end{align}\end{subequations}
where
\begin{equation*}\begin{split}
    S_k =& \{(n_1, n_2, n_3) : |n_1| \sim |n_2| \sim |n_3| \sim 2^k \text{ and } |n_1 + n_2 - n_3| \lesssim 1\}\\
    S_k^{ij,e} =& \{(n_1, n_2, n_3) \in S_k : |n_i +e n_j| \lesssim 1\}\qquad \text{for } (i,j,e) = (1,2,-), (1,3,+), (23,+)\\
    S_k^{\text{rem}} =& \{(n_1, n_2, n_3) \in S_k : ||n_1| - |n_2|| \gg 1, ||n_1| - |n_3|| \gg 1, ||n_2| - |n_3|| \gg 1\}\\
    =& S_k \setminus \left( S_k^{12,-} \cup S_k^{13,+} \cup S_k^{23,+}\right)
\end{split}\end{equation*}
and $\bsim$ denotes terms like~\eqref{eqn:n1-like-n2} where the summation in $\vec{n}$ is taken over another set $S_k^{ij,\pm}$.  Let us first consider the sum~\eqref{eqn:n1-like-n2}.  Because of the frequency restriction, $n_2 = n_1 + O(1)$ and $n_3 = 2 n_1 + O(1)$.  Thus, by~\Cref{lem:fs-loc-est,thm:pseudolocality-lemma}, the leading order contribution to the pseudoproducts in~\eqref{eqn:n1-like-n2} will come from the regions $x = t 2^{2\ell} (n_1 + O(1))^2$ and $x = t2^{2\ell}(2n_1 + O(1))^2$.  Let us define
\begin{equation*}\begin{split}
    \chi_{\ell, n + O(1)} =& \sum_{m = n + O(1)} \chi_{\ell, m}\\
    Q^\tcts_{\ell,n + O(1)} =& \sum_{m = n + O(1)} Q^\tcts_{\ell,m}\\
\end{split}\end{equation*}
By combining the refined decay estimate from~\Cref{lem:fs-loc-est} with the fact that $\chi_{\ell, n_1 + O(1)}$ is supported on a region of size $t 2^{2\ell + k}$ in physical space, we see that
\begin{equation*}
    \lVert \chi_{\ell, n_1 + O(1)}Q^\tcts_{\ell,n_3 + O(1)} \fs \rVert_{L^2} \lesssim \epsilon t^{-1/3} 2^{-\frac{3}{2}k-\ell/2}
\end{equation*}
Combining this with the fact that $\vec{n} \in S^{12,-}_k$ is determined by $n_1$ up to $O(1)$ ambiguity, we find that
\begin{equation*}\begin{split}
    \lVert \eqref{eqn:n1-like-n2} \rVert_{L^2} \lesssim& \smashoperator[lr]{\sum_{\substack{t^{-\frac{1}{3}} \lesssim 2^{\ell} \lesssim t^{-\frac{1}{5}}\\2^{k+\ell} \gg t^{-\frac{1}{5}}}}} t 2^{5\ell+k} \sum_{\vec{n} \in S^{12,-}_k}  \lVert T_{m_{\ell}^{\vec{n}}} (Q^{\tcts}_{\sim k + \ell} \fs, Q^{\tcts}_{\sim k + \ell} \fs, \chi_{\ell, n_1 + O(1)} \overline{Q^\tcts_{\ell, n_3 + O(1)} \fs})\rVert_{L^2} + \bbetter\\
    \lesssim& \epsilon^3 t^{-\frac{1}{3}} \sum_{\crampedsubstack{t^{-\frac{1}{3}} \lesssim 2^\ell \lesssim t^{-\frac{1}{5}}\\2^{k+\ell} \gg t^{-\frac{1}{5}}}}\sum_{|n_1| \sim 2^k} 2^{\frac{7}{2}\ell - \frac{3}{2}k} + \bbetter\\
    \lesssim& \epsilon^3 t^{-\frac{31}{30}}
\end{split}\end{equation*}
We now turn to~\eqref{eqn:all-n-different}.  The main contributions here occur in the regions $\supp \chi_{\ell, n_j + O(1)}$ for $j = 1,2,3$.  Thus, we find that
\begin{equation*}\begin{split}
    \lVert \eqref{eqn:all-n-different} \rVert_{L^2} \lesssim& \smashoperator[r]{\sum_{\crampedsubstack{2^\ell \lesssim t^{-1/5}\\ 2^{k+\ell} \gg t^{-1/5}\\ \vec{n} \in S^{\text{rem}}_k}}} t 2^{5\ell + k} \lVert T_{m_\ell^{\vec{n}}} (Q^\tcts_{\sim k + \ell} \fs, \chi_{\ell, n_1 + O(1)} Q^\tcts_{\ell, n_2 + O(1)} \fs, \chi_{\ell, n_1 + O(1)} \overline{Q^\tcts_{\ell, n_3 + O(1)} \fs} ) \rVert_{L^2} + \bbetter\\
    \lesssim& \epsilon^3 t^{-2/3} \sum_{\substack{t^{-1/3} \lesssim 2^\ell \lesssim t^{-1/5}\\ 2^{k+\ell} \gg t^{-1/5}}} 2^{\frac{5}{2}\ell + k} \sum_{\vec{n} \in S^{\text{rem}}_k} \frac{1}{|n_1^2 - n_2^2|} \frac{1}{|n_2^2 - n_3^2|}
\end{split}\end{equation*}
We first consider the sum in $\vec{n}$.  For $i < j$
\begin{equation*}
    |n_i^2 - n_j^2| = M_{ij} m_{ij} \sim 2^k m_{ij}
\end{equation*}
where $m_{ij} = ||n_i| - |n_j||$ and $M_{ij} = |n_i| + |n_j|$.  Now,
\begin{equation*}
    m_{23} = |m_{12} \pm m_{13}|
\end{equation*}
and more generally knowledge any two of the quantities $m_{12}$, $m_{13}$ and $m_{23}$ determines the third up to finite ambiguity.  The condition $|n_1 - n_2 + n_3| \lesssim 1$ then allows us to determine $\vec{n}$ up to finite ambiguity.  Thus, we have the bound
\begin{equation*}\begin{split}
    \sum_{\vec{n} \in S^{\text{rem}}_k} \frac{1}{|n_1^2 - n_2^2|} \frac{1}{|n_2^2 - n_3^2|} \lesssim& 2^{-2k}\sum_{1 \ll m_{12}, m_{13} \lesssim 2^k}  \frac{1}{m_{12}m_{13}}\\
    \lesssim& k^2 2^{-2k}
\end{split}
\end{equation*}
which allows us to conclude that
\begin{equation*}\begin{split}
    \lVert \eqref{eqn:all-n-different} \rVert_{L^2} \lesssim& \epsilon^3 t^{-2/3} \sum_{\substack{t^{-1/3} \lesssim 2^\ell \lesssim t^{-1/5}\\ 2^{k+\ell} \gg t^{-1/5}}} k^2 2^{\frac{5}{2}\ell - k} \\
    \lesssim& \epsilon^3 t^{-\frac{7}{6}}
\end{split}\end{equation*}
which is more than sufficient, completing the argument for $\rmI_1$.

Let us now consider $\rmII_1$.  As in~\eqref{eqn:m-R-1-symbol-decomp}, we can write
\begin{equation*}
    \partial_\xi m^{R,1} = \partial_\xi m^{R,1}_{\text{lo}} + \partial_\xi m^{R,1}_{\text{hi,hi,lo}} + \partial_\xi m^{R,1}_{\text{hi,hi,hi}}
\end{equation*}
where $\partial_\xi m^{R,1}_{\text{lo}}$ is supported on the region where $\max\{|\xi|, |\eta|, |\sigma|\} \sim t^{-1/5}$, $\partial_\xi m^{R,1}_{\text{hi,hi,lo}}$ is supported in the region $|\xi| \sim t^{-1/5}, \max\{|\xi|, |\eta|, |\sigma|\} \gg \min\{|\xi|, |\eta|, |\xi - \eta -\sigma|\} \gtrsim t^{-1/5}$, and $\partial_\xi m^{R,1}_{\text{hi,hi,hi}}$ is supported on $|\xi| \sim t^{-1/5}$, $|\eta| \sim |\sigma| \sim |\xi - \eta - \sigma| \gg t^{-1/5}$.  
\begin{equation*}
    \rmII_1 = \rmII_1^\text{lo} + \rmII_1^\text{hi,hi,lo} + \rmII_1^\text{hi,hi,hi} 
\end{equation*}
As before, the estimates for the pseudoproduct $\rmII_1^{\text{lo}}$ follow from those for $\rmII_2$.  Turning to $\rmII_1^\text{hi,hi,lo}$, we can use~\eqref{eqn:T-m-t-1-5-k-j-est} to obtain the estimates
\begin{equation*}\begin{split}
    \lVert \rmII_1^{\text{hi,hi,lo}} \rVert_{L^2} =& \lVert T_{\partial_\xi m^{R,1}_\text{hi,hi,lo}}(\fs, \fs,\overline{\fs}) \rVert_{L^2}\\
    \lesssim& \sum_{2^k \gg 2^j \gtrsim t^{{1/5}}} t^{1/5} 2^k \lVert T_{m_{t^{-1/5},j,k}} (\fs, \fs, \overline{\fs}) \rVert_{L^2} + \bsime\\
    \lesssim& \epsilon^3 t^{-\frac{31}{30}}
\end{split}\end{equation*}
Now, we turn to the high frequency interactions.  Recalling that $\partial_\xi m^{R,1}_{\text{hi,hi,hi}}$ is supported in a region where $|\xi| \sim t^{-1/5}$, we find that
\begin{subequations}\begin{align}
    \lVert \rmII_1^{\text{hi,hi,hi}} \rVert_{L^2} \lesssim& \sum_{\substack{2^\ell \sim t^{-1/5}\\2^{k+\ell} \gg t^{-1/5}}} t^{1/5} 2^{k + \ell} \sum_{\vec{n} \in S_k^{12,-}} \lVert T_{m_\ell^{\vec{n}}}(\fs, \fs, \overline{\fs})\rVert_{L^2}\label{eqn:dxi-n1-like-n2}\\
    &+ \sum_{\substack{2^\ell \sim t^{-1/5}\\2^{k+\ell} \gg t^{-1/5}}} t^{1/5} 2^{k + \ell} \sum_{\vec{n} \in S_k^{\text{rem}}} \lVert T_{m_\ell^{\vec{n}}}(\fs, \fs, \overline{\fs})\rVert_{L^2}\label{eqn:dxi-all-n-different}\\
    &+ \bsim\notag
\end{align}\end{subequations}
By modifying the reasoning for~\cref{eqn:n1-like-n2,eqn:all-n-different}, we see that
\begin{equation*}\begin{split}
    \lVert \eqref{eqn:dxi-n1-like-n2} \rVert_{L^2} \lesssim& \epsilon^3 t^{-\frac{14}{15}}\\
    \lVert \eqref{eqn:dxi-all-n-different} \rVert_{L^2} \lesssim& \epsilon^3 t^{-7/6}
\end{split}\end{equation*}

We now turn to~$\rmIII_1$.  As above, we write
\begin{equation*}\begin{split}
    \rmIII_1 =& T_{m^{R,1}_{\text{lo}}}(\fs, t|\fs|^2\fs, \overline{\fs}) + T_{m^{R,1}_{\text{hi,hi,lo}}}(\fs, t|\fs|^2\fs, \overline{\fs}) + T_{m^{R,1}_{\text{hi,hi,hi}}}(\fs, t|\fs|^2\fs, \overline{\fs})\\
    =:& \rmIII_1^\text{lo} + \rmIII_1^\text{hi,hi,lo} + \rmIII_1^\text{hi,hi,hi}
\end{split}\end{equation*}
Again, the bound for $\rmIII_1^\text{lo}$ follows from the same reasoning used for $\rmIII_2$, so we are left to consider the last two terms.  For $\rmIII_1^\text{hi,hi,lo}$, we separate the sum where the cubic term $|\fs|^2\fs$ is localized at high frequency from the one where it is localized to a low frequency to write
\begin{subequations}\begin{align}
    \rmIII_1^\text{hi,hi,lo} =& \sum_{2^k \gg 2^j \gtrsim t^{-1/5}} t2^k T_{m_{t^{-1/5},k,j}}(\fProjc{\sim k}\fs, \fProjc{\sim k}|\fs|^2\fs, \fProjc{\sim j}\overline{\fs})\label{eqn:m-R-1-c-hhl-cubic-hi}\\
    &+ \sum_{2^k \gg 2^j \gtrsim t^{-1/5}} t2^k T_{m_{t^{-1/5},k,j}}(\fProjc{\sim k}\fs, \fProjc{\sim j}|\fs|^2\fs, \fProjc{\sim k}\overline{\fs})\label{eqn:m-R-1-c-hhl-cubic-lo}\\
    &+ \bsim\notag
\end{align}
\end{subequations}
For the high frequency case, based on the physical-space localization of the terms in the pseudoproduct, we see that the leading order contributions from each term in the sum comes in the regions $x \sim t2^{2k}$ and $x \sim t 2^{2j}$.  A straightforward calculation then yields the bound
\begin{equation*}\begin{split}
    \lVert \eqref{eqn:m-R-1-c-hhl-cubic-hi} \rVert_{L^2} \leq& \sum_{2^k \gg 2^j \gtrsim t^{-1/5}} \lVert t2^k \chi_{\sim k}^\tcts T_{m_{t^{-1/5},k,j}}(\fProjc{\sim k}\fs, \fProjc{\sim k}|\fs|^2\fs, \fProjc{\sim j}\overline{\fs}) \rVert_{L^2}\\
    &+ \sum_{2^k \gg 2^j \gtrsim t^{-1/5}} \lVert t2^k \chi_{\sim j}^\tcts T_{m_{t^{-1/5},k,j}}(\fProjc{\sim k}\fs, \fProjc{\sim k}|\fs|^2\fs, \fProjc{\sim j}\overline{\fs}) \rVert_{L^2}+ \bbetter\\
    \lesssim& \epsilon^5 t^{-\frac{37}{30}}
\end{split}\end{equation*}
Similarly, for the low frequency sum, we find that
\begin{equation*}\begin{split}
    \lVert \eqref{eqn:m-R-1-c-hhl-cubic-lo} \rVert_{L^2} \leq& \sum_{2^k \gg 2^j \gtrsim t^{-1/5}} \lVert t2^k \chi_{\sim k}^\tcts T_{m_{t^{-1/5},k,j}}(\fProjc{\sim k}\fs, \fProjc{\sim j}|\fs|^2\fs, \fProjc{\sim k}\overline{\fs}) \rVert_{L^2}\\
    &+ \sum_{2^k \gg 2^j \gtrsim t^{-1/5}} \lVert t2^k \chi_{\sim j}^\tcts T_{m_{t^{-1/5},k,j}}(\fProjc{\sim k}\fs, \fProjc{\sim j}|\fs|^2\fs, \fProjc{\sim k}\overline{\fs}) \rVert_{L^2}+ \bbetter\\
    \lesssim& \epsilon^5 t^{-\frac{11}{10}}
\end{split}
\end{equation*}
For $\rmIII_1^\text{hi,hi,hi}$, we note that
\begin{equation*}\begin{split}
    \lVert \rmIII_1^\text{hi,hi,hi} \rVert_{L^2} =& \biggl\lVert\sum_{2^j \gg t^{-1/5}} t 2^j T_{m_j}(Q_{\sim j}^\tcts \fs, Q_{\sim j}^\tcts (|\fs|^2 \fs), \overline{Q_{\sim j}^\tcts \fs}) \biggr \rVert_{L^2}\\
    \lesssim& \sum_{2^j \gg t^{-1/5}} t 2^j \lVert Q_{\sim j}^\tcts \fs \rVert_{L^4}^2 \lVert Q^\tcts_{\sim j} (|\fs|^2\fs) \rVert_{L^\infty}\\
    \lesssim& \epsilon^5 t^{-\frac{9}{10}}
\end{split}\end{equation*}

Finally, we turn to $\rmIV_1$.  Here, we must take advantage of the space-time resonance structure of the nonlinearity. Note that by the definition of $\rho$, $\supp m^{R,1}$ contains no space resonances, so we can freely integrate by parts in frequency.  Doing so, we obtain
\begin{subequations}\begin{align}
    \rmIV_1 =& T_{\nablaes \cdot \mu^{R,1}}(\fs,\fs,\overline{\fs}) \label{eqn:m-R-1-d-1}\\
    &+ tT_{\mu^{R,1}_\eta}(\fs,|\fs|^2\fs,\overline{\fs}) \label{eqn:m-R-1-d-2}\\
    &+\bsim\notag
\end{align}\end{subequations}
where 
\begin{equation*}
    \mu^{R,1} = \frac{\partial_\xi \phimKdV \nablaes \phimKdV m^{R,1}}{|\nablaes \phimKdV|^2}
\end{equation*}
is a vector valued symbol with components $\mu^{R,1}_\eta$ and $\mu^{R,1}_\sigma$. Since $\nablaes \cdot \mu^{R,1}$ obeys the same bounds as $\partial_\xi m^{R,1}$, the term~\eqref{eqn:m-R-1-d-1} can be handled using the same reasoning as for~$\rmII_1$.  Similarly, since $\mu^{R,1}$ has the same support and decay properties as $m^{R,1}$, the bounds for~\eqref{eqn:m-R-1-d-2} follow from the same arguments used to bound~$\rmIII_1$.

\subsection{Estimates and cancellations for the remaining terms}\label{sec:final-ngn}	
The only remaining terms are~\eqref{eqn:tnr-deriv-hits-IP} and \eqref{eqn:m-R-2-d-tnr-f}.  Defining
\begin{equation*}\begin{split}
    T_n(s) =& T_{\mu_s^\cT}(u_n, w_n, \overline{u_n}) + T_{\mu_s^\cT}(u_n, \mSS_n, \overline{w_n}) + T_{\mu_s^\cT}(w_n, \mSS_n, \overline{u_n})\\
    &+ ie^{-\frac{i}{3}\xi^3}T_{\mu_s^\tmkdv}(\fs^+,\fs^+,\overline{\fs^+})(n + 2t) + ie^{-\frac{i}{3}\xi^3}T_{\mu_s^\tmkdv}(\fs^-,\fs^-,\overline{\fs^-})(n - 2t)
\end{split}\end{equation*}
we can use the preceding arguments to write
\begin{subequations}\begin{align}
    \eqref{eqn:tnr-deriv-hits-IP} + \eqref{eqn:m-R-2-d-tnr-f} =& \int_1^t s \Re\left\langle T_n(s), e^{is\Psi} n \partial_s g_n(s)\right\rangle\;ds\notag\\
    =&  \int_1^t s \Im\left\langle T_n(s), |u_n|^2 \Psi Lw_n \right\rangle\;ds\label{eqn:cancel-terms-psiLw}\\
    &+ \int_1^t s^2 \Re\left\langle T_n(s), \tilde{T}_n(s) \right\rangle\;ds\label{eqn:cancel-terms-t-tilde}\\
    &+ \int_1^t s \Re\left\langle T_n(s), H_n(s) \right\rangle\;ds\label{eqn:cancel-terms-easy}
\end{align}\end{subequations}
where
\begin{equation*}\begin{split}
    \tilde{T}_n(s) =& e^{is\Psi}\left(T_{\mu_s^\cT \phi e^{is\phi}}(f_n, g_n, \overline{f_n}) + T_{\mu_s^\cT \phi e^{is\phi}}(f_n, h_n, \overline{g_n}) + T_{\mu_s^\cT \phi e^{is\phi}}(g_n, h_n, \overline{f_n})\right.\\
    &\qquad\qquad\left.+ i T_{\mu_s^\tmkdv \phimKdV e^{is\phimKdV}}(\theta^+,\theta^+,\overline{\theta^+})(n + 2t)\right.\\
    &\qquad\qquad \left.+ i T_{\mu_s^\tmkdv \phimKdV e^{is\phimKdV}}(\theta^-,\theta^-,\overline{\theta^-})(n - 2t)\right)
\end{split}\end{equation*}
represents the terms in $e^{is\Psi} n\partial_s g$ which we handled by integrating by parts in time, and $H_n(s)$ collects the remaining terms.  By our prior arguments, the terms in $H_n(s)$ decay pointwise in time, giving us the bound
\begin{equation*}
    \lVert H_n(s) \rVert_{\ell^2} \lesssim M^2\epsilon^2(1 + \epsilon^2) s^{-1} \lVert ng_n \rVert_{\ell^2} + \epsilon (1 + \epsilon^2) s^{-\frac{9}{10}}
\end{equation*}
By combining~\eqref{eqn:T-mu-s-bd} and~\eqref{eqn:mu-tmkdv-est}, we see that
\begin{equation}\label{eqn:T-cancel-sec-bd}
    \lVert T_n(s) \rVert_{\ell^2} \lesssim M^2\epsilon^2 s^{-1}\lVert ng_n \rVert_{\ell^2} + \epsilon^3 s^{-\frac{9}{10}}
\end{equation}
so, by Cauchy-Schwarz,
\begin{equation*}
    |\eqref{eqn:cancel-terms-easy}| \lesssim M^2 \epsilon^4 t^{1/5} + \int_1^t M^2\epsilon^2 s^{-1} \lVert ng_n \rVert_{\ell^2}^2\;ds
\end{equation*}
which is consistent with~\eqref{eqn:n-gn-desired-est}.  Turning to~\eqref{eqn:cancel-terms-psiLw}, we can use the self-adjointness of $\Psi$ and Cauchy-Schwarz to write
\begin{equation*}
    |\eqref{eqn:cancel-terms-psiLw}| \lesssim \int_1^t s \bigl\lVert \Psi \bigl( |u_n|^2 T_n(s)\bigr) \bigr\rVert_{\ell^2} \lVert ng_n \rVert_{\ell^2}\;ds
\end{equation*}
Dividing dyadically in frequency, we can write
\begin{subequations}\begin{align}
    \left\lVert\Psi \left( |u_n|^2 T_n(s)\right)\right\rVert_{\ell^2} \lesssim& \left\lVert\sum_{j} \left(\sum_{2^k \gtrsim 2^j} 2^j Q_j^+  |Q_{\sim k}^+ u_n|^2 Q_k^+ T_{\mu_s^\cT}(u_n, w_n, \overline{u_n}) \right)\right\rVert_{\ell^2}\label{eqn:cancel-terms-psi-Lw-str}\\
    &+ \left\lVert\sum_{j} \sum_{2^k \gtrsim 2^j} 2^j  Q_j^+ \left( |Q_{\sim k}^+ u_n|^2 Q_k T_{\mu_s^\tmkdv}(\fs^+, \fs^+, \overline{\fs^+}) (n + 2t) \right)\right\rVert_{\ell^2}\label{eqn:cancel-terms-psi-Lw-ss}\\
    &+ \bsime\notag
\end{align}\end{subequations}
For~\eqref{eqn:cancel-terms-psi-Lw-str}, we will focus on the contribution from $\supp \chi_{\sim k}^+$, since the remaining contribution can be handled more easily using the decay estimates for $u_n$ with~\eqref{eqn:T-mu-s-bd}.  We can write
\begin{equation*}\begin{split}
    Q_k^+ T_{\mu_s^\cT}(u_n, w_n, \overline{u_n}) =& \sum_{2^\ell \gtrsim 2^k} T_{m_\ell}(Q_\ell^+ u_n, w_n, \overline{u_n}) + T_{m_\ell}(u_n, Q_\ell^+ w_n, \overline{u_n}) + \bsim
\end{split}\end{equation*}
and a quick calculation shows that
\begin{equation*}
    \lVert \chi_{\sim k}^+ T_{m_\ell}(Q_{\sim \ell}^+ u_n, w_n, \overline{u_n}) \rVert_{\ell^2} + \lVert \chi_{\sim k}^+ T_{m_\ell}(u_n, Q_{\sim \ell}^+ w_n, \overline{u_n}) \rVert_{\ell^2} \lesssim \begin{cases}
        M^2\epsilon^2 s^{-1} \lVert n g_n \rVert_{\ell^2} c_k & |k - \ell| \leq C\\
        M^2\epsilon^2 s^{-4/3} 2^{k - 2\ell} & \ell > k + C
    \end{cases}
\end{equation*}
which gives the estimate
\begin{equation}
    \lVert \chi_{\sim k}^+ Q_k^+ T_{\mu_s^\cT}(u_n, w_n, \overline{u_n}) \rVert_{\ell^2} \lesssim M^2 \epsilon^2 s^{-1} \lVert ng_n \rVert_{\ell^2} c_k
\end{equation}
Combining this estimates with the decay estimates for $Q_{\sim k}^+ u_n$, summing in $\ell$, and using almost orthogonality and pseudolocality, we find that
\begin{equation*}\begin{split}
    \eqref{eqn:cancel-terms-psi-Lw-str}\lesssim& \biggl\lVert\sum_{j} \sum_{2^k \gtrsim 2^j} 2^j \chi_{\sim k}^+ Q_j^+\bigl(  |Q_{\sim k}^+ u_n|^2 Q_k^+ T_{\mu_s^\cT}(u_n, w_n, \overline{u_n})\bigr) \biggr\rVert_{\ell^2} + \bbetter\\
    \lesssim& \left(\sum_{j} \sum_{2^k \gtrsim 2^j} 2^{2j} \lVert Q^+_{\sim k} u_n \rVert_{\ell^\infty}^4 \lVert \chi_{\sim k}^+ T_{\mu_s^\cT}(u_n, w_n, \overline{u_n})\rVert_{\ell^2}^2\right)^{1/2} + \bbetter\\
    \lesssim& M^4 \epsilon^4 s^{-2} \lVert ng_n \rVert_{\ell^2}
\end{split}\end{equation*}
The argument for~\eqref{eqn:cancel-terms-psi-Lw-ss} is similar once we notice that
\begin{equation*}
    \lVert \chi_{\sim k}^\tcts Q_k^\tcts T_{\mu_s^\tmkdv}(\fs^+, \fs^+, \overline{\fs^+}) \rVert_{L^2} \lesssim \begin{cases}
        \epsilon^3 s^{-2/5} & 2^k \sim s^{-1/5}\\
        0 & \text{else}
        \end{cases}
\end{equation*}

The only remaining term is~\eqref{eqn:cancel-terms-t-tilde}.  Observe that
\begin{equation*}
    \partial_s T_n(s) = \tilde{T}_n(s) + O_{\ell^2}\bigl(M^2\epsilon^2 s^{-2} \lVert ng_n \rVert_{\ell^2} + M^2 \epsilon^3 s^{-\frac{19}{10}}\bigr)
\end{equation*}
Thus, using~\eqref{eqn:T-cancel-sec-bd}, we see that
\begin{equation*}\begin{split}
    \frac{1}{2}\partial_s \lVert T_n(s) \rVert_{\ell^2}^2 =& \Re \langle T_n(s), \tilde{T}_n(s) \rangle + O\bigl(M^4\epsilon^4 s^{-3} \lVert ng_n \rVert_{\ell^2} + M^2 \epsilon^6 s^{-\frac{9}{5}} \bigr)
\end{split}\end{equation*}
which gives us the bound
\begin{equation*}\begin{split}
    |\eqref{eqn:cancel-terms-t-tilde}| \lesssim& \int_1^t \frac{s^2}{2} \partial_s \lVert T_n(s) \rVert_{\ell^2}^2\;ds  + \int_1^t M^4\epsilon^4 s^{-1} \lVert ng_n \rVert_{\ell^2}^2\;ds + M^2\epsilon^6 t^{1/5}\\
    \lesssim& \left. s^2 \lVert T_n(s) \rVert_{\ell^2}^2 \right|_{s = 1}^{s=t} + \int_1^t s \lVert T_n(s) \rVert_{\ell^2}^2\;ds + \int_1^t M^4\epsilon^4 s^{-1} \lVert ng_n \rVert_{\ell^2}^2\;ds + M^2\epsilon^6 t^{1/5}\\
    \lesssim& M^2\epsilon^2 \lVert ng_n(t) \rVert_{\ell^2}^2 + M^2 \epsilon^4 t^{1/5} + \int_1^t M^2\epsilon^2 s^{-1} \lVert ng_n\rVert_{\ell^2}^2\;ds
\end{split}\end{equation*}
This completes the proof of~\eqref{eqn:n-gn-desired-est}.

\section{Decay at the degenerate frequencies\label{sec:degen-freq-decay}}

We now consider the problem of estimating $\partial_t \alpha_{\pm}$.  For notational simplicity, we will focus on the estimate for $\partial_t \alpha_+$, since the estimate for $\partial_t \alpha_{-}$ is similar.  We begin by observing that $\partial_t \alpha_+$ can be written as a sum:
\begin{equation}\label{eqn:dt-alpha-+-sum}
    \partial_t \alpha_+ = \sum_{n \in \bbZ} e^{-i\frac{\pi}{2}n} |u_n|^2 \Psi u_{n}
\end{equation}
\subsection{Reduction to estimates involving \texorpdfstring{$w_n$}{w\_n}}
As discussed in the introduction, the linear decay estimates are not sufficient to bound~\eqref{eqn:dt-alpha-+-sum}.  Instead, we will separate out the contribution from the approximately self-similar solution to write
\begin{subequations}\begin{align}
    \partial_t \alpha_+ =& \sum_{n \in \bbZ} e^{-i\frac{\pi}{2}n} |u_n|^2 \Psi w_{n}\label{eqn:dt-a-Psi-w}\\
    &+ \sum_{n \in \bbZ} e^{-i\frac{\pi}{2}n} 2\Re(u_n \overline{w}_n)\Psi \mSS_{n}\label{eqn:dt-a-Psi-SS}\\
    &+ \sum_{n \in \bbZ}e^{-i\frac{\pi}{2}n} |\mSS_n|^2\Psi \mSS_{n} \label{eqn:dt-a-all-SS}
\end{align}\end{subequations}
We begin by considering the third term.  By using Fourier inversion, we find that
\begin{subequations}\begin{align}
    \eqref{eqn:dt-a-all-SS} 
    =& 2 e^{i\pi t}\iint_{\bbR^2} \SSCutoff[\eta]\hat{\fs}(-\eta) \SSCutoff[\sigma]\hat{\fs}(-\sigma) \SSCutoff[(\eta + \sigma)]\overline{\hat{\fs}}(-\eta-\sigma) \sin(\eta + \sigma) \;d\eta d\sigma\nonumber\\
    =& -2e^{i\pi t} \iint_{\bbR^2} \SSCutoff[\eta]\hat{\fs}(\eta) \SSCutoff[\sigma]\hat{\fs}(\sigma) \SSCutoff[(\eta + \sigma)]\overline{\hat{\fs}}(\eta+\sigma) (\eta + \sigma) \;d\eta d\sigma \label{eqn:degen-freq-III-cancel}\\
    &-2e^{i\pi t} \iint_{\bbR^2} \SSCutoff[\eta]\hat{\fs}(\eta) \SSCutoff[\sigma]\hat{\fs}(\sigma) \SSCutoff[(\eta+\sigma)]\overline{\hat{\fs}}(\eta+\sigma) (\sin(\eta + \sigma) - (\eta + \sigma)) \;d\eta d\sigma \label{eqn:degen-freq-III-small}
\end{align}\end{subequations}
For the first term~\eqref{eqn:degen-freq-III-cancel}, using Fourier inversion and the fact that $\SSCutoff[D]\fs$ has constant argument, we find
\begin{equation*}\begin{split}
    \eqref{eqn:degen-freq-III-cancel}  =& -i\int (\SSCutoff[D] \fs(x))^2 \SSCutoff[D] \partial_x \overline{\fs}(x)\;dx\\
    =& -i\int \partial_x\left((\SSCutoff[D] \fs(x))^2 \SSCutoff[D] \overline{\fs}(x)\right)\;dx\\
    =& 0
\end{split}\end{equation*}
Thus,
\begin{equation}\begin{split}
    |\eqref{eqn:dt-a-all-SS}| =& \left|\eqref{eqn:degen-freq-III-small}\right|\\
    \lesssim& \lVert (\SSCutoff[D] \fs(x))^2 (\SSCutoff[D](\sin(D) - D) \overline{\fs}(-x)) \rVert_{L^1}\\
    \lesssim& \lVert \SSCutoff[D] \fs(x) \rVert_{L^2} \sum_{2^j \lesssim t^{-1/5}} 2^{3j} \lVert P_j \fs \rVert_{L^4}^2\\
    \lesssim& \epsilon^3 t^{-6/5}
\end{split}\end{equation}

\subsection{Estimates involving \texorpdfstring{$w_n$}{w\_n}}

We now turn to estimating~\eqref{eqn:dt-a-Psi-w} (the argument for~\eqref{eqn:dt-a-Psi-SS} is similar).  By inverting the Fourier transform, we find that
\begin{equation*}
    \eqref{eqn:dt-a-Psi-w} = \iint m(\eta, \sigma) e^{it\phi} \hat{g}(\fhPi-\eta-\sigma,t) \hat{f}(\eta) \overline{\hat{f}}(-\sigma) \;d\eta d\sigma
\end{equation*}
with
\begin{equation*}
    m(\eta,\sigma) = \cos\left(\fhPi - \eta - \sigma\right)
\end{equation*}
The main contribution to $\rmI$ comes from the three points where $\nabla \phi$ vanishes.  These points are given by
\begin{equation*}\begin{split}
    p_1 = (\eta_1,\sigma_1) = \left(\fhPi, -\fhPi\right)\\
    p_2 = (\eta_2,\sigma_2) = \left(-\frac{5\pi}{6}, \frac{\pi}{6}\right)\\
    p_3 = (\eta_3,\sigma_3) = \left(-\frac{\pi}{6}, \frac{5\pi}{6}\right)
\end{split}\end{equation*}
Notice that the symbol $m$ vanishes at all of these points.  At the points $p_2$ and $p_3$, $\phi$ has a non-degenerate Hessian, so we can perform the usual quadratic stationary phase approximate.  At $p_1$, the Hessian of $\phi$ vanishes, which produces slower decay. However, this is compensated by the fact that $\hat{g}\left(\fhPi - \eta - \sigma\right)$ also vanishes at $p_1$, giving us an integrable decay rate.

We now rewrite~\eqref{eqn:dt-a-Psi-w} in a form that isolates contribution from each stationary point.  Let $\chi_r(\eta,\sigma)$, $r=1,2,3$ be smooth functions with disjoint support such that $\chi_r$ is equal to $1$ in some neighborhood of $p_r$ and vanishes in a neighborhood of $p_m$ for each $m \neq r$.  Then, we can write 
\begin{subequations}\begin{align}
    \eqref{eqn:dt-a-Psi-w} =& \iint m(\eta,\sigma) \chi_1(\eta,\sigma) e^{is\phi} \hat{g}(\fhPi-\eta-\sigma,s) \hat{f}(\eta,s) \overline{\hat{f}}(-\sigma,s) \;d\eta d\sigma\label{eqn:p1-st-pt}\\
    &+ \iint m(\eta,\sigma) \chi_2(\eta,\sigma) e^{is\phi} \hat{g}(\fhPi-\eta-\sigma,s) \hat{f}(\eta,s) \overline{\hat{f}}(-\sigma,s) \;d\eta d\sigma\label{eqn:p2-st-pt}\\
    &+ \iint m(\eta,\sigma) \chi_3(\eta,\sigma) e^{is\phi} \hat{g}(\fhPi-\eta-\sigma,s) \hat{f}(\eta,s) \overline{\hat{f}}(-\sigma,s) \;d\eta d\sigma\label{eqn:p3-st-pt}\\
    &+ \iint m(\eta,\sigma) \bigl(1-\sum_{r=1}^3\chi_3(\eta,\sigma)\bigr) e^{is\phi} \hat{g}(\fhPi-\eta-\sigma,s) \hat{f}(\eta,s) \overline{\hat{f}}(-\sigma,s) \;d\eta d\sigma\label{eqn:nonstat-rgn}
\end{align}\end{subequations}
We will show how to control~\cref{eqn:p1-st-pt,eqn:p2-st-pt,eqn:p3-st-pt}: the contribution from~\eqref{eqn:nonstat-rgn} is more easily controlled, since the phase is nonstationary.

\subsubsection{The contribution from the quadratic stationary points}

We first show how to control~\eqref{eqn:p2-st-pt} (with a similar argument giving us the estimate for~\eqref{eqn:p3-st-pt}).  Let us write $\chi_{2} = \sum_{\ell_1,\ell_2} \chi_{2,\ell_1,\ell_2}$, where
\begin{equation*}
    \chi_{2,\ell_1,\ell_2}(\eta,\sigma) = \chi_2(\eta,\sigma) \tilde{\psi}^{[s^{-4/9}]}_{\ell_1}(\eta - \eta_2) \tilde{\psi}^{[s^{-4/9}]}_{\ell_1}(\sigma - \sigma_2)
\end{equation*}
and let $\ell_0 = \ell_0(s)$ be such that $2^{\ell_0} < s^{-4/9} \leq 2^{\ell_0+1}$.  Then, we can further decompose~\eqref{eqn:p2-st-pt} as
\begin{subequations}\begin{align}
    \eqref{eqn:p2-st-pt} =&  \iint m(\eta,\sigma) \chi_{2,\ell_0,\ell_0}(\eta,\sigma) e^{is\phi} \hat{g}(\fhPi-\eta-\sigma,s) \hat{f}(\eta,s) \overline{\hat{f}}(-\sigma,s) \;d\eta d\sigma\label{eqn:p2-st-pt-lo}\\
    &+ \sum_{\substack{\ell_1 > \min\{\ell_0+1, \ell_2\}}}\iint m(\eta,\sigma) \chi_{2,\ell_1,\ell_2}(\eta,\sigma) e^{is\phi} \hat{g}(\fhPi-\eta-\sigma,s) \hat{f}(\eta,s) \overline{\hat{f}}(-\sigma,s) \;d\eta d\sigma\label{eqn:p2-st-pt-hi}\\
    &+ \bsim\notag
\end{align}\end{subequations}
For~\eqref{eqn:p2-st-pt-lo}, a simple volume bound combined with the smallness of $m$ over the support of $\chi_{2,\ell_0,\ell_0}$ yields
\begin{equation*}
    |\eqref{eqn:p2-st-pt-lo}| \lesssim s^{-4/3} \lVert \hat{f} \rVert_{L^\infty}^2\lVert \hat{g} \rVert_{L^\infty} \lesssim M^2\epsilon^3 s^{-\frac{37}{30}} 
\end{equation*}
where to obtain the final estimate, we have used the fact that 
$$\lVert \hat{g} \rVert_{L^\infty} \lesssim \lVert ng_n \rVert_{L^2} \lesssim \epsilon t^{\frac{1}{10}}$$
with the last inequality coming from the work in~\Cref{sec:n-ell2}.  Turning to~\eqref{eqn:p2-st-pt-hi}, we integrate by parts to obtain
\begin{subequations}\begin{align}
    \eqref{eqn:p2-st-pt-hi} =& \sum_{\substack{\ell_1 > \min\{\ell_0+1, \ell_2\}}} \frac{i}{s} \iint \deles \cdot \Bigl(\deles \phi \frac{m \chi_{2,\ell_1,\ell_2}}{|\deles \phi|^2}\Bigr) e^{is\phi} \hat{g}(\fhPi-\eta-\sigma) \hat{f}(\eta) \overline{\hat{f}}(-\sigma) \;d\eta d\sigma \label{eqn:p2-st-pt-hi-1}\\
    &+ \sum_{\substack{\ell_1 > \min\{\ell_0+1, \ell_2\}}}  \frac{i}{s} \iint \partial_\eta \phi \frac{m \chi_{2,\ell_1,\ell_2}}{|\deles \phi|^2} e^{is\phi} \hat{g}(\fhPi-\eta-\sigma) \partial_\eta \hat{f}(\eta) \overline{\hat{f}}(-\sigma) \;d\eta d\sigma \label{eqn:p2-st-pt-hi-2}\\
    &- \sum_{\substack{\ell_1 > \min\{\ell_0+1, \ell_2\}}}  \frac{i}{s} \iint \partial_\eta \phi \frac{m \chi_{2,\ell_1,\ell_2}}{|\deles \phi|^2} e^{is\phi} \partial_\eta \hat{g}(\fhPi-\eta-\sigma) \partial_\eta \hat{f}(\eta) \overline{\hat{f}}(-\sigma) \;d\eta d\sigma \label{eqn:p2-st-pt-hi-3}\\
    &+ \bsim \notag
\end{align}
\end{subequations}
Let us write
\begin{equation*}\begin{split}
    m^1_{2,\ell_1,\ell_2} = \psi^+_{\leq \ell_2}(\xi) \deles \cdot \Bigl(\deles \phi \frac{m \chi_{2,\ell_1,\ell_2}}{|\deles \phi|^2}\Bigr), \qquad\qquad
    m^2_{2,\ell_1, \ell_2} = \psi^+_{\leq \ell_2}(\xi) \partial_\eta \phi \frac{m \chi_{2,\ell_1,\ell_2}}{|\deles \phi|^2}
\end{split}\end{equation*}
The symbols $2^{\ell_1} m^1_{2,\ell_1,\ell_2}$ and $m^2_{2,\ell_1, \ell_2}$ satisfy the hypotheses in~\Cref{rmk:freq-loc-symbol-bds}, so by Hausdorff-Young,
\begin{equation*}\begin{split}
    |\eqref{eqn:p2-st-pt-hi}| \leq& |\eqref{eqn:p2-st-pt-hi-1}| + |\eqref{eqn:p2-st-pt-hi-2}| + |\eqref{eqn:p2-st-pt-hi-3}| + \bsime\\
    \leq& \sum_{\substack{\ell_1 > \min\{\ell_0+1, \ell_2\}}} s^{-1}\bigl\lVert T_{m^1_{\ell_1,\ell_2}}(w_n, \psi_{\ell_1}(D - \eta_2) u_n,\overline{\psi_{\ell_2}^{[s^{-4/9}]}(D - \sigma_2)u_n})\bigr\rVert_{\ell^1}\\
    &+ \sum_{\substack{\ell_1 > \min\{\ell_0+1, \ell_2\}}} s^{-1}\bigl\lVert T_{m^1_{\ell_1,\ell_2}}(w_n, L u_n,\overline{\psi_{\ell_2}^{[s^{-4/9}]}(D - \sigma_2)u_n})\bigr\rVert_{\ell^1}\\
    +& \sum_{\substack{\ell_1 > \min\{\ell_0+1, \ell_2\}}} s^{-1}\bigl\lVert T_{m^1_{\ell_1,\ell_2}}(Lw_n, (Q^{+}_{\gg - 100} + Q^-_{\gg -100}) u_n,\overline{\psi_{\ell_2}^{[s^{-4/9}]}(D - \sigma_2)u_n})\bigr\rVert_{\ell^1}\\
    \lesssim& M^2\epsilon^3 s^{-\frac{37}{30}}\log s
\end{split}\end{equation*}

\subsubsection{The contribution from the degenerate stationary point}

The argument for~\eqref{eqn:p1-st-pt} is similar to the one for~\eqref{eqn:p2-st-pt}.  Introducing the cut-off functions
    \begin{equation*}
    \chi_{3,\ell_1,\ell_2}(\eta,\sigma) = \chi_2(\eta,\sigma) \tilde{\psi}^{[s^{-1/3}]}_{\ell_1}(\eta - \eta_2) \tilde{\psi}^{[s^{-1/3}]}_{\ell_1}(\sigma - \sigma_2)
\end{equation*}
we divide in frequency, writing
\begin{subequations}\begin{align}
    \eqref{eqn:p1-st-pt} =& \iint \mu(\eta,\sigma) \chi_{3,\ell_0,\ell_0}  e^{is\phi} \hat{g}(\fhPi-\eta-\sigma,s) \hat{f}(\eta,s) \overline{\hat{f}}(-\sigma,s) \;d\eta d\sigma\label{eqn:p1-st-pt-lo}\\
    &+ \sum_{\ell_1 > \max\{\ell_1, \ell_0\}} \iint \mu(\eta,\sigma) \chi_{3,\ell_1,\ell_2}  e^{is\phi} \hat{g}(\fhPi-\eta-\sigma,s) \hat{f}(\eta,s) \overline{\hat{f}}(-\sigma,s) \;d\eta d\sigma\label{eqn:p1-st-pt-hi}\\
    &+\bsim\notag
\end{align}\end{subequations}
where $\ell_0 = \ell_0(s)$ satisfies $2^{\ell_0} < s^{-1/3} \leq 2^{\ell_0+1}$.  Noting that \begin{equation*}
    \hat{g}(\fhPi -\eta+\sigma,s)| \lesssim |\eta + \sigma|^{1/2} \lVert ng_n \rVert_{\ell^2} \lesssim \epsilon s^{\frac{1}{10}}|\eta - \sigma|^{1/2}
\end{equation*}
a simple volume bound gives
\begin{equation*}
    |\eqref{eqn:p1-st-pt-lo}| \lesssim M^2\epsilon^3s^{-\frac{16}{15}}
\end{equation*}
Turning to~\eqref{eqn:p1-st-pt-hi}, an integration by parts yields
\begin{equation*}\begin{split}
    \eqref{eqn:p1-st-pt-hi} =& \sum_{\ell_1 > \max\{\ell_1, \ell_0\}}  \frac{i}{s} \hat{T}_{m^1_{\ell_1\ell_2}}(Q_{\sim \ell_1}^{+} u_n,Q_{\lesssim \ell_1}^{+}w_n,\overline{Q_{\sim \ell_2}^{+}u_n})(\fhPi) \\
    &- \sum_{\ell_1 > \max\{\ell_1, \ell_0\}} \frac{i}{s} \hat{T}_{m^2_{\ell_1\ell_2}}(\invLinProp[s] nf_n,Q_{\lesssim \ell_1}^{+}w_n,\overline{Q_{\sim \ell_2}^{+}u_n})(\fhPi)\\
    &- \sum_{\ell_1 > \max\{\ell_1, \ell_0\}}  \frac{i}{s} \hat{T}_{m^2_{\ell_1\ell_2}}(Q_{\sim \ell_1}^{+} u,Lw_n,\overline{Q_{\sim \ell_2}^{+}u})(\fhPi)\\
\end{split}\end{equation*}
where
\begin{equation*}\begin{split}
    m^1_{1,\ell_1 \ell_2}(\xi,\eta,\sigma) =& \psi^{[s^{-1/3}]}_{\lesssim \ell_1}(\xi) \deles \cdot \frac{m\chi_{1,\ell_1, \ell_2} \deles \phi}{|\deles \phi|^2}\\
    m^2_{1,\ell_1 \ell_2}(\xi,\eta,\sigma) =& \psi^{[s^{-1/3}]}_{\lesssim \ell_1}(\xi)\frac{m\chi_{1,\ell_1, \ell_2} \partial_\eta \phi}{|\deles \phi|^2}
\end{split}\end{equation*}
A quick calculation shows that $2^{2\ell_1} m^1_{1,\ell_1, \ell_2}$ and $2^{\ell_1} m^2_{1,\ell_1, \ell_2}$ satisfy the conditions given in~\Cref{rmk:freq-loc-symbol-bds}, so by the Hausdorff-Young inequality,
\begin{equation*}\begin{split}
    \left|\frac{i}{s}\hat{T}_{m^1_{\ell_1\ell_2}}(Q_{\sim \ell_1}^{+} u_n,Q_{\lesssim \ell_1}^{+}w_n,\overline{Q_{\sim \ell_2}^{-}u_n})(\fhPi) \right| \lesssim& s^{-1} 2^{-2\ell_1} \lVert Q_{\sim \ell_1}^{+} u_n \rVert_{\ell^\infty} \lVert Q_{\lesssim \ell_1}^{+}w_n \rVert_{\ell^2} \lVert Q_{\ell_2}^{-} u_n \rVert_{\ell^2}\\
    \lesssim& M^2 \epsilon^3 s^{-7/5} 2^{-3/2 \ell_1 + \ell_2/2}\\
    \left|\frac{i}{s}\hat{T}_{m^2_{\ell_1\ell_2}}(\invLinProp[s] n f_n,Q_{\lesssim \ell_1}^{+}w_n,\overline{Q_{\sim \ell_2}^{-}u_n})(\fhPi) \right| \lesssim& s^{-1} 2^{-\ell_1} \lVert nf_n \rVert_{\ell^2} \lVert w_n \rVert_{\ell^\infty} \lVert Q_{\ell_2}^{-} u_n \rVert_{\ell^2}\\
    \lesssim& M^2 \epsilon^3 s^{-\frac{37}{30}} 2^{-\ell_1 + \ell_2/2}\\
    \left|\frac{i}{s}\hat{T}_{m^2_{\ell_1\ell_2}}(Q^+_{\sim \ell_1} u_n,Lw_n,\overline{Q_{\sim \ell_2}^{-}u_n})(\fhPi) \right| \lesssim& s^{-1} 2^{-\ell_1} \lVert ng_n \rVert_{\ell^2} \lVert Q^+_{\sim \ell_1} u_n \rVert_{\ell^\infty} \lVert Q_{\ell_2}^{-} u_n \rVert_{\ell^2}\\
    \lesssim& M^2 \epsilon^3 s^{-7/5} 2^{-3/2\ell_1 + \ell_2/2}
\end{split}\end{equation*}
Summing, we find that
\begin{equation*}
    |\eqref{eqn:p1-st-pt-hi}| \lesssim M^2\epsilon^3 s^{-\frac{16}{15}}
\end{equation*}

\section{The stationary phase estimate\label{sec:L-infty}}
We now consider the behavior of $\hat f(t)$.  There are two regions of interest: the region $|\xi \mp \fhPi| \lesssim t^{-1/3}$ where $\xi$ is near the degenerate frequencies $\pm \fhPi$, and the complementary region $|\xi \mp \fhPi| \gtrsim t^{-1/3}$.

\subsection{Estimates for the degenerate frequency regions}
Suppose $\left|\xi-\fhPi\right| \lesssim t^{-1/3}$.  Based on the estimates in~\eqref{sec:degen-freq-decay}, we see that
\begin{equation*}
    |\hat{f}(\fhPi, t)| \leq |\hat{u_*}(\fhPi)| + \int_1^t |\partial_t \alpha_+(s)|\;ds \lesssim \epsilon
\end{equation*}
so it suffices to control $|\hat{f}(\xi,t) - \hat{f}(\fhPi,t)|$.  But, the Sobolev-Morrey embedding gives us the bound
\begin{equation*}
    |\hat{f}(\xi,t) - \hat{f}(\eta,t)| \lesssim \lVert nf_n \rVert_{\ell^2} |\xi - \fhPi|^{1/2} \lesssim \epsilon t^{1/6} |\xi - \fhPi|^{1/2}
\end{equation*}
which shows that $|\hat{f}(\xi,t)|$ is bounded for $|\xi - \fhPi| \lesssim t^{-1/3}$.

\subsection{The perturbed Hamiltonian dynamics}

  Symmetrizing~\eqref{eqn:dt-f-hat} in $\eta$ and $\xi - \eta - \sigma$, we find that
\begin{equation}\label{eqn:f-hat-symmetrized-eqn}
	\partial_t \hat{f}(\xi_0,t) = \frac{i}{4\pi} \iint \mu(\xi_0,\eta,\sigma) e^{it\phi} \hat{f}(\eta,t) \hat{f}(\xi_0 - \eta - \sigma, t) \overline{\hat{f}}(-\sigma, t)\;d\eta d\sigma
\end{equation}
where $\mu(\xi, \eta, \sigma) = \left(\cos \eta + \cos(\xi - \eta - \sigma)\right)$
For fixed $\xi_0 \neq \pm \frac{\pi}{2}$, $\phi(\xi_0, \cdot, \cdot)$ has six stationary points.  Two of these points can be neglected because they occur where the null-form $\mu$ also vanishes.  The remaining four points $P_j = (\eta_j, \sigma_j)$, $j=1,2,3,4$ are given in~\Cref{tab:stationary-point-data}.  Observe that the pairs of points $\{P_1,P_4\}$ and $\{P_1,P_3\}$ are separated by a distance $\bigO{\xi_0 \mp \fhPi}$, while all other pairs of stationary points are uniformly separated in $\xi_0$.  A sketch of the situation is given in~\Cref{fig:st-pt-pic}.

\begin{table}[b]
\begin{tabular}{|c|c|c|c|c|}
    \hline
     $P_j$ &  $(\eta_j,\sigma_j)$ & $\phi_j = \phi(\xi,\eta_j,\sigma_j)$ & $D_j = \det(\Hess \phi)$ & $S_j = \sgn(\Hess \phi)$\\
     \hline\hline
     $P_1$ & $(\xi_0, -\xi_0)$ & $0$ & $-4\cos^2 \xi_0$ & 0\\\hline
     $P_2$ & $\left(\pi +\frac{\xi_0}{3}, \frac{\xi_0}{3}\right)$ & $-8 \cos^3(\xi_0/3)$ & $12 \cos^2 \frac{\xi_0}{3}$ &$-\sgn(\cos \frac{\xi}{3})$\\\hline
     $P_3$ & $\left(\frac{\xi_0 - \pi}{3}, \frac{\xi_0 + 2\pi}{3}\right)$ & $\phantom{-}8 \cos^3\left(\frac{\pi - \xi_0}{3}\right)$ & $12 \cos^2 \left(\frac{\xi_0 -\pi}{3}\right)$ & $\sgn\left(\cos \left(\frac{\xi_0 - \pi}{3}\right)\right)$\\\hline
     $P_4$ & $\left(\frac{\xi_0 + \pi}{3}, \frac{\xi_0 - 2\pi}{3}\right)$ & $-8 \cos^3\left(\frac{\xi_0 + \pi}{3}\right)$ & $12 \cos^2\left(\frac{\xi_0 + \pi}{3}\right)$ &$\sgn\left(\cos \left(\frac{\xi_0 + \pi}{3}\right)\right)$\\\hline
\end{tabular}
\caption[Function values at the stationary points]{\label{tab:stationary-point-data} The four stationary points of $\deles \phi$ which are not canceled by the null-form, together with data about the phase there.}
\end{table} 

\begin{figure}
    \def\piconst{12}
    \def\picxi{7}
    \begin{center}
    \begin{tikzpicture}[scale=0.3]
        \draw[loosely dashed](-0.5 * \piconst,-\piconst)--(-0.5*\piconst,\piconst);
        \draw[loosely dashed](0.5 * \piconst,-\piconst)--(0.5*\piconst,\piconst);
        \draw[loosely dashed](- \piconst,-0.5 *\piconst)--(\piconst,-0.5*\piconst);
        \draw[loosely dashed]( -\piconst,0.5*\piconst)--(\piconst,0.5*\piconst);
        \draw[fill=black] (\picxi,-\picxi) circle (0.1em) node[anchor=north west]{$P_1$};
        \draw[fill=black] (-\piconst + \picxi / 3,\picxi / 3) circle (0.1em) node[anchor=north west]{$P_2$};
        \draw[fill=black] (-\piconst/3 + \picxi / 3,\picxi / 3 + 2* \piconst /3) circle (0.1em) node[anchor=north west]{$P_3$};
        \draw[fill=black] (\piconst/3 + \picxi / 3,\picxi / 3 - 2* \piconst /3) circle (0.1em) node[anchor=south west]{$P_4$};
    \end{tikzpicture}
    \end{center}
    \caption[The stationary points of $\phi$ for fixed $\xi$]{\label{fig:st-pt-pic} The four stationary points for $\xi_0 \approx \fhPi$.  The dotted lines are $\eta,\sigma = \pm \fhPi$.}
\end{figure}
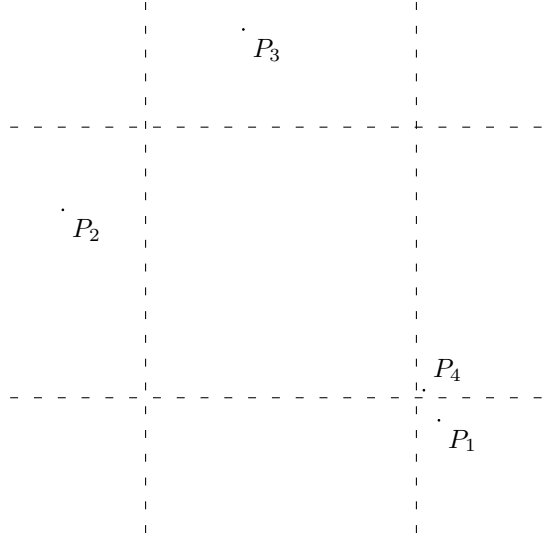

We now claim that
\begin{equation}\label{eqn:d-t-f-hat-st-ph-form}\begin{split}
    \partial_t \hat{f} =& \frac{i\sgn(\cos \xi_0)}{4t} |\hat{f}(\xi_0)|^2 \hat{f}(\xi_0) + \sum_{j=2}^4 \frac{c_j(\xi_0)}{t} e^{it\phi_j} \hat{f}(\eta_j) \hat{f}(\xi_0 - \eta_j - \sigma_j) \overline{\hat{f}(-\sigma_j)}\\
    &\qquad+ O\left(M^3\epsilon^3 t^{-1}(|\cos(\xi_0)| t^{1/3})^{-\stPhErr}\right)
\end{split}\end{equation}
where $c_j(\xi_0) = i \frac{\cos(\eta_j)}{2|D_j|^{1/2}}e^{i\frac{\pi}{4}S_j}$ for $D_j$, $S_j$ given in~\Cref{tab:stationary-point-data}.  We now show that the stationary phase expansion~\eqref{eqn:d-t-f-hat-st-ph-form} is enough to prove that $\hat{f}$ is bounded in time and evolves by logarithmic phase rotation.  Defining
\begin{equation}\label{eqn:B-defn}
    B(t,\xi_0) := -\frac{1}{4} \int_{1}^t \frac{|\hat f(\xi_0,s)|^2}{s}\;ds
\end{equation}
and letting $v = \exp(iB)\hat{f}$, and integrating in time, we find that
\begin{equation*}
    v(t_2) - v(t_1) = \sum_{j=2}^4 \int_{t_1}^{t_2} \frac{c_j(\xi_0)}{s} e^{is\phi_j} e^{iB} \hat{f}(\eta_j) \hat{f}(\xi_0 - \eta_j - \sigma_j) \overline{\hat{f}(-\sigma_j)} \;ds  + O\left(M^3\epsilon^3 (\cos(\xi_0) t_1^{1/2})^{-\stPhErr}\right)
\end{equation*}
Integrating by parts in $s$ and using~\cref{eqn:d-t-f-hat-st-ph-form,eqn:B-defn} to express the derivatives yields
\begin{equation*}\begin{split}
    |v(t_2) - v(t_1)| \lesssim& \sum_{j=2}^4\Biggl. \frac{\bigl| \hat{f}(\eta_j) \hat{f}(\xi_0 - \eta_j - \sigma_j) \overline{\hat{f}(-\sigma_j)} \bigr|}{s|\phi_j|} \Biggr|_{s = t_1}^{s = t_2}+ \int_{t_1}^{t_2} \frac{\left| \hat{f}(\eta_j) \hat{f}(\xi_0 - \eta_j - \sigma_j) \overline{\hat{f}(-\sigma_j)} \right|}{|\phi_j|s^2}  ds\\
    &\qquad+ \int_{t_1}^{t_2} \frac{\bigl|\partial_s B \hat{f}(\eta_j) \hat{f}(\xi_0 - \eta_j - \sigma_j) \overline{\hat{f}(-\sigma_j)} \bigr|}{s |\phi_j|}   \; ds \\
    &+ \int_{t_1}^{t_2} \frac{\bigl| \partial_s \hat{f}(\eta_j) \hat{f}(\xi_0 - \eta_j - \sigma_j) \overline{\hat{f}(-\sigma_j)} \bigr|}{s |\phi_j|} ds +\bsim \\
    &+ O\bigl(M^3\epsilon^3 (\cos(\xi_0) t_1^{1/2})^{-\stPhErr}\bigr)\\
    \lesssim& M^3\epsilon^3 (\cos(\xi_0) t_1^{1/2})^{-\stPhErr}
\end{split}\end{equation*}
It follows that $\lim_{t \to \infty} v(t) =: f_\infty(\xi)$ exists, and that
\begin{equation*}
    \hat{f}(\xi,t) = \exp(-iB) f_{\infty}(\xi) + O\left(M^3\epsilon^3 (\cos(\xi_0) t^{1/3})^{-\stPhErr}\right)
\end{equation*}

\subsection{The stationary phase estimates}

In this section, we will prove the stationary phase estimate~\eqref{eqn:d-t-f-hat-st-ph-form}.  Let $|\cos \xi_0| \sim 2^k$.  We will show how to obtain the estimates for $\sin \xi_0 \geq 0$, $k \leq -20$; that is, for $|\xi_0 - \fhPi| \ll 1$.  The estimates $|\xi_0 + \fhPi| \ll 1$ are similar once we interchange the roles of $P_3$ and $P_4$, and the estimates for $\xi_0$ bounded away from $\pm \fhPi$ are easier because the distances between stationary points are bounded below.  Based on the frequency space geometry sketched in~\Cref{fig:st-pt-pic}, we write
\begin{equation*}
    \partial_t \hat{f}(\xi_0) = \rmI^{(1,4)} + \rmI^{(2)} + \rmI^{(3)} + R
\end{equation*}
where
\begin{equation*}\begin{split}
    \rmI^{(1,4)} =& \frac{i}{4\pi} \iint  \psi_{\leq -10}(\eta - \fhPi) \psi_{\leq -10}(\sigma + \fhPi) e^{it\phi} F(\xi_0,\eta,\sigma)\;d\eta d\sigma\\
    \rmI^{(2)} =& \frac{i}{4\pi} \iint \psi_{\leq -10}(\eta - \eta_2) \psi_{\leq -10}(\sigma -\sigma_2) e^{it\phi} F(\xi_0,\eta,\sigma)\;d\eta d\sigma\\
    \rmI^{(3)} =& \frac{i}{4\pi} \iint \psi_{\leq -10}(\eta - \eta_3) \psi_{\leq -10}(\sigma -\sigma_3) e^{it\phi} F(\xi_0,\eta,\sigma)\;d\eta d\sigma\\
    R =& \partial_t \hat{f}(\xi_0) - \rmI^{(1,4)} - \rmI^{(2)} - \rmI^{(3)}
\end{split}\end{equation*}
where, for brevity, we have introduced the notation
\begin{equation*}
    F(\xi, \eta, \sigma) = \mu(\xi, \eta,\sigma) \hat{f}(\eta) \hat{f}(\xi - \eta + \sigma) \overline{\hat{f}(-\sigma)}
\end{equation*}
Here $I^{(2)}$ and $I^{(3)}$ contain the contribution from near the stationary points $P_2$ and $P_3$ (respectively), $I^{(1,4)}$ contains the contribution from the region near $P_1$ and $P_4$, and $R$ is the remainder.  We claim that
\begin{equation}\label{eqn:st-ph-I14}\begin{split}
    \rmI^{(1,4)} =& \frac{i \sgn(\cos \xi)}{4t}|\hat{f}(\xi_0)|^2\hat{f}(\xi_0) + \frac{c_4(\xi)}{t}e^{it\phi_3}\hat{f}(\eta_4)\hat{f}(\xi_0 - \eta_4 - \sigma_4) \overline{\hat{f}(-\sigma_4)} \\
    &\qquad + \bigO{M^3\epsilon^3 t^{-1} \left(t^{1/3} 2^k\right)^{-\stPhErr}}
\end{split}\end{equation}
\begin{equation}\label{eqn:st-ph-I2}
    \rmI^{(2)} = \frac{c_2(\xi_0)}{t}e^{it\phi_2}\hat{f}(\eta_2)\hat{f}(\xi_0 - \eta_2 - \sigma_2) \overline{\hat{f}(-\sigma)} + \bigO{M^3\epsilon^3 t^{-1} \left(t^{1/3} 2^k\right)^{-\stPhErr}}
\end{equation}
\begin{equation}\label{eqn:st-ph-I3}
    \rmI^{(3)} = \frac{c_3(\xi_0)}{t}e^{it\phi_2}\hat{f}(\eta_2)\hat{f}(\xi_0 - \eta_2 - \sigma_2) \overline{\hat{f}(-\sigma)} + \bigO{M^3\epsilon^3 t^{-1} \left(t^{1/3} 2^k\right)^{-\stPhErr}}
\end{equation}
\begin{equation}\label{eqn:st-ph-R}
    R = \bigO{M^3\epsilon^3 t^{-1} \left(t^{1/3} 2^k\right)^{-\stPhErr}}
\end{equation}
which, taken together, give the estimate~\eqref{eqn:d-t-f-hat-st-ph-form}.  The estimate for $\rmI^{(1,4)}$ is the most challenging, since it contains the two stationary points $P_1$ and $P_4$ which coalesce as $\xi_0 \to \fhPi$.  Thus, we will focus on proving~\eqref{eqn:st-ph-I14}, as \cref{eqn:st-ph-I2,eqn:st-ph-I3,eqn:st-ph-R} are easier.  We introduce the decomposition
\begin{equation*}\begin{split}
    \rmI^{(1,4)} =& \frac{i}{4\pi} \iint \psi_{\ll k}(\eta - \fhPi) \psi_{\ll k}(\sigma + \fhPi) e^{it\phi} F(\xi_0,\eta,\sigma)\;d\eta d\sigma\\
    &+ \frac{i}{4\pi} \iint \biggl(\psi_{\sim k}(\eta - \fhPi) \psi_{\lesssim k}(\sigma + \fhPi) + \psi_{\ll k}(\eta - \fhPi) \psi_{\sim k}(\sigma + \fhPi)\biggr) e^{it\phi} F(\xi_0,\eta,\sigma)d\eta d\sigma\\
    &+ \frac{i}{4\pi} \iint \left(1 - \psi_{\lesssim k}(\eta - \fhPi) \psi_{\lesssim k}(\sigma + \fhPi)\right) e^{it\phi} F(\xi_0,\eta,\sigma)\;d\eta d\sigma\\
    =:& \rmI^{(1,4)}_\text{lo} + \rmI^{(1,4)}_\text{stat} + \rmI^{(1,4)}_\text{hi}
\end{split}\end{equation*}
We begin with the bound for $\rmI^{(1,4)}_{\text{lo}}$.  In this region, $|\deles \phi| \sim 2^{2k}$, so we can integrate by parts in frequency to obtain
\begin{equation*}\begin{split}
    \rmI^{(1,4)}_{\text{lo}} 
    &= t^{-1} \hat{T}_{m^1_\text{lo}}(Q^+_{\lesssim k} u_n,Q^+_{\lesssim k} u_n, \overline{Q^+_{\lesssim k}u}_n)(\xi_0)\\
    &- t^{-1} \hat{T}_{m^2_\text{lo}}(L u_n,Q^+_{\lesssim k} u_n, \overline{Q^+_{\lesssim k}u_n})(\xi_0)\\
    &-\bsim
\end{split}\end{equation*}
where 
\begin{equation*}
    m^1_{\text{lo}} = \frac{1}{4\pi}\psi_{\lesssim k}(\xi - \xi_0)\deles \cdot \left(\frac{\deles \phi}{|\deles \phi|^2}\mu(\xi_0,\eta,\sigma) \psi_{\ll k}(\eta - \fhPi) \psi_{\ll k}(\sigma + \fhPi) \right)
\end{equation*}
\begin{equation*}
    m^2_{\text{lo}} = \frac{1}{4\pi}\psi_{\lesssim k}(\xi - \xi_0)\left(\frac{\partial_\eta \phi}{|\deles \phi|^2}\mu(\xi_0,\eta,\sigma) \psi_{\ll k}(\eta - \fhPi) \psi_{\ll k}(\sigma + \fhPi) \right)
\end{equation*}
Both $m^1_\text{lo}$ and $m^2_{\text{lo}}$ are supported on a region of volume $\bigO{2^{3k}}$, and satisfy the symbol bounds
\begin{equation*}\begin{split}
    |\partial_{\xi,\eta,\sigma}^{\alpha} m^1_{\text{lo}}| \lesssim_\alpha& 2^{-(2 + |\alpha|)k}\\
    |\partial_{\xi,\eta,\sigma}^{\alpha} m^2_{\text{lo}}| \lesssim_\alpha& 2^{-(1 + |\alpha|)k}\\
\end{split}\end{equation*}
Thus, by~\Cref{rmk:freq-loc-symbol-bds}, we find
\begin{equation*}\begin{split}
    t^{-1}  \left|\hat{T}_{m^1_\text{lo}}(Q^+_{\lesssim k} u_n,Q^+_{\lesssim k} u_n, \overline{Q^+_{\lesssim k}u_n})(\xi_0) \right| 
    \lesssim& t^{-1} 2^{-2k} \lVert Q^+_{\ll k} u_n \rVert_{\ell^2}^2 \lVert u_n \rVert_{\ell^\infty}\\
    \lesssim& M^3\epsilon^3 t^{-4/3} 2^{-k}\\
t^{-1} \left|\hat{T}_{m^2_\text{lo}}(L u_n,Q^+_{\lesssim k} u_n, \overline{Q^+_{\lesssim k}u_n})(\xi_0) \right| 
    \lesssim& t^{-1} 2^{-k} \lVert Q^+_{\ll k} u_n \rVert_{\ell^2}\lVert nf_n \rVert_{\ell^2} \lVert u_n \rVert_{\ell^\infty}\\
    \lesssim& M^3\epsilon^3 t^{-7/6} 2^{-k/2}
\end{split}	\end{equation*}
so $\rmI^{(1,4)}_{\text{lo}} = \bigO{M^3\epsilon^3 t^{-1} (t^{1/3} 2^k)^{-1/2}}$, which is better than required.  Turning to $\rmI^{(1,4)}_{\text{hi}}$, we write
\begin{equation*}\begin{split}
    \rmI^{(1,4)}_{\text{hi}} =& \sum_{\ell > k + C} \frac{i}{4\pi} \iint  \psi_{\ell}(\eta - \fhPi) \psi_{\leq \ell}(\sigma + \fhPi) e^{it\phi} F(\xi_0,\eta,\sigma)\;d\eta d\sigma\\
    &+ \sum_{\ell > k + C} \frac{i}{4\pi} \iint \psi_{< \ell}(\eta - \fhPi) \psi_{\ell}(\sigma + \fhPi) e^{it\phi} F(\xi_0,\eta,\sigma)\;d\eta d\sigma\\
    =:& \sum_{\ell > k + C} J_{\ell} + \tilde{J}_{\ell}
\end{split}\end{equation*}
Over the support of $J_{\ell}$, we find that $|\deles \phi| \sim 2^{2\ell}$, so we can integrate by parts to obtain
\begin{equation*}\begin{split}
    J_{\ell} =& t^{-1} \hat{T}_{m^1_{\text{hi},\ell}}(Q^+_{\lesssim \ell} u,Q^+_{\lesssim \ell} u, \overline{Q^+_{\lesssim \ell} u})(\xi_0)\\
    &+ \hat{T}_{m^2_{\text{hi},\ell}}(\linProp (n f_n),Q^+_{\lesssim \ell} u, \overline{Q^+_{\lesssim \ell} u})(\xi_0)\\
    &+ \bsim
\end{split}\end{equation*}
where 
\begin{equation*}
    m^1_{\text{hi},\ell} = -\frac{1}{4\pi}\psi_{\lesssim \ell}(\xi - \xi_0)\deles \cdot \left(\frac{\deles \phi}{|\deles \phi|^2}\mu(\xi_0,\eta,\sigma) \psi_{\ll k}(\eta - \fhPi) \psi_{\ll k}(\sigma + \fhPi) \right)
\end{equation*}
\begin{equation*}
    m^2_{\text{hi},\ell} = -\frac{1}{4\pi}\psi_{\lesssim \ell}(\xi - \xi_0)\left(\frac{\partial_\eta \phi}{|\deles \phi|^2}\mu(\xi_0,\eta,\sigma) \psi_{\ll k}(\eta - \fhPi) \psi_{\ll k}(\sigma + \fhPi) \right)
\end{equation*}
Arguing as above, we find that
\begin{equation*}\begin{split}
    t^{-1}\left| \hat{T}_{m^1_{\text{hi},\ell}}(Q^+_{\lesssim \ell} u,Q^+_{\lesssim \ell} u, \overline{Q^+_{\lesssim \ell} u})(\xi_0)\right| 
    \lesssim& M^3\epsilon^3 t^{-4/3} 2^{-\ell}\\
    t^{-1}\left| \hat{T}_{m^1_{\text{hi},\ell}}(\linProp nf_n,Q^+_{\lesssim \ell} u, \overline{Q^+_{\lesssim \ell} u})(\xi_0)\right| 
    \lesssim& M^3\epsilon^3 t^{-7/6} 2^{-\ell/2}
\end{split}\end{equation*}
Similar estimates hold for $\tilde{J}_\ell$.  Thus, summing over $\ell > k + C$, we find that
\begin{equation*}
    \rmI^{(1,4)}_{\text{hi}} = \bigO{M^3\epsilon^3 t^{-1} (t^{1/3} 2^k)^{-1/2}}
\end{equation*}
which is better than required.

It only remains to estimate $\rmI^{(1,4)}_{\text{stat}}$.  Here, it is important to consider the contributions from the two stationary points.  To do so, we write
\begin{equation}
    \rmI^{(1,4)}_\text{stat} = J^{(1)} + J^{(4)} + J^{\text{rem}}
\end{equation}
where
\begin{equation*}
    J^{(j)} = \frac{i}{4\pi} \iint \psi_{\ll k}(\eta - \eta_j) \psi_{\ll k}(\sigma - \sigma_j) e^{it\phi} F(\xi_0,\eta,\sigma)\;d\eta d\sigma
\end{equation*}
for $j = 1,4$, and $J^{\text{rem}} = \rmI^{(1,4)}_{\text{stat}} - J^{(1)} - J^{(4)}$ contains the contribution from away from the stationary points.  We focus on the contributions from the $J^{(j)}$: the remainder term is simpler.  We begin by dividing dyadically in frequency, writing
\begin{equation*}\begin{split}
    J^{(j)} =& \sum_{\ell \ll k} \frac{i}{4\pi} \iint \psi^{[\ell_0]}_{\ell}(\eta - \eta_j) \psi^{[\ell_0]}_{\leq \ell}(\sigma - \sigma_j) e^{it\phi} F(\xi_0, \eta, \sigma)\;d\eta d\sigma \\
    &+ \sum_{\ell \ll k} \frac{i}{4\pi} \iint \psi^{[\ell_0]}_{< \ell}(\eta - \eta_j) \psi^{[\ell_0]}_{\ell}(\sigma - \sigma_j) e^{it\phi} F(\xi_0, \eta, \sigma)\;d\eta d\sigma \\
    =:& \sum_{2^{\ell} \ll 2^k} J^{(j)}_{\ell} + \tilde{J}^{(j)}_{\ell}
\end{split}\end{equation*}
where
\begin{equation}\label{eqn:ell-0-defn}
    2^{\ell_0} \sim t^{-1/3}(t^{1/3}2^k)^{3/7}
\end{equation}
We first bound the sum for $\ell > \ell_0$.  In this case, $|\deles \phi| \sim 2^{k + \ell}$ over the integrand, and we can integrate by parts to obtain
\begin{equation*}\begin{split}
    J^{(j)}_{\ell} 
    =& t^{-1}\hat{T}_{m^{j,1}_{\text{stat},\ell}}(P^{\eta_j}_{\ell_1}u_n, \tilde{P}^{\xi_0 - \eta_j + \sigma_j}_{\ell}u_n, \overline{Q^+_{\sim k} u_n})(\xi_0)\\
    &+ t^{-1}\hat{T}_{m^{j,2}_{\text{stat},\ell}}(L u_n, \tilde{P}^{\xi_0 - \eta_j + \sigma_j}_{\ell}u_n, \overline{Q^{+}_{\sim k}u_n})(\xi_0)\\
    &+\bsim\\
\end{split}	\end{equation*}
where $\tilde{P}^\zeta_{\ell}$ is the Fourier multiplier with symbol $\widehat{\tilde{P}^\zeta_{\ell}}(\xi) = \psi_{\lesssim \ell}(\xi - \zeta)$ and $m^{j,1}_{\text{stat},\ell}$ and $m^{j,2}_{\text{stat},\ell}$ are symbols satisfying
\begin{equation*}\begin{split}
    |\partial_{\xi,\eta,\sigma} m^{j,1}_{\text{stat},\ell}| \lesssim 2^{-(|\alpha| - 2)\ell},&\qquad\qquad \Vol(\supp m^{j,1}_{\text{stat},\ell}) \sim 2^{3\ell}\\
    |\partial_{\xi,\eta,\sigma} m^{j,2}_{\text{stat},\ell}| \lesssim 2^{-(|\alpha| - 1)\ell},&\qquad\qquad \Vol(\supp m^{j,2}_{\text{stat},\ell}) \sim 2^{3\ell}
\end{split}\end{equation*}
A quick calculation similar to the ones performed above yields that
\begin{equation*}\begin{split}
    |J^{(j)}_\ell| \lesssim& M^3\epsilon^3 t^{-3/2}2^{-k/2-\ell} +  M^3\epsilon^3 t^{-4/3} 2^{-k/2} 2^{-\ell/2}
\end{split}\end{equation*}
The $\tilde{J}^{(j)}_\ell$ obey similar bounds, so we can sum to obtain
\begin{equation*}
    \sum_{\ell > \ell_0} |J^{(j)}_\ell| + |\tilde{J}^{(j)}_\ell| \lesssim M^3\epsilon^3 t^{-1} (t^{1/3}2^k)^{-\frac{1}{52}}
\end{equation*}
It only remains to consider the term with $\ell = \ell_0$.  By definition, $\tilde{J}^{(j)}_{\ell_0} = 0$, so we consider $J^{(j)}_{\ell_0}$.  We will estimate this term using the method of stationary phase.  Using the $C^{1/2}$ bounds on $\hat{f}$ and the fact that ${|\mu(\xi_0, \eta_j, \sigma_j)| \sim 2^k}$ for $j = 1,4$,	we find that
\begin{equation*}
    |F(\xi_0,\eta,\sigma) - F(\xi_0, \eta_j, \sigma_j)| \lesssim M^3 \epsilon^3 \left((|\eta - \eta_j| + |\sigma - \sigma_j|) +  t^{1/6}2^k (|\eta - \eta_j| + |\sigma - \sigma_j|)^{1/2}\right)
\end{equation*}
so
\begin{equation*}
    J^{(j)}_{\ell_0} = \frac{i}{4\pi}F(\xi,\eta_j, \sigma_j) \iint \psi^{[\ell_0]}_{\ell}(\eta - \eta_j) \psi^{[\ell_0]}_{\leq \ell}(\sigma - \sigma_j) e^{it\phi} \;d\eta d\sigma + \bigO{M^3\epsilon^3\left(2^{3\ell_0} + t^{1/6}2^k 2^{5/2\ell_0}\right)}
\end{equation*}
Now, since $P_j = (\eta_j, \sigma_j)$ is a stationary point of $\phi$, we can write
\begin{equation*}
    \phi(\xi_0,\eta,\sigma) = \phi_j + Q_j (\eta - \eta_j, \sigma - \sigma_j) + \bigO{|\eta - \eta_j|^3 + |\sigma - \sigma_j|^3}
\end{equation*}
where $Q_j$ is the quadratic form associated to $\Hess \phi(\xi_0, \eta_j, \sigma_j)$ and $\phi_j = \phi(\xi_0,\eta_j,\sigma_j)$ (see~\Cref{tab:stationary-point-data}).  Using the bound $|e^{ix} - 1| \leq |x|$, we can 
\begin{equation*}\begin{split}
    J^{(j)}_{\ell_0} =& \frac{i}{4\pi} F(\xi,\eta_j, \sigma_j) e^{it\phi_j} \iint \psi_{\leq \ell_0}(\eta - \eta_j) \psi_{\leq \ell_0}(\sigma - \sigma_j) e^{itQ_j(\eta - \eta_j, \sigma - \sigma_j)} \;d\eta d\sigma\\
    &+ \bigO{M^3\epsilon^3 \left(2^{3\ell_0} + t^{1/6}2^{5/2\ell_0}2^k + t 2^{5\ell_0}2^k\right)}
\end{split}\end{equation*}
Finally, setting $\zeta = 2^{\ell_0}(\eta - \eta_j)$, $\omega = 2^{\ell_0}(\sigma - \sigma_j)$ and applying the stationary phase lemma, we find that
\begin{equation*}\begin{split}
    \iint \psi_{\leq \ell_0}(\eta - \eta_j) \psi_{\leq \ell_0}(\sigma - \sigma_j) e^{itQ_j(\eta - \eta_j, \sigma - \sigma_j)} \;d\eta d\sigma =& 2^{-2\ell_0}\iint \psi_{0}(\zeta) \psi_{0}(\omega) e^{it2^{-2\ell_0}Q_j(\zeta, \omega)} \;d\zeta d\omega\\
    =& \frac{2\pi}{|D_j|^{1/2} t}e^{i\frac{\pi}{4} S_j} + \bigO{t^{-2}2^{-2\ell_0-2k}}
\end{split}\end{equation*}
where $D_j$ is the determinant of the Hessian at $P_j$ (see~\Cref{tab:stationary-point-data}).  Inserting this into the expression for $J^{(j)}_{\ell_0}$ and recalling the definition of $\ell_0$, we find that
\begin{equation}\label{eqn:I-bound-st-phase}
    J^{(j)}_{\ell_0} = \frac{i\cos(\eta_j)}{2|D_j|^{1/2}t}e^{i\frac{\pi}{2} S_j} \hat{f}(\eta_j) \hat{f}(\xi - \eta_j - \sigma_j) \overline{\hat{f}}(-\sigma_j) + \bigO{M^3\epsilon^3 t^{-1}(2^k t^{1/3})^{-\stPhErr}}
\end{equation}
as required.

	\bibliography{sources}{} 
	\bibliographystyle{plain}

\end{document}